\newif\ifpreprint
\preprinttrue  % Preprint publication. Comment for Journal version
\ifpreprint

\documentclass[12pt]{article}
\usepackage{fullpage}
\title{Data-Driven Performance Guarantees \\for Classical and Learned Optimizers}

\author{Rajiv Sambharya and Bartolomeo Stellato}
\date{%
    Department of Operations Research and Financial Engineering\\
    Princeton University\\[1em]%
    \today
}

\usepackage[round,authoryear]{natbib}
\usepackage{amsmath,enumitem,mathtools,amsthm,amssymb}

\renewcommand\arraystretch{1.2}

\newtheorem{theorem}{Theorem}

\newtheorem{lemma}[theorem]{Lemma}

\newcommand{\acks}[1]{\section*{Acknowledgments}#1}
\newtheorem{definition}{Definition}[section]

\usepackage[colorlinks,
            linkcolor=cyan,
            citecolor=cyan,
            urlcolor=cyan,
            bookmarks]{hyperref}

\else

\documentclass[twoside,11pt]{article}
\usepackage{blindtext}
\usepackage{jmlr2e}
\usepackage{lastpage}
\usepackage{amsmath}

\ShortHeadings{Data-Driven Performance Guarantees for Classical and Learned Optimizers}{Sambharya and Stellato}
\firstpageno{1}

\renewcommand\arraystretch{1.2}

\title{Data-Driven Performance Guarantees \\for Classical and Learned Optimizers}

\author{%
 \name Rajiv Sambharya \email sambhar9@seas.upenn.edu\\
 \addr Electrical and Systems Engineering, University of Pennsylvania, Philadelphia, PA, USA\\
 \name Bartolomeo Stellato \email bstellato@princeton.edu\\
 \addr Operations Research and Financial Engineering, Princeton University, Princeton, NJ, USA
}

\editor{Silvia Villa}
\jmlrheading{26}{2025}{1-\pageref{LastPage}}{5/24; Revised 2/25}{8/25}{24-0755}{Rajiv Sambharya and Bartolomeo Stellato}

\usepackage{caption}
\usepackage{graphicx}
\usepackage{subcaption}
\usepackage{hyperref}
\usepackage{mathtools}
\fi
\usepackage{verbatim}
\usepackage{multicol}
\usepackage{algorithm}% http://ctan.org/pkg/algorithms
\usepackage[noend]{algpseudocode} % Avoid "end" and make it look cleaner

\usepackage{subcaption}
\usepackage{makecell}
\usepackage{todonotes}
\usepackage{csvsimple}	% reading CSV files in tables
\usepackage{booktabs}   % Nicer Tables
\usepackage[flushleft]{threeparttable}  % Table notes
\usepackage[round-mode=places, round-integer-to-decimal, round-precision=4,
    table-format = 1.4,
    table-number-alignment=center,
    round-integer-to-decimal]{siunitx}
\usepackage{adjustbox}  % To adjust table length

\newcommand{\bigline}{\;\bigg|\;}
\newcommand*{\startlegend}{-0.2}
\newcommand*{\enlegend}{0.7}
\newcommand{\risk}{r_\mathcal{X}}
\newcommand{\emprisk}{\hat{r}_S}
\newcommand{\exprisk}{R_\mathcal{X}}
\newcommand{\expemprisk}{\hat{R}_S}
\newcommand{\surrogate}{\ell^{\rm logistic}}
\newcommand{\surrogateexpemprisk}{\hat{R}^{\rm logistic}_S}
\DeclareDocumentCommand{\T}{ O{z} O{} }{T\IfValueT{#2}{(#1,\theta_{#2})}\IfNoValueT{#2}{(#1,\theta)}}
\DeclareDocumentCommand{\Tj}{ O{k} O{z} O{} }{T^{#1}\IfValueT{#3}_{\theta_{#3}}{(#2)}}

\DeclareDocumentCommand{\CB}{ O{} }{C_{B_{\IfValueTF{#1}{\theta_{#1}}{\theta}}}}
\DeclareDocumentCommand{\RB}{ O{} }{R_{B_{\IfValueTF{#1}{\theta_{#1}}{\theta}}}}

\newcommand{\eg}{{\it e.g.}}
\newcommand{\ie}{{\it i.e.}}

\newcommand{\ones}{\mathbf 1}
\newcommand{\reals}{{\mbox{\bf R}}}

\newcommand{\symm}{{\mbox{\bf S}}}  % symmetric matrices
\newcommand{\Tr}{\mathop{\bf tr}}
\newcommand{\diag}{\mathop{\bf diag}}

\newcommand{\dist}{{\bf dist{}}}
\newcommand{\fix}{\mathop{\bf fix{}}}

\newcommand{\dom}{\mathop{\bf dom}} % domain

\newcommand{\round}{\mathop{\bf round}}

\newcommand{\Sec}{Section}
\newcommand{\Thm}{Theorem}

\newcommand{\Eqn}{Equation}

\newcommand{\Eqnshort}{Eq.}
\newcommand{\reviewChanges}[1]{{#1}}

\newcommand{\ccircle}[3]{%
  \raisebox{0.55\height}{
  \begin{tikzpicture}[baseline=(node.base)]%
    \node[circle, fill={rgb,255:red,#1;green,#2;blue,#3}, yshift=14] (node) at (0,1) {};%
  \end{tikzpicture}%
  }
}

\newcommand{\bigccircle}[3]{%
  \raisebox{0.25\height}{
  \begin{tikzpicture}[baseline=(node.base)]%
    \node[circle, fill={rgb,255:red,#1;green,#2;blue,#3}, yshift=14,  minimum size=0.65cm] (node) at (0,1) {};%
  \end{tikzpicture}%
  }
}

\newcommand{\linestraight}[3]{%
  \raisebox{0.5\height}{
  \begin{tikzpicture}%[baseline=(node.base)]%
    \draw[line width=2.5pt, color={rgb,255:red,#1;green,#2;blue,#3}] (0,-.1) -- (0.75,-.1);
  \end{tikzpicture}%
  }
}

\newcommand{\lineblock}[3]{%
  \raisebox{-0.2\height}{
  \begin{tikzpicture}%[baseline=(node.base)]%
    \draw[line width=8.0pt, color={rgb,255:red,#1;green,#2;blue,#3}] (0,0) -- (0.75,0);
  \end{tikzpicture}%
  }
}

\newcommand{\linediamondhollow}[3]{%
    \begin{tikzpicture}[baseline={(0,-.1)}]
      \coordinate (A) at (0.15,0);
      \coordinate (B) at (.25,-.1);
      \coordinate (C) at (.35,0);
      \coordinate (D) at (.25,.1);
      \draw[line width=1.0pt, draw={rgb,255:red,#1;green,#2;blue,#3}] (A) -- (B) -- (C) -- (D) -- cycle;
      \draw[line width=1.0pt, dotted, color={rgb,255:red,#1;green,#2;blue,#3}] (\startlegend,0) -- (0.7,0);
    \end{tikzpicture}%
    }

\newcommand{\diagonalcross}[3]{%
  \begin{tikzpicture}[baseline={(0,-.1)}]
    \filldraw[fill={rgb,255:red,#1;green,#2;blue,#3}, draw=none, rotate=45] (-0.1,-0.03) rectangle (0.1,0.03);
    \filldraw[fill={rgb,255:red,#1;green,#2;blue,#3}, draw=none, rotate=-45] (-0.1,-0.03) rectangle (0.1,0.03);
    \draw[line width=0.8pt, color={rgb,255:red,#1;green,#2;blue,#3}] (-.45,0) -- (0.45,0);
  \end{tikzpicture}%
}

\usetikzlibrary{shapes.geometric}

\newcommand{\drawstar}[3]{%
  \begin{tikzpicture}[baseline={(0,-.1)}]
    \node[star, star points=5, star point ratio=2.25, fill={rgb,255:red,#1;green,#2;blue,#3}, inner sep=0pt, outer sep=0pt, minimum size=0.3cm, draw=none] at (0,0) {};
    \draw[line width=0.8pt, color={rgb,255:red,#1;green,#2;blue,#3}] (-0.45,0) -- (0.45,0); % Adjust the length and color as needed
  \end{tikzpicture}%
}

\newcommand{\linecircle}[3]{%
  \begin{tikzpicture}[baseline={(0,-.1)}]
    \draw[draw=none, fill={rgb,255:red,#1;green,#2;blue,#3}](0.25, 0) circle (.08);
    \draw[line width=0.8pt, color={rgb,255:red,#1;green,#2;blue,#3}] (\startlegend,0) -- (\enlegend,0);
  \end{tikzpicture}%
  }

\newcommand{\linecirclehollow}[3]{%
  \begin{tikzpicture}[baseline={(0,-.1)}]
    \draw[line width=0.8pt, draw={rgb,255:red,#1;green,#2;blue,#3}](0.25, 0) circle (.08);
    \draw[line width=0.8pt, dotted, color={rgb,255:red,#1;green,#2;blue,#3}] (\startlegend,0) -- (\enlegend,0);
  \end{tikzpicture}%
  }

  \newcommand{\lineplushollow}[3]{%
  \begin{tikzpicture}[baseline={(0,-.1)}]
    \def\thickness{0.04}
    \def\length{0.11}
    \def\middle{0.25}
    \coordinate (A) at (-\length+\middle, \thickness);
    \coordinate (B) at (-\thickness+\middle, \thickness);
    \coordinate (C) at (-\thickness+\middle, \length);
    \coordinate (D) at (\thickness+\middle, \length);
    \coordinate (E) at (\thickness+\middle, \thickness);
    \coordinate (F) at (\length+\middle, \thickness);
    \coordinate (G) at (\length+\middle, -\thickness);
    \coordinate (H) at (\thickness+\middle, -\thickness);
    \coordinate (I) at (\thickness+\middle, -\length);
    \coordinate (J) at (-\thickness+\middle, -\length);
    \coordinate (K) at (-\thickness+\middle, -\thickness);
    \coordinate (L) at (-\length+\middle, -\thickness);
    
    \draw[line width=0.8pt, color={rgb,255:red,#1;green,#2;blue,#3}] (A) -- (B) -- (C) -- (D) 
    -- (E) -- (F) -- (G) -- (H)-- (I) -- (J) -- (K) -- (L) -- cycle;
    \draw[line width=0.8pt, dotted, color={rgb,255:red,#1;green,#2;blue,#3}] (\startlegend,0) -- (\enlegend,0);
  \end{tikzpicture}%
}

\newcommand{\linesquare}[3]{%
  \begin{tikzpicture}[baseline={(0,-.1)}]
    \draw[draw=none, fill={rgb,255:red,#1;green,#2;blue,#3}](0.179, -.071) rectangle (.321,.071);
    \draw[line width=0.8pt, color={rgb,255:red,#1;green,#2;blue,#3}] (\startlegend,0) -- (\enlegend,0);
  \end{tikzpicture}%
  }

\newcommand{\linelefttri}[3]{%
  \begin{tikzpicture}[baseline={(0,-.1)}]
    \coordinate (A) at (0.17,0);
    \coordinate (B) at (.3,-.1);
    \coordinate (C) at (.3,.1);
    \draw[draw=none, fill={rgb,255:red,#1;green,#2;blue,#3}] (A) -- (B) -- (C) -- cycle;
    \draw[line width=0.8pt, color={rgb,255:red,#1;green,#2;blue,#3}] (\startlegend,0) -- (\enlegend,0);
  \end{tikzpicture}%
}

\newcommand{\linerighttri}[3]{
  \begin{tikzpicture}[baseline={(0,-.1)}]
    \coordinate (A) at (0.33,0);
    \coordinate (B) at (.2,-.1);
    \coordinate (C) at (.2,.1);
    \draw[draw=none, fill={rgb,255:red,#1;green,#2;blue,#3}] (A) -- (B) -- (C) -- cycle;
    \draw[line width=0.8pt, color={rgb,255:red,#1;green,#2;blue,#3}] (\startlegend,0) -- (\enlegend,0);
  \end{tikzpicture}%
}

\newcommand{\linerighttrihollow}[3]{
  \begin{tikzpicture}[baseline={(0,-.1)}]
    \coordinate (A) at (0.33,0);
    \coordinate (B) at (.2,-.1);
    \coordinate (C) at (.2,.1);
    \draw[line width=0.8pt, draw={rgb,255:red,#1;green,#2;blue,#3}] (A) -- (B) -- (C) -- cycle;
    \draw[line width=0.8pt, dotted, color={rgb,255:red,#1;green,#2;blue,#3}] (\startlegend,0) -- (\enlegend,0);
  \end{tikzpicture}%
}

\newcommand{\linedowntri}[3]{
  \begin{tikzpicture}[baseline={(0,-.1)}]
    \coordinate (A) at (0.25,-.09);
    \coordinate (B) at (.17,.06);
    \coordinate (C) at (.33,.06);
    \draw[draw=none, fill={rgb,255:red,#1;green,#2;blue,#3}] (A) -- (B) -- (C) -- cycle;
    \draw[line width=0.8pt, color={rgb,255:red,#1;green,#2;blue,#3}] (\startlegend,0) -- (\enlegend,0);
  \end{tikzpicture}%
}

\newcommand{\lineuptri}[3]{
  \begin{tikzpicture}[baseline={(0,-.1)}]
    \coordinate (A) at (0.25,.09);
    \coordinate (B) at (.17,-.06);
    \coordinate (C) at (.33,-.06);
    \draw[draw=none, fill={rgb,255:red,#1;green,#2;blue,#3}] (A) -- (B) -- (C) -- cycle;
    \draw[line width=0.8pt, color={rgb,255:red,#1;green,#2;blue,#3}] (\startlegend,0) -- (\enlegend,0);
  \end{tikzpicture}%
}

\newcommand{\uptri}[3]{
  \begin{tikzpicture}[scale=1.7, baseline={(0,-.1)}]
    \coordinate (A) at (0.25,.09);
    \coordinate (B) at (.17,-.06);
    \coordinate (C) at (.33,-.06);
    \draw[draw=none, fill={rgb,255:red,#1;green,#2;blue,#3}] (A) -- (B) -- (C) -- cycle;
  \end{tikzpicture}%
}

\newcommand{\lwslegend}{\vspace{-1mm} \small \\
\linedowntri{0}{0}{0} Cold start \hspace{1mm}
\lineuptri{153}{102}{153} Nearest neighbor \hspace{1mm}\\
\linecirclehollow{228}{26}{28} Empirical L2WS \hspace{1mm}
\linesquare{228}{26}{28} PAC-Bayes L2WS bound \hspace{1mm}\\
}

\newcommand{\classicallegend}{\vspace{-1mm} \small \\
\linedowntri{0}{0}{0} Empirical success rate \hspace{1mm} \linelefttri{166}{86}{40} Worst-case guarantee\\
\linecircle{239}{118}{119}  Bound with $10$ samples \hspace{1mm}
\linesquare{135}{178}{212} Bound with $100$ samples \hspace{1mm}
\linerighttri{148}{207}{146} Bound with $1000$ samples \hspace{1mm}
}

\newcommand{\classicallegendquantile}{\vspace{-1mm} \small \\
\linedowntri{0}{0}{0} Empirical quantile \hspace{1mm} \linelefttri{166}{86}{40} Worst-case guarantee\\
\linecircle{239}{118}{119}  Bound with $10$ samples \hspace{1mm}
\linesquare{135}{178}{212} Bound with $100$ samples \hspace{1mm}
\linerighttri{148}{207}{146} Bound with $1000$ samples \hspace{1mm}
}

\newcommand{\sparsecodinglegend}{\vspace{-1mm} \small \\
\linedowntri{0}{0}{0} Empirical ISTA \hspace{1mm} \\
\lineplushollow{153}{102}{153} Empirical LISTA \hspace{1mm} 
\drawstar{153}{102}{153} PAC-Bayes LISTA bound \hspace{1mm}\\
\linerighttrihollow{55}{126}{184} Empirical TiLISTA \hspace{1mm} \lineuptri{55}{126}{184} PAC-Bayes TiLISTA bound \hspace{1mm}\\
\linecirclehollow{228}{26}{28} Empirical ALISTA \hspace{1mm} \linesquare{228}{26}{28} PAC-Bayes ALISTA bound \hspace{1mm}\\
\linediamondhollow{77}{175}{74} Empirical GLISTA \hspace{1mm} \diagonalcross{77}{175}{74} PAC-Bayes GLISTA bound \hspace{1mm}
}

\newcommand{\mamllegend}{\vspace{-1mm} \small \\
\linedowntri{0}{0}{0} Pretrained \hspace{1mm}
\linecirclehollow{228}{26}{28} Empirical MAML \hspace{1mm} \linesquare{228}{26}{28} PAC-Bayes MAML bound \hspace{1mm}
}

\newcommand{\mamlvisualslegend}{\vspace{-1mm} \small \\
\linestraight{228}{26}{28} Ground truth
\linestraight{55}{126}{184} Pretrained
\uptri{152}{78}{163} Used for gradients \\
\linestraight{77}{175}{74} Empirical MAML
\lineblock{255}{215}{170} Region with PAC-Bayes guarantee 
}

\newcommand{\rkfvisualslegend}{\vspace{-1mm} \small \\
\ccircle{253}{186}{187} Noisy measurements \quad
\ccircle{77}{175}{74} Optimal solution \\
\ccircle{0}{0}{255} SCS solution after $30$ steps \quad
\bigccircle{255}{231}{179} Region with PAC-Bayes guarantee
}

\newenvironment{talign*}
 {\let\displaystyle\textstyle\csname align*\endcsname}
 {\endalign}

\newcommand*{\KL}{{\rm KL}}
\newcommand*{\kl}{{\rm kl}}
\newcommand*{\KLinv}{{\rm kl}^{-1}}
\newcommand*{\param}{x}

\begin{document}
\maketitle
\begin{abstract}%
  We introduce a data-driven approach to analyze the performance of continuous optimization algorithms using generalization guarantees from statistical learning theory.
  We study classical and learned optimizers to solve families of parametric optimization problems.
  We build generalization guarantees for classical optimizers, using a sample convergence bound, and for learned optimizers, using the Probably Approximately Correct (PAC)-Bayes framework. 
  To train learned optimizers, we use a gradient-based algorithm to directly minimize the PAC-Bayes upper bound.  
  Numerical experiments in signal processing, control, and meta-learning showcase the ability of our framework to provide strong generalization guarantees for both classical and learned optimizers given a fixed budget of iterations.
  For classical optimizers, our bounds which hold with high probability are much tighter than those that worst-case guarantees provide.
  For learned optimizers, our bounds outperform the empirical outcomes observed in their non-learned counterparts.
  % For learned optimizers, our bounds outperform their non-learned variants.
  % For learned optimizers, our bounds ensure that they outperform their non-learned variants.
\end{abstract}

% \ifpreprint \else
% \begin{keywords}%
%   learning to optimize, non-convex optimization, prox-linear, generalization bounds.
% \end{keywords}
% \fi
\section{Introduction}
This paper studies continuous parametric optimization problems of the form 
\begin{equation}\label{prob:parametric_opt}
  \begin{array}{ll}
\mbox{minimize} & f(z, \param),\\
\end{array}
\end{equation}
where $z \in \reals^n$ is the decision variable, $\param \in \reals^d$ is the parameter or context drawn from some distribution $\mathcal{X}$, and $f : \reals^n \times \reals^d \rightarrow \reals \cup \{+\infty\}$ is the objective.
Problem~\eqref{prob:parametric_opt} implicitly defines a (potentially non-unique) solution $z^\star(x) \in \reals^n$.
Many applications require repeatedly solving problem~\eqref{prob:parametric_opt} with varying $\param$. 
For instance, in robotics and control, we repeatedly solve optimization problems to update the inputs (\eg, torques and thrusts) while the state (\eg, position and velocity) and the goals (\eg, reference trajectory) change~\citep{borrelli_mpc_book}.
This problem structure is also observed in other domains, such as sparse coding, where sparse signals are recovered from noisy measurements~\citep{lista}, and image restoration, where images are recovered from their corrupted versions~\citep{learned_dictionaries_images}.
% Other examples that conform to the structure of problem~\eqref{prob:parametric_opt} include sparse coding in which we repeatedly recover sparse signals from varying noisy measurements and in image restoration in which we repeatedly recover  images from corrupted ones.
% Parametric optimization is crucial in other domains as well.
% For example, parametric optimization is used in sparse coding to repeatedly recover sparse signals from different noisy measurements and in image restoration to recover 
% In sparse coding we repeatedly recover signals from noisy measurements~\citep{lista} and in image restoration we repeatedly recover similar images from corrupted ones~\citep{learned_dictionaries_images} \red{reword this}.
% These optimization problems often lack closed-form solutions and are usually a bottleneck in the systems they are a part of.
% Solving optimization problems of the form~\eqref{prob:parametric_opt} is often a bottleneck for the systems they are apart of.
These optimization problems usually do not admit closed-form solutions, so instead, iterative algorithms are needed to search for an optimal solution.
First-order methods, which only rely on first-order derivatives~\citep{fom_book}, are a popular approach to solve problem~\eqref{prob:parametric_opt} due to their cheap per-iteration cost.
Typically, first-order methods repeatedly apply a mapping $T: \reals^n \times \reals^d \rightarrow \reals^n$, obtaining iterations of the form
\begin{equation}\label{eq:fp}
  z^{k+1}(x) = T(z^k(x), x).
\end{equation}
Due to the limited time available to compute the solutions between instances of problem~\eqref{prob:parametric_opt}, in several applications we can only afford a fixed number of iterations of algorithm~\eqref{eq:fp}.
% Due to limited time available to compute the solutions in real-time applications\blue{B: cite real-time appls}, we can only afford a XXX CONTINUE FROM HERE...
% In many cases, such as in real-time applications, strict (or desired) latency requirements dictate that the algorithm, as defined by steps~\eqref{eq:fp}, is constrained to a limited number of iterations.
% In many cases, \eg, in real-time applications, there are strict (or desired) latency requirements, meaning that the algorithm given by~\eqref{eq:fp} is constrained to a limited number of iterations.
% we only have the time to run a given budget of iterations.
In such settings, obtaining strong performance guarantees on the quality of the solution within this iteration budget is essential, particularly for safety-critical applications. 
% where the cost of unreliability is unacceptable.
% It is important to get performance guarantees provided a budget of iterations, especially for real-time safety-critical applications.

% \paragraph{Theoretical and computer-assisted worst-case analysis.}
Analyzing the worst-case performance of the first-order method~\eqref{eq:fp} has been intensely studied in optimization literature by constructing asymptotic convergence rates of the algorithms (see ~\citep[\Sec~5]{fom_book} and~\citep[\Sec~2]{lscomo}).
In contrast, the performance estimation problem (PEP) \citep{pep,pep2} approach recently emerged as a powerful tool for the numerical computation of exact worst-case guarantees for first-order methods after only a finite number of iterations.
% In pursuit of obtaining numerical guarantees within a budget of iterations, 
There are two main drawbacks of both the theoretical and computer-assisted worst-case analyses.
First, worst-case guarantees are pessimistic by definition; they provide guarantees for the most adverse problem instance among a class of problems, even if such instance occur very infrequently.
Second, these analyses typically consider a general class of functions (\eg, strongly convex and smooth functions) without leveraging the specific parametric nature inherent in problem~\eqref{prob:parametric_opt}.
In contrast to worst-case analysis, a probabilistic approach may provide less pessimistic results for applications where high-probability bounds are acceptable instead of strict worst-case guarantees~\citep{average_case}.

While guarantees for first-order methods are important for ensuring reliability, these algorithms often suffer from slow convergence in practice~\citep{zhang2020globally}.
% This practical limitation, not mitigated by tighter theoretical guarantees, has led to the exploration of alternate approaches.
% This practical limitation is not mitigated by merely having tighter theoretical guarantees, which primarily address convergence criteria without improving actual convergence speed.
% This limitation has spurred interest in alternative approaches to achieve faster convergence.
To mitigate this limitation, the \emph{learning to optimize} paradigm~\citep{amos_tutorial,l2o,balcan2020data} takes advantage of the parametric setting of our interest, and uses machine learning to predict the solutions to problem~\eqref{prob:parametric_opt}, thereby significantly reducing the solve time compared with classical solvers (\ie, those without learned components).
% One such approach that takes advantage of the parametric setting of our interest is This method leverages 
A common strategy is to learn algorithm steps~\citep{lista} or initializations~\citep{l2ws}. %that perform well over a \emph{parametric family} of optimization problems.
Learned optimizers have shown promise in a range of domains, \eg, in inverse problems~\citep{lista}, convex optimization~\citep{l2ws,rlqp}, meta-learning~\citep{maml}, and non-convex optimization~\citep{e2e_survey,online_milliseconds}.
However, guaranteeing convergence of learned optimizers is a challenge since the algorithm steps have been replaced with learned variants~\citep{l2o,amos_tutorial}.
While asymptotic convergence can sometimes be guaranteed by construction~\citep{l2ws} or by safeguarding~\citep{safeguard_convex,safeguard_l2o}, these approaches do not provide performance bounds within a fixed number of iterations.
To address this shortcoming, several methodologies have been developed to construct generalization bounds for learned optimizers, for example, using Rademacher complexity~\citep{reasoning_layer,l2ws_l4dc} and the PAC-Bayes framework~\citep{bartlett2022generalization,pac_bayes_gen_l2o,l2ws}\reviewChanges{.}
However, such bounds tend to be loose (in many cases not reported), or sometimes, the generalization bound itself can be difficult to compute.

% generalization bounds offer a potential solution; however, obtaining practical generalization guarantees for learned optimizers remains a challenge~\citep{l2o,amos_tutorial}.

\paragraph{Our contributions.}
In this paper, we present a data-driven approach based on statistical learning theory to obtain performance guarantees for both classical and learned optimizers based on fixed-point iterations.
For classical optimizers, our approach differs from existing worst-case analysis frameworks.
Instead of worst-case guarantees, we construct data-driven guarantees that hold with high probability over the parametric family of optimization problems.
Meanwhile, for learned optimizers, we rely on PAC-Bayes theory~\citep{McAllester_PAC_Bayes,pac_bayes_intro} to provide generalization bounds, and moreover, use gradient-based methods to optimize the bounds themselves.
Our method is not limited to standard metrics to analyze optimization algorithms (\eg, distance to the optimal solution or fixed-point residual); rather, we provide guarantees on any metric as long as it can be evaluated.
We summarize our contributions as follows.
\begin{itemize}
  \item We provide probabilistic guarantees for classical optimizers in two steps: first, we run the optimizer for a given number of iterations on each problem instance in a given dataset; then, we apply a sample convergence bound by solving a one-dimensional convex optimization problem.
  % We then extend this method to provide probabilistic guarantees for optimization algorithms that are initialized from a warm start.
  \item We construct generalization bounds for learned optimizers using PAC-Bayes theory. 
  In addition, we develop a framework to learn optimizers by directly minimizing the PAC-Bayes bounds using gradient-based methods.
  After training, we calibrate the PAC-Bayes bounds by sampling the weights of the learned optimizer, and subsequently running the optimizer for a fixed number of steps for each problem instance in a given dataset. Then, we compute the PAC-Bayes bounds via solving two one-dimensional convex optimization problems.
  % We not only provide generalization guarantees based on the PAC-Bayes framework for stochastic learned optimizers, but also minimize the upper bound in a two-phase process.
  % In the first, optimization phase, we apply gradient-based methods to optimize the PAC-Bayes bounds.
  % Since the tightest generalization guarantees given by the PAC-Bayes theory require solving a one-dimensional convex program, we include this step into our learning architecture as a layer.
  % We use tools from differentiable optimization to backpropagate through this layer, and show that the derivative always exists.
  % Our second, calibration phase to obtain the bounds involves a few steps: first, sampling the weights from the stochastic learned optimizer, and subsequently running the optimizer for a fixed number of steps for each sample on each instance of the parametric problem; second, applying two PAC-Bayes bounds via solving two one-dimensional convex optimization problems.
  \item We apply our method to compute guarantees for classical optimizers on several examples including image deblurring \reviewChanges{and robust Kalman filtering}, \reviewChanges{illustrating that our bounds that hold with high probability outperform bounds from worst-case analyses.}
  % significantly outperforming bounds from worst-case analyses.
  We also showcase our generalization guarantees for several learned optimizers: LISTA~\citep{lista} and its variants~\citep{alista,glista}, learned warm starts~\citep{l2ws}, and model-agnostic meta-learning~\citep{maml}.
  Our generalization guarantees accurately represent the benefits of learning by outperforming the empirical performance observed in their non-learned counterparts.
\end{itemize}
% \blue{not many parameters leads to very tight results: where to put this?}
  % \item A novel way to provide convergence guarantees.
  % Convergence guarantees sometimes exist, but not always~\citep{l2o}.
  % Usually convergence is guaranteed by construction.
  % Here, we provide PAC-Bayes-based convergence guarantees, which can provide convergence guarantees for methods that don't come with convergence guarantees.

\paragraph{Notation.}
We denote the set of non-negative vectors of length $n$ as $\reals^n_+$ and the set of vectors with positive entries of length $n$ as $\reals_{++}^n$.
We let the set of vectors consisting of natural numbers of length $n$ be $\mathbf{N}^n$.
The set of $n \times n$ symmetric matrices is denoted as $\symm^{n}$, and the set of $n \times n$ positive semidefinite matrices is denoted as $\symm^n_+$.
% For a function $F: \reals^d \rightarrow \reals^p$, we denote its Jacobian evaluated at a point $z \in \reals^n$ with $\partial F(z) \in \reals^{n \times p}$.
% When the function $F$ has multiple positional arguments, the Jacobian corresponding to the $i$-th argument (starting from index $1$) is denoted by $\partial_i F$.
We denote the trace of a square matrix $A$ with $\Tr(A)$ and its determinant as $\det A$.
For a matrix $A$, we denote its spectral norm with $\|A\|_2$ and its Frobenius norm with $\|A\|_F$.
% For a vector $v \in \reals^n$ and symmetric matrix $A \in \symm^n$, we denote the induced norm as $\|v\|_A = \sqrt{v^T Av}$.
% To denote element-wise multiplication between vectors, we use the Hadamard product, \ie, $u \odot v$.
For two vectors $u \in \reals^n$ and $v \in \reals^n$, we denote its element-wise multiplication \reviewChanges{with} $u \odot v$.
We round a vector $v$ element-wise to the nearest integer with $\round(v)$.
For a vector $v \in \reals^n$, the diagonal matrix $V \in \mathbf{S}^{n}$ with entries $V_{ii} = v_i$ for $i=1, \dots, n$ is given by $\diag(v)$.
% We denote $k$ applications of the operator $T: \reals^n \rightarrow \reals^n$ to a vector $v \in \reals^n$ with $T^k(v)$.
The all ones vector of length $d$ is denoted as $\mathbf{1}_d$.
For a vector $v$, the operation $\textbf{sign}(v)$ returns, for each element, a value of $+1$ if the corresponding element in $v$ is non-negative and $-1$ \reviewChanges{otherwise}.
% element-wise gives a value of one if the sign of $v$ is non-negative and minus one otherwise.
% To obtain the determinant of a square matrix $A$, we write ${\rm det}A$.
For any closed and convex set $C$, we let $\dist_{C}: \reals^n \rightarrow \reals$ be the distance function: $\dist_{C}(x) = \min_{s \in C} \|s - x\|_2$.
We denote expectation and probability with $\mathbf{E}$ and $\mathbf{P}$ respectively.
Finally, for a boolean condition $c$, we let $\ones(x) = 1$ if $c$ is true, and $0$ otherwise.
% \begin{equation*}
%   \mathbf{1}(c) = \begin{cases}
%     1 & \text{if } c \text{ is true}\\
%     0 & \text{otherwise}.
%   \end{cases}
% \end{equation*}
\paragraph{Outline.}
We structure the rest of the paper as follows.
Section~\ref{sec:related_work} reviews the literature on i) guarantees on classical optimizers and ii) learned optimizers, focusing on existing methods and generalization guarantees associated with them.
In \Sec~\ref{sec:learned_optimizers}, we introduce the mechanics of both classical and learned optimizers.
% \red{
% In \Sec~\ref{sec:prob_background}, we present the probabilistic bounds, a sample convergence bound and the McAllester bound, needed for the guarantees we provide later.}
In \Sec~\ref{sec:classical} we introduce our method for obtaining data-driven guarantees for classical optimizers.
We then focus on learned optimizers in the next two sections.
In \Sec~\ref{sec:gen_l2o}, we provide our generalization guarantees for learned optimizers derived from the PAC-Bayes framework.
Then in \Sec~\ref{sec:opt_pac_bayes}, we present a gradient-based algorithm designed to optimize the PAC-Bayes bound itself.
After that, in \Sec~\ref{sec:numerical_experiments}, we present numerous numerical experiments with data-driven guarantees for both classical and learned optimizers.
Finally, in \Sec~\ref{sec:conclusion} we conclude.

\section{Related work}\label{sec:related_work}
% \blue{B:maybe remove this paragraph}
% We first review two strategies for analyzing the worst-case performance of classical optimizers: theoretical and computer-assisted approaches.
% Then we transition onto learned optimizers, focusing on two predominant approaches in the literature: learning initializations and learning algorithm steps.
% Lastly, we examine existing generalization guarantees in learning to optimize, and then review a closely-related area, meta-learning, and its associated generalization guarantees.

\paragraph{\reviewChanges{Theoretical and computer-assisted worst-case analysis.}}
Theoretical convergence analysis techniques for first-order methods typically focus on general classes of problems \reviewChanges{~\citep{lscomo,fom_book}}.
% \red{Theoretical convergence analysis techniques for first-order methods typically focus on general classes of problems, for example, fixed-point problems with a contractive or averaged operator~\citep[\Sec~2.4]{lscomo}, and minimization problems with a strongly convex, smooth objective~\citep[\Sec~5]{fom_book}.}
Many analyses provide upper bounds on the asymptotic rate of convergence for an algorithm~\citep{Giselsson2014LinearCA,Hong2012OnTL} that are tight in certain cases~\citep{nesterov}.
However, there are cases where upper bounds are not tight, because they either lack corresponding lower bounds~\citep{pep2} or they are only known up to a constant~\citep{ryu_ospep}.
\reviewChanges{E}ven if the asymptotic rate is tight (\ie, there exists at least one iteration where the worst-case rate is exactly met), it may be pessimistic: the algorithm may still perform significantly better during most iterations (\eg, the local convergence rate may be better than the global one~\citep{local_lin_conv}).
Most importantly, these analyses are fundamentally pessimistic and do not exploit the parametric structure.
% Due to these considerations, our probabilistic data-driven guarantees for classical optimizers are much stronger than those that worst-case analysis can provide; we illustrate this in our numerical examples in \Sec~\ref{sec:numerical_experiments}.
A less-explored area is average-case analysis which analyzes \reviewChanges{an algorithm's performance} in expectation over a class of problems.
This approach, while not pessimistic, is designed to analyze the asymptotic convergence rate rather than provide numerical guarantees, and is further limited by its focus on unconstrained problems, as highlighted in existing works~\citep{average_case,avg_case_halting}.
\reviewChanges{
Computer-assisted approaches like PEP~\citep{pep,pep2,ryu_ospep} and \reviewChanges{integral quadratic constraints}~\citep{lessard2016IQC,Fazlyab2017AnalysisOO,Taylor2018LyapunovFF} have emerged as methods to obtain numerical worst-case guarantees, but they do not take advantage of the parametric nature of problem~\eqref{prob:parametric_opt}.}
To bridge this gap,~\citet{perfverifyqp} introduced a technique inspired by neural network verification~\citep{fazlyabsdp} to compute worst-case guarantees of fixed-point algorithms for parametric quadratic programs (QPs).
However, this approach deals with relaxations that become looser and more computationally expensive as the number of steps increases.
\reviewChanges{
Our probabilistic guarantees for classical optimizers complement worst-case analysis by demonstrating that, for those willing to accept high-probability bounds instead of stricter worst-case guarantees, our approach provides significantly stronger performance bounds over the parametric family.
}

\paragraph{\reviewChanges{Learning initializations and algorithm steps.}}
A common strategy in learning to optimize is to learn high-quality initializations.
\citet{l2ws_l4dc} and \citet{l2ws} unroll, \ie, differentiate through~\citep{algo_unrolling,reasoning_layer}, algorithm steps to learn warm starts, thereby reducing solve times for \reviewChanges{convex problems}.
\reviewChanges{Some works learn initializations in a decoupled fashion~\citep{warm_start_power_flow,mak2023learning,constraint_informed_traj_ws,learn_active_sets}, while others directly learn the optimal solution, and rather than warm-starting an algorithm, ensure feasibility and optimality with a correction step~\citep{donti2021dc3,deep_learning_mpc_karg,mpc_constrained_neural_nets}.}
% Other works learn initializations in a decoupled fashion, for instance, in optimal power flow~\citep{warm_start_power_flow,mak2023learning}\reviewChanges{, trajectory optimization~\citep{constraint_informed_traj_ws}, and control~\citep{learn_active_sets}.}
% Other works directly learn the optimal solution, and rather than warm-starting an algorithm, directly ensure feasibility and optimality with a correction step\reviewChanges{~\citep{donti2021dc3,deep_learning_mpc_karg,mpc_constrained_neural_nets}}.
% In this work, we propose a method designed to integrate with these methods, optimizing and calibrating generalization guarantees for \emph{any} learned optimizer.

An alternate approach is to learn the algorithm steps.
In convex optimization, learned algorithm steps have been shown to decrease solve times through learned hyperparameters~\reviewChanges{\citep{qp_accelerate_rho,rlqp,metric_learning}}, and learned acceleration schemes~\citep{neural_fp_accel_amos}.
\reviewChanges{While a lack of convergence guarantees was seen as a potential downside of learning algorithm steps~\citep{amos_tutorial}, some works have addressed this by safeguarding~\citep{safeguard_convex,safeguard_l2o}, providing convergence rate bounds~\citep{learn_mirror}, and constraining the updates~\citep{banert2021accelerated}.}
\reviewChanges{The idea of learning algorithm steps has also been used to solve non-convex problems~\citep{bai2022neural,sjolund2022graphbased,balcan_partition,balcan_cluster} and inverse problems~\citep{lista,alista,lista_cpss,hyperlista,Diamond2017UnrolledOW,plug_and_play_ryu,deep_unfolding_wireless}}.
\reviewChanges{Our method is} designed to integrate with these methods, optimizing and calibrating generalization guarantees for \emph{any} learned optimizer.
% As in the case for learned initializations, we propose a method designed to integrate with these methods, and endow them with strong generalization guarantees.
% , matrix factorization problems~\citep{sjolund2022graphbased}, partitioning problems~\citep{balcan_partition}, and clustering problems~\citep{balcan_cluster}.

% In particular, we will provide generalization guarantees.
% Our method is designed to work in conjunction with learned optimizers.
% \begin{itemize}
%   \item sparse coding: \citep{lista,alista,glista,lista_cpss}
%   \item learned warm starts: \citep{l2ws,l2ws_l4dc}
%   \item In sparse coding,~\citet{lista_cpss} introduce a necessary condition on the learned weights for asymptotic convergence.
% \end{itemize}

% \paragraph{Learning surrogate problems.}
% Learned optimizers have been applied to other areas such as non-convex optimization~\citep{online_milliseconds} and meta-learning~\citep{learning_to_optimize_malik,maml}.S
% E2E survey: \citep{e2e_survey}

\paragraph{Generalization bounds in learned optimizers.}
Despite strong empirical outcomes in certain settings, learned optimizers 
% can struggle to generalize; moreover, most of these methods 
lack generalization guarantees~\reviewChanges{\citep{l2o,amos_tutorial,mL2O}}.
% \red{To address the poor generalization of learned optimizers in practice, \citet{mL2O} use ideas from meta-learning to develop methods that quickly adapt to out-of-distribution tasks.}
To address this\reviewChanges{,}~\citet{pmlr-v206-sucker23a} and~\citet{Sucker2024LearningtoOptimizeWP} optimize PAC-Bayesian guarantees based on exponential families, but they assume exponential moment bounds, a condition 
% which can be restrictive and 
difficult to verify in practice. 
In addition, they assume %the learned optimizer assumes 
a specific 
% form with a fixed initialization, where the 
update function: a multi-layer perceptron~\citep{Sucker2024LearningtoOptimizeWP} or a gradient step with learned hyperparameters~\citep{pmlr-v206-sucker23a}.
On the other hand, our method \reviewChanges{can} be used \emph{in conjunction} with any learned optimizer, including ones with learned initializations.
% , and does not require us to find an architecture.
\reviewChanges{Other works} provide guarantees through Rademacher complexity~\citep{reasoning_layer,l2ws_l4dc}, \reviewChanges{the PAC-Bayes framework~\citep{bartlett2022generalization,pac_bayes_gen_l2o,l2ws}, and pseudo-dimesion bounds~\citep{balcan_gen_guarantees}}.
% \red{However, these guarantees are more theoretical in nature rather than practical, typically focusing on the rate at which the generalization bounds converge to zero as a function of the number of training samples.
% Hence, the} 
\reviewChanges{Yet, these} bounds tend to be loose or difficult to compute.
We construct numerical bounds by optimizing the PAC-Bayes bounds themselves, a strategy previously used for classification~\citep{nonvacuous_pac_bayes} and control~\citep{majumdar2021pac}.

% A related area to learning to optimize called \emph{data-driven algorithm design}~\citep{balcan2020data} also involves optimizing algorithms on a parametric family of problem instances, but focuses on combinatorial optimization problems like clustering~\citep{balcan_cluster} and partitioning~\citep{balcan_partition}.
% Generalization guarantees have also been explored in this area~\citep{balcan_gen_guarantees}.
% In this paper, we focus on algorithms for continuous rather than combinatorial optimization.
% Moreover, we optimize our PAC-Bayes generalization guarantees themselves to provide numerically strong bounds.

\paragraph{Meta-learning.}
Meta-learning~\citep{hospedales2020metalearning,meta_learning_survey} overlaps with learning to optimize when the parametric problem is a learning task~\citep{l2o}.
\reviewChanges{B}oth learned initializations~\citep{maml} and algorithm updates~\citep{learning_to_optimize_malik,learn_learn_gd_gd,VeLO} have been effectively used in meta-learning.
\reviewChanges{Methods have been developed to improve generalization in practice~\citep{almeida2021generalizable,l2gen_l2o} and to provide theoretical generalization bounds~\citep{pmlr-v80-amit18a,pmlr-v97-balcan19a}.
Yet existing bounds tend to be challenging to evaluate or loose.}
% On the theoretical front, many works provide generalization bounds for meta-learning~\citep{pmlr-v80-amit18a,pmlr-v97-balcan19a}, yet existing bounds tend to be challenging to evaluate or vacuous.
Addressing this issue, \citet{farid_metalearning} derive a novel PAC-Bayes bounds, focusing on practically useful guarantees.
\reviewChanges{Our method is more general in that it can be applied to not only learning tasks, but also optimization and inverse problems.}
% While our method can be used to endow meta-learning methods with generalization guarantees, it is not limited to learning tasks -- we also consider parametric optimization and inverse problems.
%\red{(at the meta-level)}

% \blue{B: put the following paragraph together with generalization bounds and learned optimizers paragraph.}
% \paragraph{Data-driven algorithm design.}

% \red{B: think aobut notation. Try to be consistent wth amortized tutorial. I would probably do $g_{w}(z^k, \theta)$. We may want to change $\theta \to x, w \to \theta$. $z$ stays the same.}
\section{Classical and learned optimizers}\label{sec:learned_optimizers}
% In this section, we detail the mechanics and metrics used for training learned optimizers, laying the groundwork for our subsequent development of PAC-Bayes generalization guarantees.
% In this section, we detail how to run classical and learned optimizers, laying the groundwork for our subsequent development of PAC-Bayes generalization guarantees.
% In this section, we provide details on classical and learned optimizers, focusing on how to run them, evaluate them, and for learned optimizers, how to train them.
In this section, we delve into the mechanics of classical and learned optimizers, laying the groundwork for the bounds we provide later.
In \Sec~\ref{subsec:classical} we explain how to run and evaluate classical optimizers, focusing on fixed-point optimization algorithms.
For learned optimizers, we first explain how to run and evaluate them given fixed weights in \Sec~\ref{subsec:mechanics_l2o}, and then how to train them to learn the weights in \Sec~\ref{subsec:train_l2o}.

\subsection{Running and evaluating classical optimizers}\label{subsec:classical}
As it turns out, problem~\eqref{prob:parametric_opt} can often be written as an equivalent fixed-point problem
\begin{equation}\label{prob:l2ws}
  \begin{array}{ll}
\mbox{find} & z \quad \mbox{subject to} \quad z = T(z,x),
\end{array}
\end{equation}
where $T: \reals^n \times \reals^d \rightarrow \reals^n$ is the fixed-point operator.
Indeed, nearly all convex optimization problems can be reformulated as a finding the fixed-point of an operator~\citep{lscomo} which often represents the optimality conditions~\reviewChanges{\citep{COSMO}}.
We denote the set of fixed-points for the fixed-point problem parametrized by $x$ to be $\fix T_x$.
\reviewChanges{Note that} the ground truth solution $z^\star(x)$ satisfies the fixed-point condition $z^\star(x) \in \fix T_x$.
For classical optimizers, we focus on the parametric fixed-point problem~\eqref{prob:l2ws} as it is a convenient way of analyzing worst-case performance~\citep{infeas_detection,Giselsson2014LinearCA}.
This in turn allows for a direct comparison of our guarantees with those previously established.

\paragraph{Initializations.}
In classical optimizers, the initialization is not learned \reviewChanges{and is typically set} to the zero vector, \ie, $z^0(x) = 0$, which is often referred to as a \emph{cold start}.
In contexts where we have an estimate of the solution, it is common to \emph{warm-start} the problem from this point.
For example, in \reviewChanges{model predictive control}~\citep{borrelli_mpc_book}, where similar instances of the same problem are solved sequentially, the problems are often warm-started from the previous solution shifted by one time index~\citep{nonlinear_mpc}.
% \red{in MPC problems the initialization is the shifted previous solution. It would be nice to reference that here.}

\paragraph{Algorithm steps.}
% One popular method to solve the problem~\eqref{prob:l2ws} is to repeatedly apply the operator $T$, obtaining the iterates given by $z^{k+1}(x) = T(z^k(x), x)$ from \Eqn~\eqref{eq:fp}.
\reviewChanges{The iterates $z^{k+1}(x) = T(z^k(x), x)$ in \eqref{eq:fp} are a popular way to solve problem~\eqref{prob:l2ws}.}
Many classical optimizers consist of fixed-point iterations, \eg, gradient descent, proximal gradient descent~\citep{prox_algos}, and ADMM~\citep{Boyd_admm}.

\paragraph{Evaluation metrics.}
A variety of metrics can be used to evaluate the performance of algorithms. 
A standard metric is the fixed-point residual~\citep[\Sec~2.4]{lscomo}
\begin{equation*}
\phi^{\rm fp}(z, x) = \|T(z, x) - z\|_2,
\end{equation*}
which quantifies the gap between successive iterations. 
Such metrics, assess the quality of candidate solutions for problems parametrized by $x$. 
To determine if an optimization algorithm meets specific performance benchmarks, we introduce the 0--1 error function
\begin{equation}\label{eq:zero_one_classical}
e(x) = \mathbf{1}\left(\phi \left(z^k(x), x \right) \geq \epsilon \right),
\end{equation}
assigning a value of 1, if the performance metric $\phi(z,x)$ exceeds a specified threshold $\epsilon$ after $k$ steps, indicating a failure to meet the desired criteria, and 0 otherwise.
We later provide guarantees for this error function $e(x)$, for \emph{any} underlying metric $\phi(z,x)$ in \Sec~\ref{sec:classical}.

\paragraph{Convergence.}
Under certain conditions on the operator $T$, the fixed-point iterates in~\eqref{eq:fp} are known to converge to a fixed-point, \ie, $\lim_{k \rightarrow \infty} \|z^k(x) - z^\star(x)\|_2 = 0$ for some $z^\star(x)$ in the set of fixed-points $\fix T_x$.
For instance, if the operator $T$ is contractive, linearly convergent, or averaged (see Appendix~\ref{sec:op_theory} for their definitions), and the set of fixed-points is non-empty, then the iterates are guaranteed to converge~\reviewChanges{\citep[\Sec~2.4]{lscomo}}.
\reviewChanges{We refer the reader to Appendix~\ref{sec:op_theory} for the rates of convergence for these cases.}

\subsection{Running and evaluating learned optimizers}\label{subsec:mechanics_l2o}
% Many learning to optimize methods quickly compute a candidate solution $\hat{z}_\theta(x)$ for problem~\eqref{prob:parametric_opt} by integrate machine learning into iterative optimization algorithms.
% Learned optimizers typically learn either the initial point or the steps for a given algorithm, by adjusting some weights denoted as~$\theta \in \reals^p$.
The goal of learning to optimize methods is to accelerate an algorithm to quickly find a high-quality candidate solution $\hat{z}_\theta(x)$ for problem~\eqref{prob:parametric_opt}.
% These methods commonly integrate machine learning into iterative optimization algorithms, allowing the optimization procedure itself to be learned. 
Learned optimizers typically learn either the initial point or the steps for a given algorithm, by adjusting some weights~$\theta \in \reals^p$.

% \blue{B: these lines are very repetitive. Shorten them}
\paragraph{Learned initializations.}
Some learned optimizers focus on learning the initializations for algorithms~\citep{l2ws,maml}.
Typically, this involves predicting an initial point $z^0 \in \reals^n$ from the parameter $x$ \reviewChanges{with a function $h_\theta: \reals^d \rightarrow \reals^n$}:
\begin{equation*}
  z^{0}_\theta(x) = h_\theta(x).
\end{equation*}
% Here the function $h_\theta: \reals^d \rightarrow \reals^n$ is the learned function for determining the initial point. % of the optimizer.

\paragraph{Learned algorithm steps.}
Another common strategy in learned optimizers is to learn the steps of the algorithm, which can be represented as
\begin{equation*}
z^{k+1}_\theta(x) = T_\theta(z^k_\theta(x), x).
\end{equation*}
Here\reviewChanges{,} the function $T_\theta: \reals^n \times \reals^d \rightarrow \reals^n$ is the learned update rule.
\reviewChanges{Note that the iterates $z_\theta^k(x)$ depend on the parameter $x$ and the weights $\theta$.}
% We make the dependence of the iterates on the parameter $x$ and the learned weights $\theta$ clear with the notation $z_\theta^k(x)$.

\paragraph{Evaluation metrics.}
% \red{
%   \begin{itemize}
%     \item evaluation metric depends on the task at hand and what we are interested in
%     \item for example
%     \item one: inverse problems: regression
%     \item two: meta-learning
%     \item the metric $\phi$ could be the same or different from the objective $f$
%   \end{itemize}
% }
% Depending on the task at hand, the objective $f$ can take different forms.
The evaluation metric $\phi$ depends on the task at hand.
For inverse problems, a common metric of interest is the squared distance to the ground truth solution $\phi^{\rm mse}(z,\param) = \|z - z^\star(x)\|^2_2$.
% \begin{equation*} %\label{eq:reg}
%     \phi^{\rm mse}(z,\param) = \|z - z^\star(x)\|^2_2.
%   \end{equation*}
In meta-learning, a common measure is the performance on a learning task over an unseen dataset $\mathcal{D}^{\rm test}$ and learning objective $\mathcal{L}$~\citep{maml,learning_to_optimize_malik}, which fits \reviewChanges{into} our framework with the performance metric $\phi^{\rm meta}(z,x) = \mathcal{L}(z, \mathcal{D}^{\rm test})$.
% This meta-learning formulation fits into our framework with the objective
% \begin{equation*}
%   \phi^{\rm meta}(z,x) = \mathcal{L}(z, \mathcal{D}^{\rm test}).
% \end{equation*}
As in the classical optimizers case, we consider the 0--1 error function associated with an underlying metric $\phi$.
In this case, the error function depends on the weights $\theta$:
\begin{equation}\label{eq:zero_one_loss_l2o}
  e_\theta(x) = \mathbf{1}\left(\phi \left(z^k_\theta(x), x \right) \geq \epsilon \right).
\end{equation}
It is important to remark that the metric $\phi$ can also be different from the objective $f$ from problem~\eqref{prob:parametric_opt}.
% \blue{B: what is $f$ here? Was there ever a doubt you wanted to use $f$ here?}
Our generalization guarantees are designed to provide bounds for the error function $e_\theta(x)$ with \emph{any} underlying metric $\phi$.
% \red{B: remember to connect $f$ and $\phi$ here.}
% Depending on the task at hand, the objective $f$ can take different forms.
% One common example of an objective-based loss is the fixed-point residual~\citep{l2ws,neural_fp_accel_amos}
% \begin{equation}\label{eq:fp_res}
%     f_\theta(z,\param) = \|z - T_\theta(z,x)\|_2.
%   \end{equation}
% In meta-learning, the objective loss often measures the performance on a learning task equipped with a dataset $\mathcal{D}^{\rm test}$ and learning objective $\mathcal{L}$~\citep{maml,learning_to_optimize_malik}.
% This meta-learning formulation fits into our framework with the objective
% \begin{equation*}
%   f_\theta(z,x) = \mathcal{L}(z, \mathcal{D}^{\rm test}).
% \end{equation*}
% As in the classical case, we consider the 0--1 error function associated with an underlying metric $\phi$.
% In this case, the error function depends on the weights $\theta$:
% \begin{equation}\label{eq:zero_one_loss_l2o}
%   e_\theta(x) = \mathbf{1}\left(\phi \left(z^k_\theta(x), x \right) \geq \epsilon \right).
% \end{equation}
% \red{B: remember to connect $f$ and $\phi$ here.}

\paragraph*{Convergence for learned optimizers.}
\reviewChanges{When} the algorithm steps are replaced with learned variants, convergence may not be guaranteed~\citep{l2o,amos_tutorial}.
% \red{Nonetheless, in particular cases, convergence can sometimes be guaranteed by construction~\citep{l2ws,banert2021accelerated} or by safeguarding~\citep{safeguard_convex}.}
% To address this, methods have been developed to safeguard learned optimizers~\citep{safeguard_convex}, reverting to a fallback update if satisfactory progress is not made.
% Other methods guarantee convergence, for example, by only learning the initialization for algorithms that are known to converge~\citep{l2ws} or by ensuring the learned updates do not deviate too much from a known convergent method~\citep{banert2021accelerated}.

\subsection{Training learned optimizers}\label{subsec:train_l2o}
In this subsection, we formulate the learning to optimize training problem, beginning with the loss functions.
Depending on the task at hand, the loss can take varying forms, generally falling into two categories: regression-based and objective-based~\citep{amos_tutorial}.
\begin{description}
  \item[Regression-based loss.] The \emph{regression-based loss} measures the distance to a ground truth solution $z^\star(x)$, \ie,
\begin{equation}\label{eq:reg_loss}
  \ell^{\rm reg}_\theta(\param) = \|\hat{z}_\theta(x) - z^\star(\param)\|_2^2.
\end{equation}
  \item[Objective-based loss.] The \emph{objective-based loss} directly penalizes the objective $f$:
\begin{equation}\label{eq:obj_loss}
  \ell^{\rm obj}_\theta(\param) = f(\hat{z}_\theta(x), x).
\end{equation}
Unlike the regression-based loss, the objective-based loss does not require access to ground truth solutions.
% For a more-detailed comparison of the two types of losses see~\citet[\Sec~2.2]{amos_tutorial}.
% We note that the objective-base loss does not require having accessing to a ground truth solution $z^\star(x)$ as opposed to the regression-based loss which does.
\end{description}

\paragraph{The learning to optimize training problem.}
Given the loss function, algorithm steps, and initialization we formulate the training problem as
\begin{equation}\label{prob:simplified_l2o}
  \begin{array}{ll}
\mbox{minimize} & \mathbf{E}_{x \sim \mathcal{X}} \ell_\theta(x)\\
  \mbox{subject to} &z^{k+1}_\theta(x) = T_\theta(z^k_\theta(x), \param),\quad k=0,1,\dots, K-1 \\
  &z^0_\theta(x) = h_\theta(\param).
\end{array}
\end{equation}
% \blue{B:  any reason we do not use $k$ to index the steps? It seems way more natural since the maximum number is $K$.}
Here, $K$ is the number of algorithm steps used during training, and the loss function $\ell_\theta(x)$ is either chosen to be the regression-based loss $\ell^{\rm reg}_\theta(x)$ or the objective-based loss $\ell^{\rm obj}_\theta(x)$.
The loss function is applied to the $K$-th iterate $\hat{z}_\theta(x) = z^K_\theta(x)$, but, in principle, it could be a (weighted) sum of a loss function applied all the iterates $z^0_\theta(x), \dots, z^K_\theta(x)$.
Since in general we do not know the distribution $\mathcal{X}$ to solve problem~\eqref{prob:simplified_l2o}, we approximate the expectation over $N$ independent and identically distributed (i.i.d.) training samples $S = \{x_i\}_{i=1}^N$.
In \Sec~\ref{sec:gen_l2o}, we modify the empirical training problem to provide guarantees to unseen data.

\section{Probabilistic guarantees for classical parametric optimization}\label{sec:classical}
% \red{B: maybe ``deterministic'' instead of ``classical''? Classical is a strange wording which may not always fit. Or maybe just ``parametric optimization'' without any extra wording?}
In this section, we use statistical learning theory to provide probabilistic guarantees for classical parametric optimization.
% \blue{tie back to fixed-point problem~\eqref{prob:l2ws}}
In particular, we focus on the fixed-point problem setting~\eqref{prob:l2ws}, and obtain performance guarantees on the quality of the iterates from~\eqref{eq:fp}, $z^{k+1}(x) = T(z^k(x), x)$.
Recall that the parameter $x$ is drawn in an i.i.d. fashion from distribution $\mathcal{X}$. We first provide bounds for algorithms initialized to the zero vector (\ie, cold-started) and then consider how to adapt the bounds to include warm starts.
% We focus on the broad setting of fixed-point problems from~\eqref{eq:fp}, where the goal is find a vector $z$ such that $z = T(z,x)$.
% \red{and on fixed-point iterations.}
% We set the initial point to be the zero vector.
% \red{does it need to be fixed-point?}
% Note that this is a general class of problems; indeed the optimality conditions for almost all convex optimization problems can take the form of problem~\ref{eq:fp}.
% \red{fix the metric}
% The metric we use for classical optimizers is the fixed-point residual $\|z-T(z,x)\|_2$.
% In the case of classical algorithms, neither the initialization nor the algorithm steps are learned from data.
% Since we only consider classical optimizers in this section, neither the initialization nor the algorithm steps are learned from data.

% \subsection{Obtaining guarantees via statistical learning theory}\label{subsec:beyond_classical_rates}
\paragraph{Obtaining guarantees via statistical learning theory.}
% In this subsection, we provide generalization guarantees for classical parametric optimization.
Given an underlying metric~$\phi$, a number of algorithm steps $k$, and a tolerance $\epsilon$, we consider the 0--1 error $e(x)$ given by \Eqn~\eqref{eq:zero_one_classical} which takes a value of 1 if $\phi$ is above $\epsilon$ after $k$ steps and 0 otherwise.
% Note that the dependence on $\theta$ has been removed since we consider classical parametric optimization in this section.
% Note that classical optimizers do not have a learning component, so we remove the dependence on the weights $\theta$.
% Since the parameter $x$ is drawn from the distribution $\mathcal{X}$, the error function $e(x)$ is a Bernoulli random variable.
There are three steps to obtain bounds on the risk $\risk$ given $N$ sample parameters \reviewChanges{$S$} as depicted in Figure~\ref{fig:classical_procedure}.
% \blue{B: we are redefining the training data here. Maybe there is no need if we do it before.}
First, for each sample $x$ we run $k$ fixed-point steps starting from the zero vector to obtain $z^k(x)$.
Second, we compute the empirical risk $\emprisk$, the fraction of problems that fail to reach the desired tolerance in $k$ steps.
Last, we apply the sample convergence bound from \Thm~\ref{thm:sample_conv_bound} in \reviewChanges{Appendix \Sec~\ref{sec:prob_background}} to bound the risk $\risk$ with probability at least $1 - \delta$:
\begin{equation}\label{eq:classical_lem}
\risk \leq \KLinv \biggl(\emprisk ~\Bigg|~  \frac{\log (2 / \delta)}{N} \biggr).
\end{equation}
We remark that other concentration bounds could be used in~\eqref{eq:classical_lem}, and that we could instead bound $\mathbf{E}\phi(z^k(x),x)$ directly instead of the risk $\mathbf{E}e(x)$.
% Typically a concentration bound requires having an upper bound on the loss.
Typically, using a concentration bound requires an upper bound on the metric of interest, a condition trivially satisfied by the error function.
% Typically, using a concentration bound requires an upper bound on the metric of interest; hence, we choose the 0--1 error function in this paper.
The choice to bound the error function is driven by this convenience, which also proves particularly beneficial in the analysis of learned optimizers, as discussed in Section~\ref{sec:gen_l2o}.
% We mainly choose to bound the 0--1 error function in this paper as the upper bound of one is free, an artifact that becomes more important when we consider learned optimizers in \Sec~\ref{sec:gen_l2o}.

\begin{figure}[!h]
  \centering
  \includegraphics[width=1.0\linewidth]{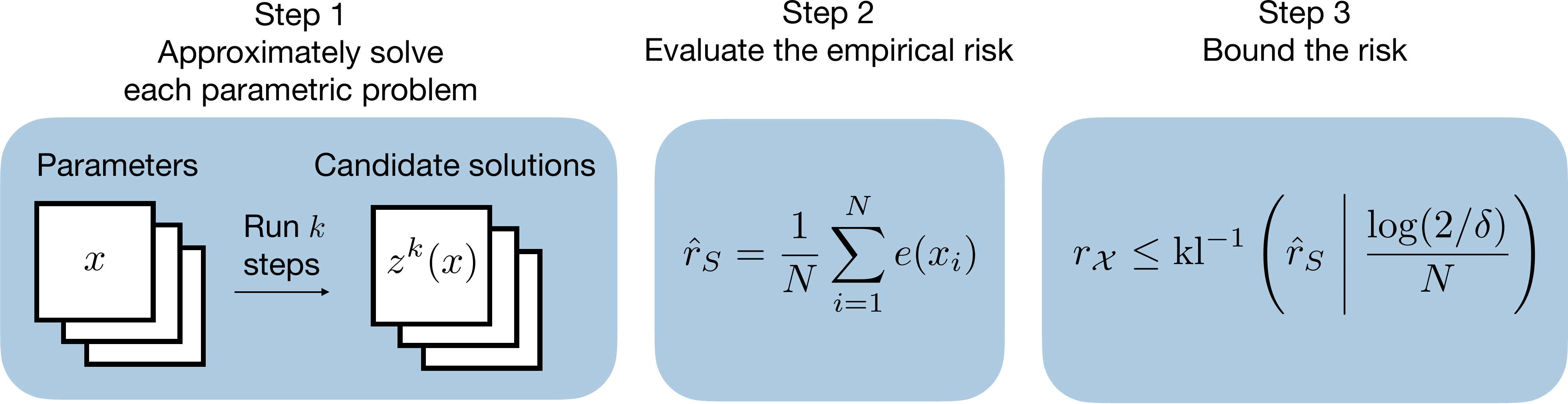}
\centering
\caption{
The procedure to generate probabilistic guarantees for classical optimizers.
Given $N$ parameter samples, we first approximately solve each parametric problem by running $k$ fixed-point steps in step $1$.
Then given an error function $e(x)$ with an underlying metric $\phi$, number of algorithm steps $k$, and tolerance $\epsilon$, we evaluate the empirical risk $\emprisk$ in step $2$.
Lastly in step $3$, we apply the sample convergence bound to bound the risk $\risk$ with high probability.
}
\label{fig:classical_procedure}
\end{figure}
% \blue{this is very similar to the sample convergence bound. what should we write here?}
% We apply the sample convergence bound~\eqref{eq:langford_bound}.
% There is no need to consider stochastic optimizers here.
% \red{explain how we use the theorem 1 --treat as Bernoulli}
% \begin{lemma}\label{lem:classical_guarantee}
%   Let $e_{\mathcal{X}}$ be the expected error function in the range $[0,1]$ and let $e_S$ be its empirical counterpart over $N$ samples.
%   For any $\delta \in (0,1)$ the following bound holds with probability at least $1 - \delta$:
%   \begin{equation*}
%     {\rm KL}(e_S \parallel e_{\mathcal{X}}) \leq \frac{\log (2 / \delta)}{N}.
%   \end{equation*}
% \end{lemma}
% \red{B: How is this different than Thm 1? Maybe I am missing something. I would just use one result instead. Also, let's always use one way to express the empirical distribution (instead of the summation in Theorem 1)}

\paragraph{Incorporating warm starts.}
It is natural to wonder if the bound~\eqref{eq:classical_lem} can be adapted to algorithms initialized from warm starts rather than from the zero vector.
% from always initializing from the zero vector to initializing from warm starts.
Indeed this adaptation is feasible, as long as the errors $e(x)$ are i.i.d. random variables.
One setting where this condition is met, what we call the nearest-neighbor warm start~\citep{l2ws} setting, assumes access to a base set of $N^{\rm base}$ problem parameters and a corresponding optimal solution for each one.
% Here, we assume we are provided a base set of $N^{\rm base}$ problems each with a corresponding optimal solution, in addition to the $N$ parametric samples.
The nearest-neighbor warm start initializes the sample problem with the given optimal solution of the nearest of the base problems measured by distance in terms of its parameter $x \in \reals^d$.
Since the metric $e(x)$ is still i.i.d., the sample convergence bound from inequality~\eqref{eq:classical_lem} holds.

\paragraph{Strengthening the bound with worst-case guarantees.}
One downside of the guarantee given by the sample convergence bound is that a non-zero term $(1/N)\log(2/\delta)$ in inequality~\eqref{eq:classical_lem} prevents the risk $\risk$ from ever reaching zero even if the empirical risk $\emprisk$ is zero and $k$ is very large.
If the underlying metric is the fixed-point residual, the worst-case guarantees from~\eqref{eq:lin_conv_rate} and~\eqref{eq:avg_rate} can provide a stronger result: a risk of exactly zero that holds with probability one, for a large enough number of iterations (provided an upper bound on the distance from the initialization to the set of optimal solutions).
In this case, we simply adapt our bounds to take the better of the worst-case guarantee and the probabilistic bound in \Eqn~\eqref{eq:classical_lem}.

\section{Generalization bounds for learned optimizers}\label{sec:gen_l2o}
In this section, we derive generalization bounds for learned optimizers using tools from PAC-Bayes theory. 
In particular, we adapt \reviewChanges{Maurer's} bound from \Thm~\ref{thm:maurer} \reviewChanges{presented in Appendix \Sec~\ref{sec:prob_background}} to allow for a data-dependent prior and then apply it for learned optimizers.
% \blue{provide contribution: data-dependent prior: generalization of other works that allows us to have multi-valued prior}
% A key feature is the use of a data-dependent prior in particular, we extend the work in the literature, allowing  

\reviewChanges{Maurer's} bound from\reviewChanges{~\eqref{eq:pac_bayes_kl_inverse}} is used to provide bounds where the weights are drawn from a distribution, while learned optimizers (as outlined in \Sec~\ref{sec:learned_optimizers}) are deterministic.
To reconcile this, we adapt the learning to optimize framework so that the weights of the learned optimizers $\theta$ are drawn from a posterior distribution~$P$.
Then, considering the 0--1 error metric $e_\theta(x)$ from \Eqn~\eqref{eq:zero_one_loss_l2o} as a function of the loss, \ie, $e_\theta(x) = \mathbf{1}(\ell_\theta(x) \ge \epsilon)$, we aim to bound the expected risk of the posterior $\exprisk(P)$ defined in \Eqn~\eqref{eq:exp_risk}.
% \blue{B: the following is a bit convoluted. I would try to simplify it in one less phrase.}
As in the case of classical optimizers, we choose to bound the error function rather than the loss.
This approach is particularly useful for learned optimizers, as obtaining an upper bound on the loss, an important assumption in the PAC-Bayes framework, can be difficult due to the lack of convergence guarantees~\citep{amos_tutorial}.
Bounding the error allows us to bypass this complication, as an upper bound of one is trivially given.

% The reason for doing so for learned optimizers is more apparent.
% Recall that the convergence of learned optimizers is not always guaranteed~\citep{amos_tutorial}, making it challenging to obtain an upper bound on the loss -- a component typically needed in the PAC-Bayes framework.
% Bounding the error allows us to bypass this complication, as an upper bound of one is trivially given.
% , noting that obtaining an upper bound on the loss may not be possible for learned optimizers.
% Indeed, recall that convergence of learned optimizers is not always guaranteed~\citep{amos_tutorial}.

% As in the case of classical optimizers, we choose to bound the error function rather than the loss, noting that obtaining an upper bound on the loss may not be possible for learned optimizers.
% Indeed, recall that convergence of learned optimizers is not always guaranteed~\citep{amos_tutorial}.

% This results in a bound on the expected risk $\exprisk(P)$ for a posterior distribution $P$.
% An artifact of using the PAC-Bayes framework is that the generalization guarantees are designed for \emph{stochastic} weights.

\paragraph{The posterior.}
To obtain the KL divergence in closed form from~\Eqn~\reviewChanges{\eqref{eq:pac_bayes_kl_inverse}}, we consider posterior and prior distributions on the algorithm weights that are multivariate normal distributions.
We further enforce a diagonal covariance structure for both.
% We consider posterior distributions that are multivariate normal distributions with diagonal covariance.
Our posterior takes the form $\mathcal{N}(w,{\bf diag}(s))$ where the mean is $w \in \reals^p$ and the covariance is ${\bf diag}(s) \in \symm^p_+$.
We use the notation $\mathcal{N}_{w,s} = \mathcal{N}(w, {\bf diag}(s))$ for convenience.

% \paragraph{Groups for the prior variance.}
\paragraph{The prior.}
In the next section, we would like to optimize over the bounds themselves; however, recall that the vanilla \reviewChanges{Maurer} bound from \Thm~\ref{thm:maurer} requires that the prior is fixed and independent of the training samples.
Our strategy is to consider a data-dependent prior where the mean is fixed and the variance is optimized over.
We then round the variance to a pre-defined grid where we use a union-bound argument to satisfy the assumptions of \Thm~\ref{thm:maurer}.
% We use a union-bound argument that allows us to choose a data-dependent variance from a pre-determined grid.
This strategy has been taken in the literature, for example in~\citet{nonvacuous_pac_bayes} and~\citet{langford_union_prior} where the covariance matrix takes the form $\Lambda = \lambda I$ for a scalar $\lambda$.
We generalize this approach, by instead partitioning the weights into $J$ groups and optimizing over a \emph{vector} $\lambda \in \reals^J_+$ rather than a scalar.
For the $j$-th group (where $j \in \{1, \dots, J\}$), we let $\mathcal{I}_j$ be the corresponding index set of weights. 
We construct the diagonal prior variance $\Lambda \in \symm^p_+$ by assigning the value $\lambda_j$ to the indices in group $\mathcal{I}_j$, \ie,
\begin{equation*}
\diag(\Lambda)_{\mathcal{I}_j} = \lambda_j \mathbf{1}_{|\mathcal{I}_j|}, \quad \text{for } j=1, \dots, J.
\end{equation*}
Here, $|\mathcal{I}_j|$ is the cardinality of the set $\mathcal{I}_j$.
This partitioning approach allows for a more nuanced adaptation to weights associated with different groups.
Consider, for instance, LISTA-type algorithms, where distinct weights are used for shrinkage thresholds and step sizes~\citep{lista,alista} (see \Sec~\ref{subsec:lista} for more details). 
Intuitively, accommodating different priors for each group can be advantageous because different weight groups can have different orders of magnitudes.
Hence, it is natural that allowing for different variances across partitions is beneficial.

\paragraph{Main generalization bound for learned optimizers.}
We now give our main generalization bound theorem which uses the union-bound argument to allow for a data-dependent prior.
Specifically, we enforce that the prior variance term $\lambda \in \reals_+^J$ takes the form $\lambda = \lambda^{\rm max} \exp(-a / b)$ for some $a \in \mathbf{N}^J$.
We design \reviewChanges{Maurer's} bound to hold for a given $a$ with probability 
\begin{equation}\label{eq:delta_a}
  \delta_a = \left(\frac{6}{\pi^2}\right)^J \frac{\delta}{\prod_{j=1}^J a_j^2},
\end{equation}
for some pre-determined $\delta \in (0,1)$.
Then with probability at least $1 - \delta$, \reviewChanges{Maurer's} bound holds uniformly for all $a \in \mathbf{N}^J$.
This strategy generalizes the union-bound arguments made in the literature~\citep{nonvacuous_pac_bayes,langford_union_prior} to allow for $J$ to be larger than one.
We formalize this result with the following theorem.

\begin{theorem}\label{thm:gen_thm}
  Consider a set of \reviewChanges{$N$ i.i.d. samples $S$}.
  Let the prior mean $w_0 \in \reals^p$, and the prior variance hyperparameters $\lambda^{\rm max} \in \reals_+$ and $b \in \reals_+$, be independent of the samples.
  Then for any $\delta \in (0, 1)$, posterior distribution $\mathcal{N}_{w,s}$, and vector $a \in \mathbf{N}^J_+$, with probability at least $1 - \delta$ the following bound holds:
\begin{equation}\label{eq:our_pac_bayes_bound}
  % {\rm kl}\left(\expemprisk(\mathcal{N}_{w,s}) \parallel \exprisk(\mathcal{N}_{w,s})\right) \leq B(w,s,\lambda).
  \reviewChanges{
  \exprisk(\mathcal{N}_{w,s}) \leq {\rm kl}^{-1}\left(\expemprisk(\mathcal{N}_{w,s}) ~\vert~ B(w,s,\lambda) \right).}
\end{equation}
Here, $\lambda = \lambda^{\rm max} \exp(-a / b)$ and the regularization term is
\begin{equation}\label{eq:pen}
  \reviewChanges{
  B(w,s,\lambda) = \frac{1}{N} \left(\KL \left(\mathcal{N}_{w,s}\parallel\mathcal{N}(w_0,\Lambda)\right) + \sum_{j=1}^J 2 \log \left(b \log \frac{\lambda^{\rm max}}{\lambda_j}\right) + J \log \frac{\pi^2}{6} + \log \frac{2 \sqrt{N}}{\delta}\right)}.
\end{equation}
\end{theorem}
Using \Eqn~\eqref{eq:kl_normal}, the KL term $\KL \left(\mathcal{N}_{w,s}\parallel\mathcal{N}(w_0,\Lambda)\right)$ simplifies to
\begin{equation*}
  -\frac{1}{2}\left(p + \mathbf{1}^T_p \log s\right) + \frac{1}{2} \sum_{j=1}^J \left(\frac{1}{\lambda_j}\|s_{\mathcal{I}_j}\|_1 + \frac{1}{\lambda_j}\|w_{\mathcal{I}_j}-(w_0)_{\mathcal{I}_j}\|^2  + |\mathcal{I}_j| \log \lambda_j \right),
\end{equation*}
where $\log s$ is applied element-wise. 
See Appendix~\ref{proof:gen_thmproof} for the proof.
% \blue{B: mention where the proof appears.}

% \begin{equation*}
%   \frac{1}{2} \sum_{j=1}^J \biggl(\frac{1}{\lambda_j}\|s_{\mathcal{I}_j}\|_1 - |\mathcal{I}_j| + \frac{1}{\lambda_j}\|w_{\mathcal{I}_j}-(w_0)_{\mathcal{I}_j}\|^2  + |\mathcal{I}_j| \log \lambda_j - \mathbf{1}^T_{|\mathcal{I}_j|} \log s \biggr).
% \end{equation*}
% \begin{equation*}
%   \frac{1}{2} \biggl(s \odot \Lambda^{-1}  - p + \|w-w_0\|_{\Lambda^{-1}}^2  + \mathbf{1}^T (\log \Lambda - \log s) \biggr).
% \end{equation*}
% \red{B: I would clarify the notation. Can we write $\diag(S)\Lambda^{-1}$? Also, what is $\diag(\Lambda^{-1})$, its diagonal? What is $\log(\Lambda)$ log det? the elementwise log? let's clarify these things.}
% In \Sec~\ref{sec:opt_pac_bayes}, we will optimize this bound over the decision variables $w$, $s$, and $\lambda$.

% is the bound that we will aim to optimize.

% \begin{equation*}
%   \frac{1}{N}\biggl({\rm KL}(\mathcal{N}_{w, s} \parallel \mathcal{N}_{w_0, \Lambda}) + \log \frac{2 \sqrt{N}}{\delta}\biggr)
% \end{equation*}

\section{Optimizing the generalization bounds}\label{sec:opt_pac_bayes}
% In this section, we describe our method for obtaining non-vacuous generalization guarantees.
In this section, we show how to optimize the PAC-Bayes bounds obtained in \Sec~\ref{sec:gen_l2o}.
In \Sec~\ref{subsec:opt_prior} we present the penalized training problem whose objective aligns with the PAC-Bayes generalization bound.
The objective of this problem includes a KL inverse term which involves solving a one-dimensional convex optimization problem.
In \Sec~\ref{sec:diff_kl_inv} we show how to use implicit differentiation to differentiate through the KL inverse.
This technique allows us to implement a gradient-based learning algorithm, which we present in \Sec~\ref{subsec:learning_algo}.
% In \Sec~\ref{subsec:learning_algo} we present a gradient-based learning algorithm to minimize the training loss of this problem.
In \Sec~\ref{subsec:tighten}, we show how to calibrate the PAC-Bayes bounds after training, thereby providing generalization guarantees on the expected risk.

\subsection{The penalized training problem}\label{subsec:opt_prior}
% \blue{
%   Obstacles
%   \begin{itemize}
%     \item obstacle 1) Prior variance must belong to discrete set -> treat as continuous variable and discretize later
%     \item obstacle 2) Error function is non-differentiable -> use logistic loss
%     \item now present the ``training problem''
%     \item obstacle 3) Practical issue with the ``training problem'' -> need to penalize it
%   \end{itemize}
% }
% In this subsection we define the PAC-Bayes training problem.
% The main idea is that we will optimize over the PAC-Bayes bounds provided in \Thm~\ref{thm:gen_thm}.
% However, there are a few obstacles.
Our overarching strategy to obtain strong generalization guarantees is to use gradient-based methods to optimize the PAC-Bayes bounds from \Thm~\ref{thm:gen_thm}.
Gradient-based methods emerge as a natural choice to optimize our PAC-Bayes bounds as they are often used to train learned optimizers~\citep{l2o,algo_unrolling}.
% are often used to train learned optimizers~\citep{l2o,algo_unrolling} this approach emerges as a natural choice.
% Given that gradient-based methods are often used to train learned optimizers~\citep{l2o,algo_unrolling} this approach emerges as a natural choice.
Indeed, the learned optimizers that we provide generalization guarantees for in our numerical experiments in \Sec~\ref{sec:numerical_experiments} are all trained with gradient-based methods in their original works.
Nevertheless, this approach gives rise to several obstacles that we will address in this section, culminating in the formulation of a \emph{penalized training problem}.
% must be addressed.
% Yet, there are several obstacles that arise because of the choice to use gradient-based methods.
% Yet, there are several obstacles, related to this approach, that prevent us from doing so.
% In this subsection, we enumerate these obstacles and propose solutions for them.
% By the end of this subsection, we will have formulated our \emph{penalized training problem}.

The first obstacle is that the decision variable $\lambda$, which corresponds to the prior variance, must belong to a discrete set as described in \Sec~\ref{sec:gen_l2o}.
To simplify the training, we treat $\lambda$ as a continuous variable, and then after training, round its value to the discrete set~\citep{nonvacuous_pac_bayes}.
The second obstacle is that the 0--1 loss function $e_\theta(x)$ is non-differentiable.
% To address this, we replace $e_\theta(x)$ with the logistic transformation of the original loss $\hat{\ell}_\theta(x)$ which is guaranteed to be in the range $(0,1)$.
To address this, we replace the loss in $e_\theta(x)$ with its logistic transformation  %of the original loss $\surrogate_\theta(x)$ given by
\begin{equation}\label{eq:logistic}
  \surrogate_\theta(x) = \frac{1}{1 + \exp(-\ell_\theta(x))}.
\end{equation}
The transformed loss $\surrogate_\theta(x)$ achieves two desired properties; it is differentiable and lies in the range $(0,1)$ (hence, a good proxy for the error function).
We denote the expected empirical risk of the logistic loss over a distribution $P$ as
\begin{equation*}
  \surrogateexpemprisk(P) =  \mathbf{E}_{\theta \sim P} \frac{1}{N} \sum_{i=1}^N \surrogate_\theta(x_i).
\end{equation*}
% which is guaranteed to be in the range $(0,1)$.
% \paragraph{Logistic transformation.}
% While the final generalization bounds we compute will be for the 0--1 error function given in \Eqn~\eqref{eq:zero_one_loss_l2o}, this loss is non-differentiable which means it is not amenable to gradient-based optimization methods.
% For the training process described in \Sec~\ref{sec:opt_pac_bayes}, we use a logistic transformation of the original loss $\ell_\theta$ given by
% \begin{equation}\label{eq:logistic}
%   \surrogate_\theta(x) = \frac{1}{1 + \exp(\ell_\theta(x))}.
% \end{equation}
% The transformed loss $\surrogate_\theta(x)$ achieves the two desired properties; it is differentiable and lies in the range $(0,1)$.
At this point, the optimization problem can be formulated as
\begin{equation}\label{prob:init_training_prob}
  \begin{array}{ll}
\mbox{minimize} &  \KLinv \left(\surrogateexpemprisk(\mathcal{N}_{w,s}) ~|~ B(w,s,\lambda) \right) \\
\mbox{subject to} & 0 \leq \lambda \leq \lambda^{\rm max} \\
&s \geq 0,
\end{array}
\end{equation}
where the decision variables are $w \in \reals^p$, $s \in \reals_+^p$, and $\lambda \in \reals_+^J$.
In practice, a third and particularly practical obstacle arises with the initial formulation of our optimization problem.
There is an imbalance between the expected empirical risk of the logistic loss $\surrogateexpemprisk(\mathcal{N}_{w,s})$ and the regularizer $B(w,s,\lambda)$.
The regularizer $B(w,s,\lambda)$ can be disproportionally large while the quantity $\surrogateexpemprisk(\mathcal{N}_{w,s})$ is always in the range $(0,1)$.
We observe that in many cases, applying gradient-based methods to solve problem~\eqref{prob:init_training_prob} tends to reduce the regularizer $B(w,s,\lambda)$ to zero, typically by making $w$ close to zero, resulting in suboptimal solutions.
% This means that the weights end up being very close to the prior mean, resulting in suboptimal solutions.
% \blue{could this be explained better?}

\paragraph{The penalized training problem.}
To remedy this problem, we add a penalty term to the objective to penalize the distance between~$B(w,s,\lambda)$ and a hyperparameter~$B^{\rm target} \in \reals_{++}$.
Now, we are ready to define our \emph{penalized training problem}:
\begin{equation}\label{prob:training_prob}
  \begin{array}{ll}
\mbox{minimize} &  \KLinv \left(\surrogateexpemprisk(\mathcal{N}_{w,s}) ~|~ B(w,s,\lambda) \right) + \mu \left(B(w,s,\lambda) - B^{\rm target}\right)^2 \\
\mbox{subject to} & 0 \leq \lambda \leq \lambda^{\rm max} \\
&s \geq 0.
\end{array}
\end{equation}
Here, $\mu \in \reals_{++}$ is a large constant term that weights the penalty.
The value of $B^{\rm target}$ will control the gap between the expected empirical risk and the expected risk.
Specifically, if $B(w,s,\lambda) = B^{\rm target}$, then the expected risk can be upper bounded with $R_\mathcal{X}(\mathcal{N}_{w,s}) \leq \hat{R}_S(\mathcal{N}_{w,s}) + \sqrt{B^{\rm target} / 2}$.
In practice, we cross-validate over values for $B^{\rm target}$, as detailed in Appendix~\ref{subsec:crossval}.

\subsection{Differentiating through the KL inverse}\label{sec:diff_kl_inv}
% We wish to use gradient-based methods to solve the penalized training problem~\eqref{prob:training_prob}.
% In order to do so, we need to compute gradients through the KL inverse $p = \KLinv(q ~|~ c)$ to obtain $\partial p(q)$ and $\partial p(c)$.
In order to use gradient-based methods to solve the penalized training problem~\eqref{prob:training_prob}, we need to compute gradients through the KL inverse $p = \KLinv(q ~|~ c)$.
However, the output $p$ is not an explicit function of the inputs $q$ and $c$.
Rather, $p$ is \emph{implicitly} defined by $q$ and $c$, and is obtained by solving the geometric program~\eqref{prob:kl_inv}.
Previous approaches that use gradient-based methods to minimize a PAC-Bayes bound, such as those employed by \citet{nonvacuous_pac_bayes} and \citet{majumdar2021pac}, sidestep this challenge by applying Pinsker's inequality from \Eqn~\eqref{eq:pinsker} to transform $p$ into an explicit function of $q$ and $c$. 
% \red{However, using Pinsker's inequality can lead to a less precise bound compared to directly solving the KL inverse problem, a discrepancy that is evident in Figure~\ref{fig:pinsker}.}
\reviewChanges{However, using Pinsker's inequality can lead to a less precise bound compared to directly solving the KL inverse problem.}
In contrast to these methods, our approach leverages the technique of implicit differentiation, supported by the implicit function theorem \reviewChanges{\citep[\Thm~1B.1]{implicit_function}}. 
% \citet{reeb2018learning} explicitly write the derivatives as
Given the KL inverse ${\rm kl}^{-1}(q ~|~ c)$, the implicit derivatives can be written as~\citep{reeb2018learning}
\begin{align*}
  \frac{\partial \; {\rm kl}^{-1}(q ~|~ c)}{\partial q} &= \frac{{\rm kl}^{-1}(q ~|~ c)(1 - {\rm kl}^{-1}(q ~|~ c))}{{\rm kl}^{-1}(q ~|~ c) - q} \left(\log \frac{q}{{\rm kl}^{-1}(q ~|~ c)} + \log \frac{1 - {\rm kl}^{-1}(q ~|~ c)}{1 - q}\right) \\
  \frac{\partial \; {\rm kl}^{-1}(q ~|~ c)}{\partial c} &= \frac{{\rm kl}^{-1}(q ~|~ c)(1 - {\rm kl}^{-1}(q ~|~ c))}{{\rm kl}^{-1}(q ~|~ c) - q}.
\end{align*}
% \begin{align*}
%   \partial p^\star(q) &= \frac{p^\star(q,c)(1 - p^\star(q,c))}{p^\star(q,c) - q} \left(\log \frac{q}{p^\star(q,c)} + \log \frac{1 - p^\star(q,c)}{1 - q}\right) \\
%   \partial p^\star(c) &= \frac{p^\star(q,c)(1 - p^\star(q,c))}{p^\star(q,c) - q}.
% \end{align*}
% \blue{B: we should definitely refer and compare to the appendix A in \url{https://arxiv.org/pdf/1810.12263}}
% \begin{theorem}\label{thm:implicit_function}
%   \citep{implicit_function}.
%   Let $F: \reals^{n \times d} \rightarrow \reals^n$ be a continuously differentiable function.
%   Then if the Jacobian $\partial F$ evaluated at $(z, x)$ is a square invertible matrix, then there exists a function $z^\star(\cdot)$ defined on a neighborhood $x$ such that $z^\star(x) = z$.
%   Furthermore, for all $x$ in this neighborhood, $F(z^\star(x), x) = 0$ and $\partial z^\star(x)$ exists.
%   The derivative can be calculated as follows by using the chain rule:
%   \begin{equation*}
%     \partial z^\star(x) = - \left(\partial_1 F(z^\star(x), x)\right)^{-1} \partial_2 F(z^\star(x), x).
%   \end{equation*}
% \end{theorem}

\paragraph{Smoothness of the KL inverse.}
% In order to apply \Thm~\ref{thm:implicit_function}, the Jacobian $\partial_1 F(z^\star(x), x)$ must be non-singular.
We note that for $q \in (0, 1)$ and $c \in \reals_{++}$, the KL inverse ${\rm kl}^{-1}(q ~|~ c)$ lies in the range $(q, 1)$~\citep[Appendix A]{reeb2018learning}.
Under these conditions, these derivatives exist, \ie, ${\rm kl}^{-1}(q ~|~ c)$ is differentiable with respect to both $q$ and $c$.
% Moreover, we bound the magnitude of the implicit derivatives $|\partial p^\star(q)|$ and $|\partial p^\star(c)|$.
% We show this result in the following theorem.
% \begin{theorem}\label{thm:KLinv}
%   Let $p^\star(q,c) = \KLinv(q ~|~ c)$ where $q \in (0,1)$ and $c \in \reals_{++}$.
%   Then $p^\star(q,c)$ is differentiable with respect to both $q$ and $c$.
% \end{theorem}
% See Appendix~\ref{proof:KLinvproof} for the proof and for the explicit form of the derivatives.
Since we use the logistic loss from~\Eqn~\eqref{eq:logistic}, the value for $q$ is always in the range $(0,1)$ (as opposed to taking a value strictly in $\{0,1\}$).
Additionally, the regularizer $c=B(w,s,\lambda)$ is always strictly positive.
Hence our implicit layer is always differentiable and thus amenable to gradient-based optimization methods.

\subsection{PAC-Bayes learning algorithm}\label{subsec:learning_algo}
% In this subsection we present a learning algorithm based on stochastic gradient descent designed to solve the training problem~\eqref{prob:training_prob}.
In this subsection we present a learning algorithm based on gradient descent to solve the penalized training problem~\eqref{prob:training_prob}.
We cannot apply vanilla gradient descent to solve problem~\eqref{prob:training_prob} yet because we cannot compute the expected empirical logistic risk $\hat{R}_S^{\rm logistic}(\mathcal{N}_{w,s})$ nor its gradients efficiently.
% solve problem~\eqref{prob:training_prob} using vanilla gradient descent yet because we cannot compute the expected empirical logistic risk $\hat{R}_S^{\rm logistic}(\mathcal{N}_{w,s})$ nor its gradients efficiently.
We can, however, compute the gradient of its unbiased estimate $(1 / N) \sum_{i=1}^N \surrogate_{w'}(x_i)$, where $w' = w + \xi \odot \sqrt{s}$ for $\xi \sim \mathcal{N}(0, I_p)$.
% Instead, we use the unbiased estimate of $\mathcal{N}_{w,s}$,
% \begin{equation*}
%   w' = w + \xi \odot \sqrt{s},
% \end{equation*}
% where $\xi$ is drawn i.i.d. from the distribution $\mathcal{N}(0, I_p)$ in each gradient step.
In each iteration we take an i.i.d. copy of $\xi$ and a step in the direction of the negative gradient of the function
\begin{equation*}
  C_S(w,s,\lambda,w') = {\rm kl}^{-1}\left(\surrogateexpemprisk(w') ~|~ B(w,s,\lambda)\right) + \mu (B(w,s,\lambda) - B^{\rm target})^2.
\end{equation*}
% The training procedure is given by Algorithm~\ref{alg:learning_algo}.
% In practice, we mini-batches during the training process.
To ensure the non-negativity of the variable $s$ and $\lambda$, we optimize over variables $\zeta \in \reals^d$ and $\eta \in \reals^J$, and set $s = \exp(\zeta)$ and $\lambda = \exp(\eta)$.
% \blue{move prior rounding to here.}
As mentioned in \Sec~\ref{sec:gen_l2o}, after the gradient descent algorithm terminates, we must round the prior $\lambda$ to fit into the pre-determined grid.
To do so, we compute $a^\star =  {\bf round}(b \log (\lambda^{\rm max} / \lambda))$ and then $\lambda^\star = \lambda^{\rm max} \exp(- a^\star / b)$.
We summarize this discretization via the function
\begin{equation}\label{eq:round_prior}
  \textbf{roundPrior}(\lambda, \lambda^{\rm max}, b) = \lambda^{\rm max} \exp \left(\frac{-\textbf{round} \left(b \log (\lambda^{\rm max} / \lambda) \right)}{b} \right),
\end{equation}
and set the rounded prior with $\lambda^\star = \textbf{roundPrior}(\lambda, \lambda^{\rm max}, b)$.
% \blue{B: This looks great, I would just make a minor change and define an actual function ${\bf roundPrior}(...)$ so that everything is clear.}

\begin{algorithm}[tb]
  \caption{PAC-Bayes Learning to solve problem~\eqref{prob:training_prob}}
  \label{alg:learning_algo}
\begin{algorithmic}[1]
  \State {\bfseries Inputs:} \\
  Target penalty: $B^{\rm target} \in \reals_{++}$\\
  Prior hyperparameters: $\lambda^{\rm max} \in \reals_{++}$, $b \in \reals_{++}$ \\
  Initial weights: $w_0 \in \reals^p$, $s_0 \in \reals^p_+$, $\lambda_0 \in (0,\lambda^{\rm max})^J$ \Comment{Random initialization}\\
  Desired probability: $\delta \in (0,1)$ \\
  Learning rate: $\gamma \in \reals_{++}$ \\
  Number of epochs: $M \in \mathbf{N}$
  \State {\bfseries Procedure:}
  \State $(w,\zeta,\nu) = (w_0,\log(s_0),\log(\lambda_0))$
  \For{$i=1$ {\bfseries to} $M$} \Comment{Loop over epochs}
    % \State $s = \exp(\zeta)$
    % \State $\lambda = \exp(\nu)$
    \State sample $\xi \sim \mathcal{N}(0,I_p)$
    \State $w' = w + \xi \odot \sqrt{\exp(\zeta)}$
    \Comment{Sample from $\mathcal{N}_{w,s}$}
    % \State $w = w - \gamma \nabla_w C_S(w,\exp(\zeta),\exp(\nu),w')$
    % \State $\zeta = \zeta - \gamma \nabla_\zeta C_S(w,\exp(\zeta),\exp(\nu),w')$ 
    % \State $\nu = \nu - \gamma \nabla_\nu C_S(w,\exp(\zeta),\exp(\nu),w')$
    \State $\displaystyle \begin{bmatrix}
      w \\
      \zeta \\
      \nu
    \end{bmatrix} = \begin{bmatrix}
      w \\
      \zeta \\
      \nu
    \end{bmatrix} - \gamma
    \begin{bmatrix}
      \nabla_w C_S(w,\exp(\zeta),\exp(\nu),w') \\
      \nabla_\zeta C_S(w,\exp(\zeta),\exp(\nu),w') \\
      \nabla_\nu C_S(w,\exp(\zeta),\exp(\nu),w')
    \end{bmatrix} $ \Comment{Gradient step}  
    \EndFor 
    \State ($w^\star, s^\star, \lambda^\star) = (\exp(\zeta), \exp(\nu), \textbf{roundPrior}(\exp(\nu), \lambda^{\rm max}, b))$ \Comment{Round prior:~\Eqnshort~\eqref{eq:round_prior}}
  \State {\bfseries Outputs:} \\
  Learned weights $(w^\star$, $s^\star$, $\lambda^\star)$
\end{algorithmic}
\end{algorithm}

\subsection{Calibrating the PAC-Bayes bounds}\label{subsec:tighten}
% In this subsection, we take the trained weights from \Sec~\ref{subsec:learning_algo} and calibrate the bounds for a given metric $\phi$, number of algorithm steps $k$, and tolerance $\epsilon$.
% There are several steps needed to obtain the final generalization bounds after the training procedure terminates.
% After the training terminates, we denote the learned posterior distribution $P = \mathcal{N}_{w^\star,s^\star}$ where $w^\star$ and $s^\star$ are the posterior mean and variance respectively and the prior variance $\lambda^\star$.
The training procedure returns the learned weights: the posterior mean $w^\star$, the posterior variance $s^\star$, and the prior variance $\lambda^\star$.
Together, these determine the posterior distribution $P = \mathcal{N}_{w^\star,s^\star}$ and the regularizer $B(w^\star, s^\star, \lambda^\star)$.
To obtain the final generalization bounds after the training procedure terminates, we need to \emph{calibrate} the PAC-Bayes bounds for a given metric $\phi$, number of algorithm steps $k$, and tolerance $\epsilon$.

Conceptually, we would like to apply the McAllester bound to bound the expected risk $\exprisk(P)$ in terms of the expected empirical risk $\expemprisk(P)$.
However, this is not immediately possible since evaluating the expected empirical risk $\expemprisk(P)$ is intractable.
% , we must first bound it, and we do this via inequality~\eqref{eq:langford_bound}.
To circumvent this issue, we generate $\hat{P}$ a Monte Carlo approximation of $P$, compute the Monte Carlo estimate of the expected empirical risk $\hat{R}_S(\hat{P})$, and bound the expected empirical risk $\expemprisk(P)$ using inequality~\eqref{eq:langford_bound}.
We fully detail and enumerate the steps needed to calibrate the bounds below.

First, we draw $H$ i.i.d.\ samples, denoted by $\{\theta_i\}_{i=1}^H$, from the posterior distribution~$P$. 
We then construct the Monte Carlo approximation $\hat{P} = (1/H) \sum_{j=1}^H \delta_{\theta_j}$, where $\delta_{\theta_j}$ represents the Dirac delta function centered at $\theta_j$.
We then run $k$ steps of the learned optimizer for each of the $H$ samples for each of the $N$ training problems.
Second, we compute the Monte Carlo approximation of the expected empirical risk $\hat{R}_S(P)$ as follows:
\begin{equation}\label{eq:sample_avg}
  \hat{R}_S(\hat{P}) = \frac{1}{NH} \sum_{i=1}^N \sum_{j=1}^H e_{\theta_j}(x_i),
\end{equation}
where the error function $e$ is based on the underlying metric $\phi$, number of steps $k$, and tolerance $\epsilon$.
% Second, since evaluating the expected empirical risk $\expemprisk(P)$ is intractable, we approximate it with unbiased estimates and apply \Eqn~\eqref{eq:langford_bound}.
% We draw $H$ i.i.d. samples, denoted by $\{\theta_i\}_{i=1}^H$, from the distribution $P$. 
% Utilizing these samples, we construct the Monte Carlo approximation $\hat{P}$, defined as $\hat{P} = (1/H) \sum_{i=1}^H \delta_{\theta_i}$, where $\delta_{\theta_i}$ represents the Dirac delta function centered at $\theta_i$.
% The Monte Carlo approximation $\hat{P}$ can be computed as
% \begin{equation}\label{eq:sample_avg}
%   \hat{R}_S(\hat{P}) = \frac{1}{NH} \sum_{i=1}^N \sum_{j=1}^H e(\theta_j, x_i).
% \end{equation}
% Third, we treated $\lambda$ as a continuous optimization variable during training, and the union bound argument requires that $\lambda$ take the form $\lambda^{\rm max}\exp(-a / b)$ for some $a \in \mathbf{N}^J$.
% Hence, we round round $\lambda$  to the nearest value that satisfies this condition in the following way:
% \begin{align*}
%   & a^\star = \text{round} \left(b \log \frac{\lambda^{\rm max}}{\bar{\lambda}} \right), \quad
%   \lambda^\star = \lambda^{\rm max} \exp \left(\frac{-a^\star}{b} \right).
% \end{align*}
Last, we apply two PAC-Bayes bounds to obtain the final bounds on the expected risk.
Using the sample convergence bound from \Thm~\ref{thm:sample_conv_bound}, the following inequality holds with probability at least $1 - \omega$ for $\omega \in (0,1)$:
\begin{equation}\label{eq:e_bar}
  \expemprisk(P) \leq \bar{R}_{S}(P)= \KLinv \left(\hat{R}_S(\hat{P}) \bigline \frac{1}{H} \log \frac{2}{\omega}\right).
\end{equation}
We then apply a union bound and our \reviewChanges{\Thm~\ref{thm:gen_thm}} to obtain the final bound on the expected risk
\begin{equation}\label{eq:final_bound}
  \exprisk(P) \leq R^\star_S(P) = {\rm kl}^{-1}(\bar{R}_{S}(P) ~|~ B(w^\star,s^\star,\lambda^\star)),
\end{equation}
which holds with probability $1 - \delta - \omega$.

% The procedure to obtain generalization guarantees on the expected risk is outlined in Algorithm~\ref{alg:calibrating}.
We outline this calibration procedure in Algorithm~\ref{alg:calibrating}, and we also depict the entire process to obtain generalization bounds for learned optimizers, including the training and calibration phases, in Figure~\ref{fig:l2o_procedure}.
We \reviewChanges{note} that the number of steps $k$ that the bounds are computed for need not be the same as the number of steps $K$ that are used to train the weights.
Moreover, the metric $\phi$ does not need to be the same as the metric used in the loss function.
\reviewChanges{We remark that some of choices made to facilitate training (\eg, using the logistic regression loss instead of the 0--1 loss) do not affect the validity of our bounds.
This is because our main generalization bound is applied \emph{after} training is complete.
Furthermore, by applying union bounds appropriately (\ie, for the prior variance, the hyperparameter $B^{\rm target}$, and the Monte Carlo approximation of the expected empirical risk), we ensure the validity of the final bound.
}
Indeed, in the numerical experiments in \Sec~\ref{sec:numerical_experiments}, we will calibrate the bounds for many different tolerances and algorithm steps, and sometimes, multiple metrics.

\begin{algorithm}[tb]
  \caption{Calibrating the PAC-Bayes bounds}
  \label{alg:calibrating}
\begin{algorithmic}[1]
  \State {\bfseries Inputs:} \\
  Learned weights: $w^\star$, $s^\star$, $\lambda^\star$ 
  \Comment{Output of Algorithm~\ref{alg:learning_algo}}\\
  Desired probabilities: $\delta, \omega \in (0,1)$ \\
  Metric: $\phi$ \\
  Number of algorithm steps: $k$\\
  Desired tolerance: $\epsilon$\\
  Number of samples: $H$ \Comment{For Monte Carlo approx.}
  \State {\bfseries Procedure:}
  % \State $a^\star = \text{round}(b \log (\lambda^{\rm max} / \bar{\lambda}))$
  % \State $\lambda^\star = \lambda^{\rm max} \exp(-a^\star / b)$ \Comment{Rounding the prior}
  % \State $B^\star = B(w^\star,s^\star,\lambda^\star,\delta)$
  \State Generate $H$ samples $\{\theta_j\}_{j=1}^H$ from $\mathcal{N}_{w^\star,s^\star}$ \Comment{Monte Carlo samples}
      \State $\hat{R} = (1/(NH)) \sum_{j=1}^H  \sum_{i=1}^N e_{\theta_j}(x_i) $ \Comment{Empirical est. \Eqnshort~\eqref{eq:sample_avg} (metric $\phi$, $k$ steps, tol. $\epsilon$)}
      \State $\bar{R} = \KLinv(\hat{R} ~|~ (1/H) \log (2 / \omega))$ \Comment{Sample convergence bound: \Eqnshort~\eqref{eq:e_bar}}
      \State $R^\star = \KLinv(\bar{R} ~|~ B(w^\star,s^\star,\lambda^\star))$ \Comment{\reviewChanges{main Thm.~\ref{thm:gen_thm} bound}: \Eqnshort~\eqref{eq:final_bound}}
  \State {\bfseries Outputs:}\\
   $R^\star$ \Comment{The final bound on the expected risk}
\end{algorithmic}
\end{algorithm}
\begin{figure}[!h]
  \centering
  \includegraphics[width=1.0\linewidth]{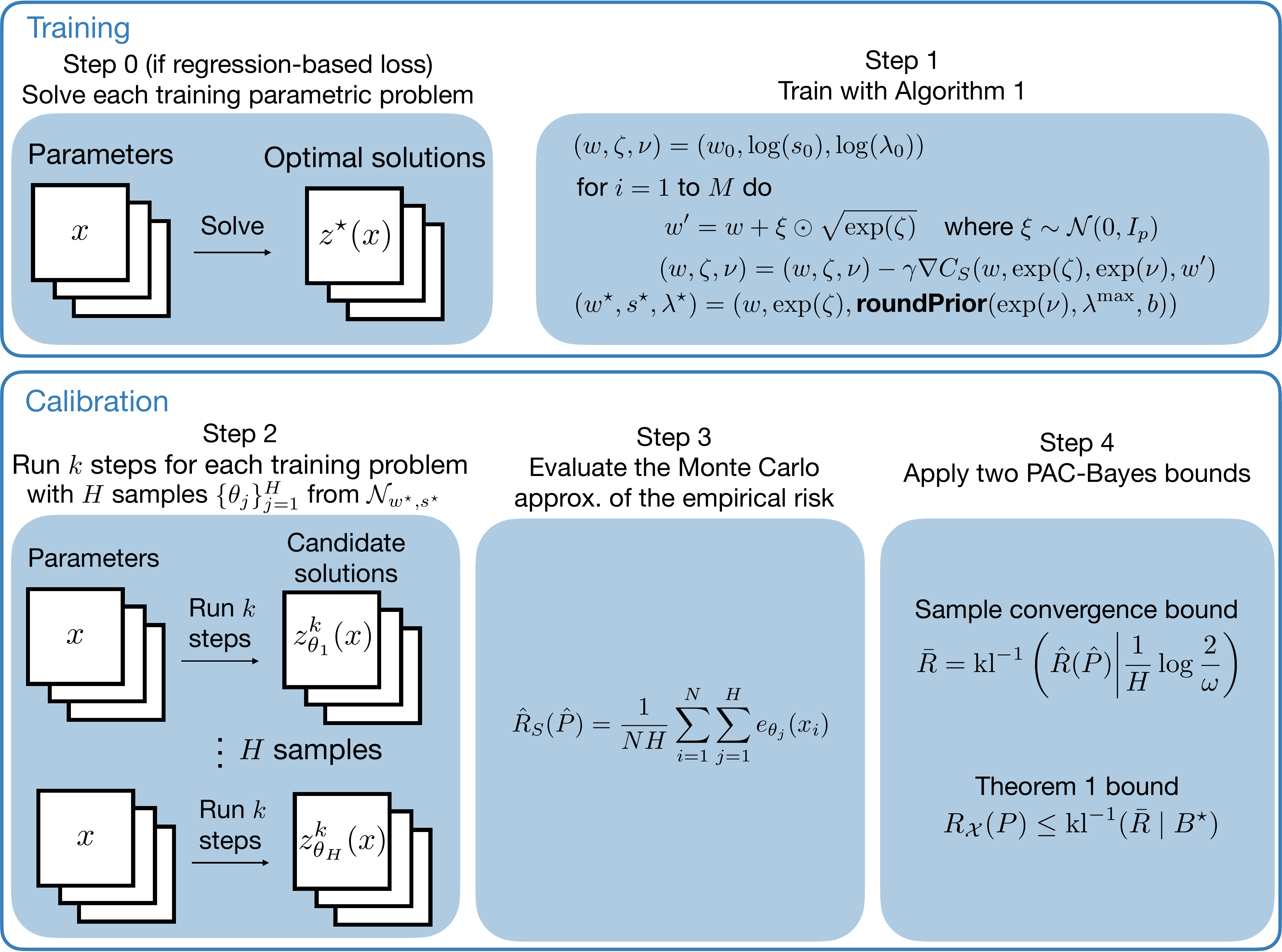}
\centering
\caption{
The two-phase procedure to generate generalization guarantees for learned optimizers for a metric $\phi$, number of algorithm steps $k$, and tolerance $\epsilon$.
The first phase is the training phase.
If the loss function is the regression-based loss, then we solve each parametric problem in step $0$ as these are needed in order to train.
In step $1$, we train the architecture to optimize the PAC-Bayes guarantee over $M$ epochs using Algorithm~\ref{alg:learning_algo}.
We also round the prior according to \Eqn~\eqref{eq:round_prior}.
Then we enter the second, calibration phase.
In step $2$ we sample weights $\{\theta_j\}_{j=1}^H$ from the distribution $\mathcal{N}_{w^\star, s^\star}$ and run $k$ algorithm steps for each training problem and each weight sample $\theta_j$.
In step $3$ we compute the Monte Carlo approximation of the empirical expected risk $\hat{R}_S(\hat{P})$.
In step $4$, we bound the expected risk $\exprisk(P)$ by applying a sample convergence bound from \Eqn~\eqref{eq:e_bar} and then \reviewChanges{\Thm~\ref{thm:gen_thm}} where the regularization term is $B^\star = B(w^\star,s^\star,\lambda^\star)$.
}
\label{fig:l2o_procedure}
\end{figure}
% We also detail the entire procedure including both the training and finalization steps in Algorithm~\ref{alg:train_and_finalize}.

% \red{
% \subsection{Classical optimization generalization bounds}
% }

\section{Experiments}\label{sec:numerical_experiments}
% \red{rewrite}
% We optimize over the prior and posterior variances as described in~\ref{subsec:opt_prior_posterior} to optimize over the loss for a given epsilon.
% Then given these values, we compute the convergence guarantees.
In this section, we illustrate the effectiveness of our guarantees for both classical and learned optimizers with numerical experiments.
The code to reproduce our results is available at
\begin{equation*}
\text{\url{https://github.com/stellatogrp/data\_driven\_optimizer\_guarantees}}.
\end{equation*}
In \Sec~\ref{subsec:experiments_classical} we apply our framework from \Sec~\ref{sec:classical} to provide guarantees for classical fixed-point optimization algorithms in the context of parametric optimization.
In \Sec~\ref{subsec:experiments_learned} we apply our training algorithm and generalization guarantees from \Sec~\ref{sec:opt_pac_bayes} to obtain strong bounds for a variety of learned optimizers.
% \blue{B: I prefer the following case (since URLs do not distinguish upper and lower case): \texttt{data\_driven\_optimizer\_guarantees}}

% \blue{B: The repo is not clean at the moment. We should remove folder \texttt{l2ws} and have clear files/folders for each example. I feel we are copy-pasting a lot from the previous project without restructuring the code}
% \red{risk and expected risk vs risk}
% In both cases, we provide bounds on the \emph{quantiles} on the underlying metric.
% See \Sec~\ref{sec:quantiles} for more details on how we construct the quantiles.
% Our approach can provide guarantees on any metric as long as it can be evaluated and when applicable we report results for task-specific metrics.

% Finally, where applicable, we repeat the results, both the bounds on the risk and the quantile, for task-specific metrics instead of the fixed-point residual.
% Specifically we provide bounds for image deblurring and robust Kalman filtering in Sections~\ref{subsubsec:iamge_deblurring} and ~\ref{subsubsec:rkf} respectively.
% We fix several tolerances and illustrate our bounds on the risk $\risk$ by plotting the quantity $1 - \risk$ over a range of evaluation steps.
% We plot the empirical quantity and the guarantees provided from worst-case analysis.
% We then plot bounds on the $30$th, $90$th, and $99$th quantiles of the fixed-point residual for any number of evaluation steps.
% We plot the upper bounds on the risk $\risk$ with lower bounds on $1 - \risk$ which is equivalent to the probability that a new problem with parameter $x$ drawn from distribution $\mathcal{X}$.

\subsection{Guarantees for classical parametric optimization}\label{subsec:experiments_classical}
\reviewChanges{In this subsection, we apply our method to obtain generalization guarantees to image deblurring in \Sec~\ref{subsubsec:iamge_deblurring} and robust Kalman filtering in~\Sec~\ref{subsubsec:rkf}.}
% \red{In this subsection, we apply our method to obtain generalization guarantees to image deblurring in \Sec~\ref{subsubsec:iamge_deblurring}, robust Kalman filtering in~\Sec~\ref{subsubsec:rkf}, and quadcopter control in \Sec~\ref{subsec:quadcopter}.}
We focus on solving convex QPs and convex conic programs, for which we use the Operator Splitting Quadratic Program (OSQP) solver~\citep{osqp} and the Splitting Conic Solver (SCS)~\citep{scs_quadratic} respectively as the fixed-point algorithm.
The fixed-point vector $z$ consists of both primal and dual variables; see~\citep[Table 1]{l2ws} for more details on how this vector is constructed for OSQP and SCS.
In each of the examples, we vary the number of algorithm steps and tolerances and report a lower bound on the success rate $1 - \risk$.
Then, by combining these bounds for the risk across many tolerances, we construct upper \emph{quantile bounds} on the fixed-point residual at each algorithm step and compare them against the empirical quantile performance.
See \Sec~\ref{sec:quantiles} for more details on how we construct the quantiles.
We show that our probabilistic guarantees are much tighter than bounds given through worst-case theoretical analysis.
Finally, where relevant, we repeat our analysis to include task-specific metrics instead of the fixed-point residual, again providing risk and quantile bounds.
The probabilistic results hold with probability at least $0.9999$ for the risk and with probability at least $0.9919$ for the quantiles.
See Appendix \Sec~\ref{sec:prob_details} for more details.
We provide guarantees for $10$, $100$, and $1000$ samples in each example.
% In this subsection, we apply the guarantees for classical parametric optimization from \Sec~\ref{sec:classical}.
% When the metric of interest is the fixed-point residual, we showcase the effectiveness of our guarantees that hold with high probability by comparing them against our best estimate of the worst-case guarantee given by theoretical convergence analysis.

\paragraph{Worst-case guarantees.}
% We compare our guarantees that hold with high probability against worst-case guarantees given by classical convergence analysis.
The worst-case guarantee is determined by our best estimate; although, we remark that a better numerical result may be possible to obtain.
Indeed, it can be difficult to check when certain conditions are met to guarantee a given convergence rate (\eg, linear convergence for ADMM)~\citep{discern_lin_rate}.
To generate our best estimate of the worst-case guarantee, we proceed as follows.
We first estimate a value for $\dist_{\fix T_x}(z^0)$ (where $z^0 = 0$ is the initial point) by taking the largest optimal solution in terms of its $2$-norm across problem instances and multiplying it by $1.1$.
Then we generate the worst-case guarantees based on the property of the fixed-point operator given in the rates from~\eqref{eq:lin_conv_rate} and ~\eqref{eq:avg_rate}.
Both OSQP and SCS are algorithms based on Douglas-Rachford splitting~\citep{infeas_detection,scs_quadratic}.
For both algorithms, we pick hyperparameters so that the fixed-point algorithm is $(1/2)$-averaged~\citep{infeas_detection,scs_quadratic}.
For OSQP, we set the penalty and relaxation parameters to be one; see~\citet{infeas_detection} for more details.
For SCS, we enforce identity scaling and set the relaxation parameter to be one; see~\citet{scs_quadratic} for more details.
Hence, the sublinear rate with $\alpha=1 / 2$ from~\eqref{eq:avg_rate} holds as a worst-case guarantee in all three instances.
This rate can be improved upon if additional conditions (\eg, strong convexity) are satisfied.
We verified our theoretical bounds with PEP~\citep{pep} using the PEPit toolbox~\citep{pepit}.
However, that approach does not scale well to more than $100$ number of iterations because, in contrast to ours, it requires solving a semidefinite \reviewChanges{program}~\citep{pep}.
It takes nearly $59$ minutes to provide guarantees for $70$ iterations for algorithms where the iterations are averaged, for which the guarantees are within $6 \%$ of the theoretical bound.
Because of the lack of scalability and how close the PEP bounds are to the theoretical guarantees, we omit them in our tables and plots.
\reviewChanges{We emphasize that our probabilistic analysis is not meant to replace worst-case analyses; rather, it is meant to offer complementary insights.
The comparisons illustrate the significant gap between worst-case bounds and the behavior observed on average over a parametric problem family.
}
% Moreover, the bounds computed over these $10$ iterations are nearly identical to the theoretical bounds, so we do not report the results.
% \red{fix this.}
% In particular, for OSQP, we set the penalty parameters and the relaxation parameters to be one.
% For SCS we avoid the non-identity scaling that is implemented in the solvers default parameters.
% In OSQP there are penalty and relaxation parameters.
% We pick all of these factors to be one which ensures that the fixed-point algorithm is $(1/2)$-averaged.
% Similarly in SCS, we ensure that the fixed-point operator is $(1/2)$-averaged by taking 

% In particular, they are $(1 / 2)$-averaged algorithms and.
% We set the penalty parameters in ADMM to be $1$ and use identity-scaled for SCS.
% We make these choices so that we can fairly compare our probabilistic guarantees against worst-case guarantees.
% Notably, we remark that our guarantees can work if different choices were made.
% But for the worst-case guarantees.
% For instance, if different penalty values are used in OSQP, then the worst-case result is in terms of a weighted norm~\citep{infeas_detection}.

% \blue{B: Check why subsubsection is in smallcaps}
\subsubsection{Image deblurring}\label{subsubsec:iamge_deblurring}
% \blue{B: I would use the same letter convention and use $z$ for the solution instead of $y$.}
The first task we consider is image deblurring.
Given a blurry image $x \in \reals^n$, the goal is to recover the original image $y \in \reals^n$.
% Both the noisy vector $b$ and the target vector $x$ are formed by stacking the columns of their respective images.
The vectors $b$ and $y$ are created by stacking the columns of the matrix representations of their images.
We formulate the image deblurring problem as the QP
\begin{equation*}
  \begin{array}{ll}
  \label{prob:img_deblur}
  \mbox{minimize} & \|A y - x\|_2^2 + \rho \|y\|_1 \\
  \mbox{subject to} & 0 \leq y \leq 1,
  \end{array}
\end{equation*}
where $y \in \reals^n$ is the decision variable.
% As this problem is a quadratic program, we solve it with OSQP.
In this problem, the matrix $A \in \reals^{n \times n}$ functions as a Gaussian blur operator embodying a two-dimensional convolutional operator.
The regularizer coefficient $\rho \in \reals_{++}$,  balances the importance of the fidelity term $\|A y - x\|_2^2$, relative to the $\ell_1$ penalty.
% We solve this problem with OSQP~\citep{osqp}.
% We use a Gaussian blue of size $8$ and add a noise with standard deviation $0.001$ to each pixel.
% We take $\rho = 10^{-4}$.
% See~\citep{l2ws} for details on the construction of the fixed-point operator 
% \red{B: Figure~\ref{fig:classical} is really cool. Do you have a way to compare this to worst-case analysis? (\eg, PEP problems). It would be great to visualize how much tighter you get with this data-driven approach}
% \begin{figure}[!h]
%   \begin{subfigure}[t]{.33\linewidth}
%     \centering
%     \includegraphics[width=\linewidth]{figures/classical/mnist/acc_0.1.pdf}
%   \end{subfigure}%
%   \begin{subfigure}[t]{.33\linewidth}
%     \centering
%     \includegraphics[width=\linewidth]{figures/classical/mnist/acc_0.01.pdf}
%   \end{subfigure}%
%   \begin{subfigure}[t]{.33\linewidth}
%     \centering
%     \includegraphics[width=\linewidth]{figures/classical/mnist/acc_0.001.pdf}
%   \end{subfigure}%
%   \centering
%   \classicallegend
%   \caption{Guarantees for OSQP to solve the image deblurring problem. 
%   As the number of samples gets larger, the expected number of problems that reach a given tolerance for the fixed-point residual (fp res) within a budget of iterations decreases.
%   Moreover, operator theory gives a guarantee that with enough iterations, all of the problems will reach the desired accuracy.}
%   \label{fig:classical_image}
% \end{figure}
% \blue{B: I would clearly define the acronym here otherwise it is not clear what NMSE means in the tables later. You could also define the function ${\bf NMSE}(z) = ...$}
\paragraph{Task-specific metric.}
In signal recovery tasks, it is common to report the normalized mean squared error (NMSE) in decibel (dB) units~\citep{l2o} between the $z$ and the original signal $\tilde{z}$ given by
\begin{equation}\label{eq:nmse}
  {\rm NMSE}(z,\bar{z}) = 10 \log_{10} \frac{\|z - \tilde{z}\|_2^2}{\|\tilde{z}\|_2^2}.
\end{equation}
% \blue{B: Is $z^\star(x)$ the optimal solution or the original not noisy image?}
% Note that this metric is a normalized version of the regression loss from \Eqn~\eqref{eq:reg_loss} and we only report results on the NMSE for this example.
% In signal processing tasks, the normalized mean squared error is a common measure the fidelity of a recovered signal.

% In the context of image deblurring, traditional metrics such as the fixed-point residual, which encompasses both primal and dual variables, may not fully capture specific aspects of image recovery accuracy.
% Consequently, we propose a task-specific metric aimed at quantifying the fidelity of the image recovery.
% To aid in visualizing our results
% \red{add the worst-case}
% \red{pep (theory) worst-case}

% \begin{figure}[!h]
%   \centering
%     \includegraphics[width=\figsize\linewidth]{figures/classical/acc_0.1.pdf}
%     \classicallegend
%     \\
%     \caption{Classical bounds for parametric image deblurring.
%       Using only the worst-case analysis from \Sec~\ref{subsec:worst_case}, we would get that it would take about $10^5$ OSQP steps to reach a fixed-point residual of $0.1$.
%       Using our guarantees from \Sec~\ref{sec:classical}, we get much tighter guarantees.
%       The bounds get tighter with more samples used.
%     }
%     \label{fig:classical}
% \end{figure}

\paragraph{Numerical example.}
We consider handwritten letters from the EMNIST dataset~\citep{emnist}.
We apply a Gaussian blur of size $8$ to each letter and then add i.i.d. Gaussian noise with standard deviation $0.001$.
The hyperparameter weighting term is $\rho = 10^{-4}$.
% \blue{B: Is this $\rho$ or $\lambda$?}
% \blue{B: We need to mention exactly how many esamples we use}

\paragraph{Results.}
Figure~\ref{fig:mnist_fp} shows our results.
In this case, the objective is strongly convex, which can be used to guarantee linear convergence~\citep{Giselsson2014LinearCA}. 
Therefore, we calculate the most optimistic linear convergence factor possible based on the performance in the samples and combine it with \eqref{eq:avg_rate} to estimate the worst-case guarantee.
\begin{figure}[!h]
  \begin{subfigure}[t]{.99\linewidth}
      \centering
      \includegraphics[width=\linewidth]{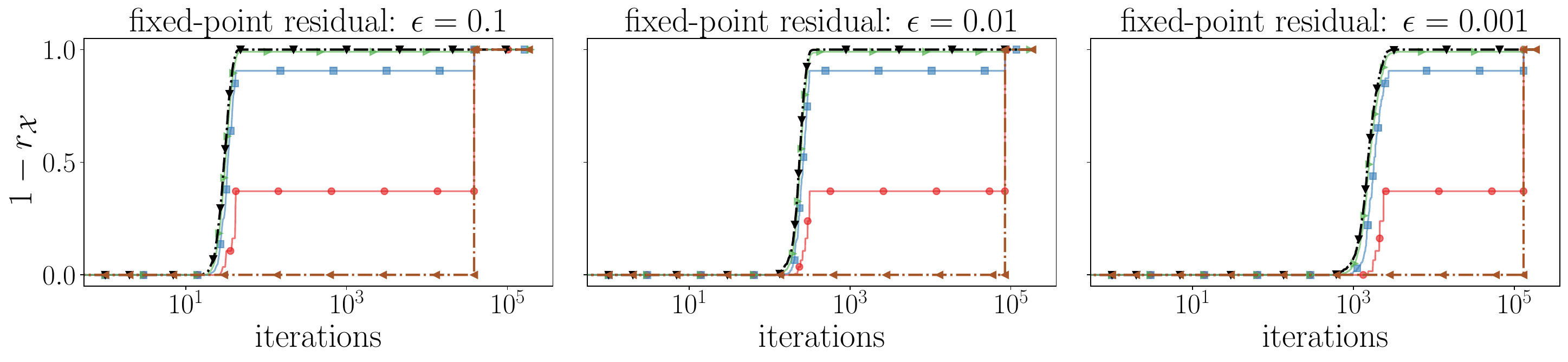}
    \end{subfigure}\\
    \begin{subfigure}[t]{.99\linewidth}
      \centering
      \includegraphics[width=\linewidth]{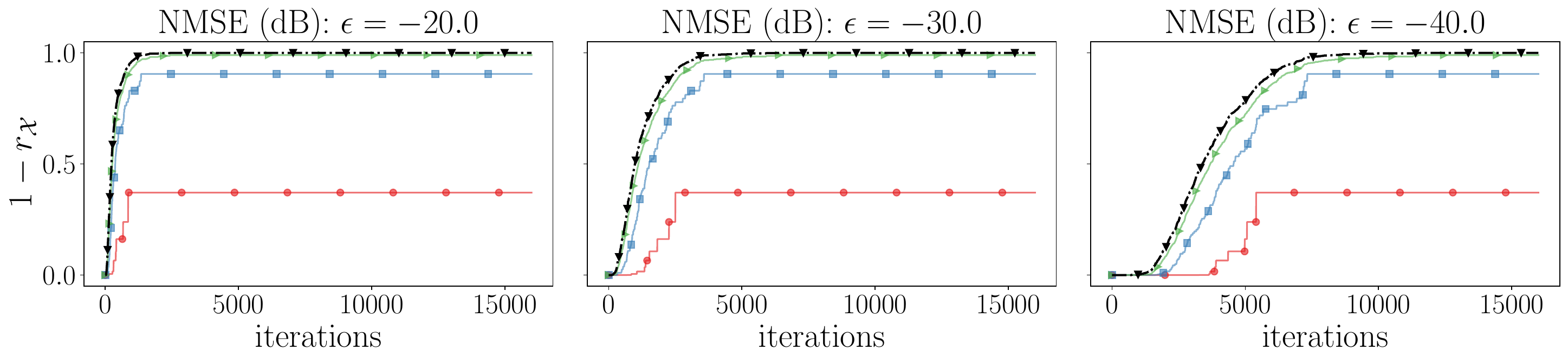}
    \end{subfigure}
  \centering
  \classicallegend
  \caption{
    Probabilistic lower bounds of the success rate for image deblurring.
    The top row shows results for the fixed-point residual (fp. res.) and the bottom row shows bounds for the quantile.
    For both metrics, the lower bounds on the success rate are tight for $N=1000$ samples.
    % The black curve is the empirical average of problems that reach the tolerance with a given number of algorithm steps.
    % With probability at least $0.99$, the red, blue, and green curves lower bound $1 - \risk$ provided $10$, $100$, and $1000$ samples respectively.
    % Moreover, operator theory gives a guarantee that with enough iterations, all of the problems will reach the a desired fixed-point residual.
    }
    \label{fig:mnist_fp}
\end{figure}

\begin{table}[!h]
  \centering
  \footnotesize
    \renewcommand*{\arraystretch}{1.0}
  \caption{The quantile results for image deblurring. 
  Left: fixed-point residual.
  Right: NMSE.
  We report the number of iterations to reach given tolerances.
  For different quantiles (Qtl.) and tolerances (Tol.), we compare the empirical (Emp.) and estimated worst-case quantities against our probabilistic bounds with a varying number of samples $N$.
  The worst-case bound holds independently of the quantile.
  % For different quantiles (Qtl.) and tolerances (Tol.), we compare the cold start (CS) and nearest neighbor (NN) empirical performances against our bounds (Bnd.). \red{fix this}
  }
  \label{tab:mnist_tab}
  \vspace*{-3mm}
  \adjustbox{max width=\textwidth}{
    \begin{tabular}{ccccccc}
      \midrule
      \multicolumn{7}{c}{Fixed-point residual}\\
      \midrule
      Qtl.&
    Tol.&
    Worst-&
    Emp.&
    \multicolumn{3}{c}{Bound}
    \\
    {} & {} & Case &{} & $N=10$ & $N=100$ & $N=1000$\\
    \midrule
    \begin{filecontents*}{./data/mnist/mnist_fp.csv}
quantile,tol,cold_start,worst,samples_10,samples_100,samples_1000
30,0.01,217,80932,383,289,270
,0.001,1342,95975,2903,1938,1674
,0.0001,8343,108779,15765,11597,10512
90,0.01,285,80932,-,383,362
,0.001,2166,95975,-,3274,2816
,0.0001,12415,108779,-,17018,15161
99,0.01,320,80932,-,-,508
,0.001,2615,95975,-,-,4046
,0.0001,14523,108779,-,-,17283

\end{filecontents*}
\csvreader[head to column names, late after line=\\]{./data/mnist/mnist_fp.csv}{
      quantile=\colQ,
    tol=\colA,
    worst=\colW,
    cold_start=\colB,
    samples_10=\colC,
    samples_100=\colD,
    samples_1000=\colE,
    }{\colQ & \colA & \colW & \colB & \colC &\colD &\colE}
    \bottomrule
  \end{tabular}
  \quad
  \begin{tabular}{cccccc}
    \midrule
    \multicolumn{6}{c}{NMSE}\\
    \midrule
    Qtl.&
  Tol.&
  Emp.&
  \multicolumn{3}{c}{Bound}
  \\
  {} & {} & {} & $N=10$ & $N=100$ & $N=1000$\\
  \midrule
  \begin{filecontents*}{./data/mnist/mnist_nmse.csv}
quantile,tol,cold_start,worst,samples_10,samples_100,samples_1000
30,-20,152,-,985,290,210
,-30,701,-,2707,1285,916
,-40,2688,-,5919,4016,3236
90,-20,698,-,-,1497,985
,-30,2392,-,-,3885,2950
,-40,5922,-,-,8050,6964
99,-20,1338,-,-,-,2731
,-30,3765,-,-,-,6707
,-40,8241,-,-,-,12077
\end{filecontents*}
\csvreader[head to column names, late after line=\\]{./data/mnist/mnist_nmse.csv}{
    quantile=\colQ,
  tol=\colA,
  cold_start=\colB,
  samples_10=\colC,
  samples_100=\colD,
  samples_1000=\colE,
  }{\colQ & \colA & \colB & \colC &\colD &\colE}
  \bottomrule
\end{tabular}
  }
\end{table}

\begin{figure}[!h]
    \begin{subfigure}[t]{.99\linewidth}
      \centering
      \includegraphics[width=\linewidth]{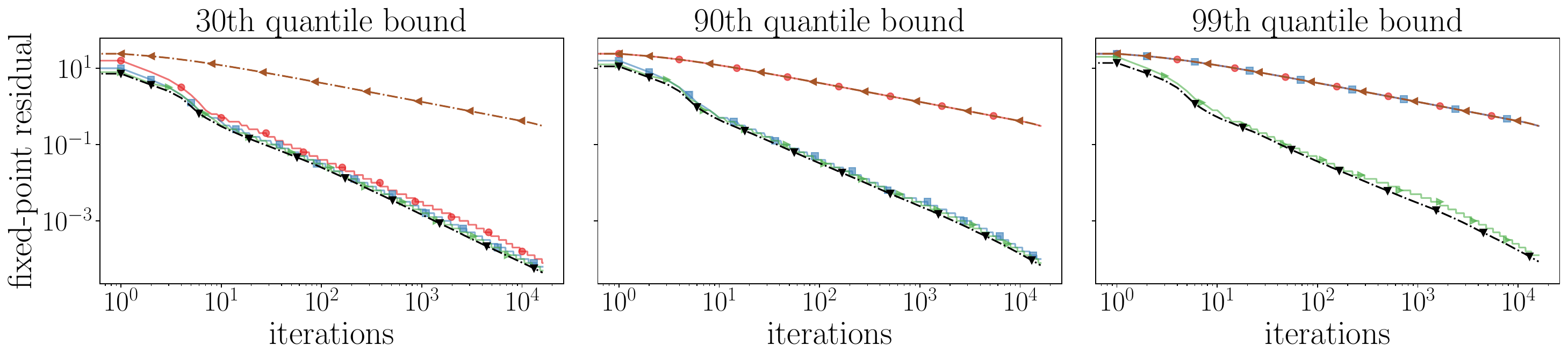}
    \end{subfigure}\\
  \begin{subfigure}[t]{.99\linewidth}
    \centering
    \includegraphics[width=\linewidth]{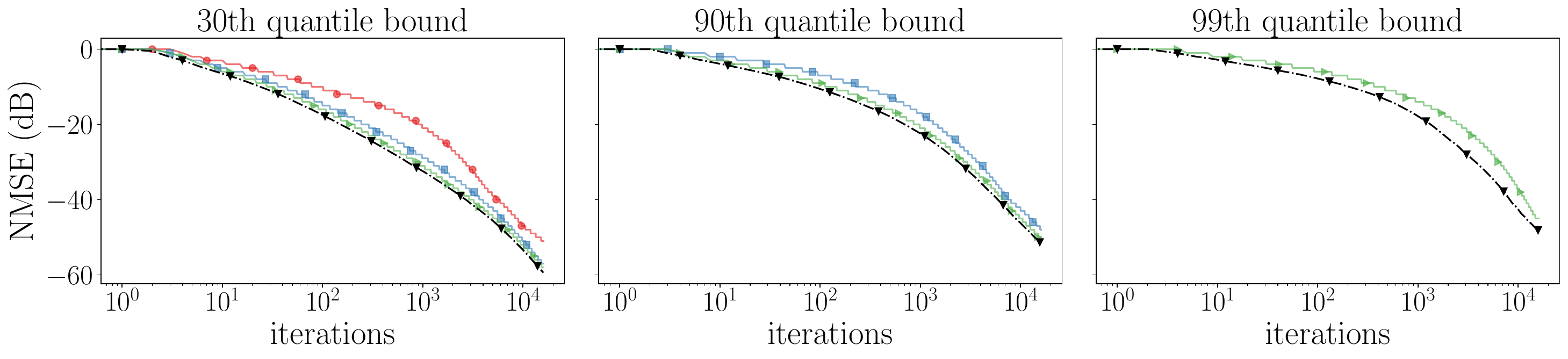}
  \end{subfigure}
  \centering
  \classicallegendquantile
  \caption{
    Probabilistic guarantees for OSQP to solve the image deblurring problem.
  The top row shows results for the fixed-point residual (fp. res.) and the bottom row shows results for the NMSE.
  The quantile bounds for both quantities improve as the number of samples increases.
  % The black curve is the empirical average of problems that reach the tolerance with a given number of algorithm steps.
  % With probability at least $0.99$, the red, blue, and green curves lower bound $1 - \risk$ provided $10$, $100$, and $1000$ samples respectively.
  % % the expected number of problems that reach a given tolerance after a given number of algorithm steps \blue{be careful here}.
  % As the number of samples gets larger, the expected number of problems that reach a given tolerance for the maximum Euclidean distance within a budget of iterations decreases.
  % Moreover, operator theory gives a guarantee that with enough iterations, all of the problems will reach the desired accuracy.
  }
  \label{fig:mnist_custom}
\end{figure}

\begin{table}[!h]
  \centering
  \footnotesize
    \renewcommand*{\arraystretch}{1.0}
  \caption{The quantile results for robust Kalman filtering. 
  Left: fixed-point residual.
  Right: max Euclidean distance.
  We report the number of iterations to reach given tolerances.
  For different quantiles (Qtl.) and tolerances (Tol.), we compare the empirical (Emp.) and estimated worst-case quantities against our probabilistic bounds with a varying number of samples $N$.
  The worst-case bound holds independently of the quantile.
  }
  \label{tab:mnist}
  \vspace*{-3mm}
  \adjustbox{max width=\textwidth}{
    \begin{tabular}{ccccccc}
      \multicolumn{7}{c}{Fixed-point residual}\\
      \midrule
      Qtl.&
    Tol.&
    Worst-&
    Emp.&
    \multicolumn{3}{c}{Bound}
    \\
    {} & {} & Case &{} & $N=10$ & $N=100$ & $N=1000$\\
    \midrule
    \begin{filecontents*}{./data/robust_kalman/robust_kalman_fp.csv}
quantile,tol,cold_start,worst,samples_10,samples_100,samples_1000
30,0.01,135,6.4e7,156,146,144
,0.001,220,6.4e9,239,230,229
,0.0001,308,6.4e11,327,321,319
90,0.01,144,6.4e7,-,158,154
,0.001,228,6.4e9,-,243,238
,0.0001,317,6.4e11,-,335,328
99,0.01,150,6.4e7,-,-,162
,0.001,233,6.4e9,-,-,247
,0.0001,324,6.4e11,-,-,336

\end{filecontents*}
\csvreader[head to column names, late after line=\\]{./data/robust_kalman/robust_kalman_fp.csv}{
      quantile=\colQ,
    tol=\colA,
    worst=\colW,
    cold_start=\colB,
    samples_10=\colC,
    samples_100=\colD,
    samples_1000=\colE,
    }{\colQ & \colA & \colW & \colB & \colC &\colD &\colE}
    \bottomrule
  \end{tabular}
  \quad
  \begin{tabular}{cccccc}
    \multicolumn{6}{c}{Max Euclidean distance}\\
    \midrule
    Qtl.&
  Tol.&
  Emp.&
  \multicolumn{3}{c}{Bound}
  \\
  {} & {} & {} & $N=10$ & $N=100$ & $N=1000$\\
  \midrule
  \begin{filecontents*}{./data/robust_kalman/robust_kalman_custom.csv}
quantile,tol,cold_start,samples_10,samples_100,samples_1000
30,0.01,81,117,92,80
,0.001,147,193,164,156
,0.0001,228,270,236,235
90,0.01,103,-,144,121
,0.001,175,-,222,194
,0.0001,250,-,307,267
99,0.01,116,-,-,150
,0.001,176,-,-,222
,0.0001,251,-,-,311
\end{filecontents*}
\csvreader[head to column names, late after line=\\]{./data/robust_kalman/robust_kalman_custom.csv}{
    quantile=\colQ,
  tol=\colA,
  cold_start=\colB,
  samples_10=\colC,
  samples_100=\colD,
  samples_1000=\colE,
  }{\colQ & \colA & \colB & \colC &\colD &\colE}
  \bottomrule
\end{tabular}
  }
\end{table}

\subsubsection{Robust Kalman filtering}\label{subsubsec:rkf}
% We now turn our attention to robust Kalman filtering.
Kalman filtering~\citep{kalman_filter} is a popular method to predict system states in the presence of noise in dynamic systems.
In this example, we consider robust Kalman filtering~\citep{rkf} which mitigates the impact of outliers and model misspecifications to track a moving vehicle from noisy data location as in~\citet{neural_fp_accel_amos}.
The linear dynamical system with matrices $A \in \reals^{n_s \times n_s}$, $B \in \reals^{n_s \times n_u}$, and~$C \in \reals^{n_o \times n_s}$ is given by
\begin{equation}\label{eq:rkf_dynamics_eqn}
    s_{t+1} = As_t + Bw_t, \quad y_t = C s_t + v_t,\quad \text{for}\quad t=0,1,\dots,
\end{equation}
where $s_t \in \reals^{n_s}$ is the state, $y_t \in \reals^{n_o}$ is the observation, $w_t \in \reals^{n_u}$ is the input, and $v_t \in \reals^{n_o}$ is a perturbation to the observation.
% The matrices $A \in \reals^{n_s \times n_s}$, $B \in \reals^{n_s \times n_u}$, and~$C \in \reals^{n_o \times n_s}$ give the dynamics of the system.
We aim to recover the state $x_t$ from the noisy measurements $y_t$ by solving the following problem:
\begin{equation}
\begin{array}{ll}
\label{prob:rkf}
\mbox{minimize} & \sum_{t=1}^{T-1} \|w_t\|_2^2 + \mu \psi_{\rho} (v_t) \\
\mbox{subject to} & s_{t+1} = A s_t + B w_t \quad t=0, \dots, T-1 \\
& y_t = C s_t + v_t \quad t=0, \dots, T-1.\\
\end{array}
\end{equation}
Here, the Huber penalty function~\citep{huber} with parameter $\rho \in \reals_{++}$ that robustifies against outliers is given by
% \bnote{Mention this is the Huber loss (with citation)}
\[\psi_{\rho}(a) = \begin{cases}
      \|a\|_2 & \|a\|_2 \leq \rho \\
      2 \rho \|a\|_2 - \rho^2 & \|a\|_2 \geq \rho. \\
   \end{cases}
\]
The given quantity $\mu \in \reals_{++}$ weights this penalty term.
The decision variables are the $s_t$'s, $w_t$'s, and $v_t$'s, while the parameters are the observed $y_t$'s: $x = (y_0, \dots, y_{T-1})$.
We formulate problem~\eqref{prob:rkf} as a second-order cone program, and use SCS~\citep{scs_quadratic} to solve it.
% Problem~\eqref{prob:rkf} can be formulated as a second-order cone program, so we use SCS~\citep{scs_quadratic} to solve it.
% \red{resolve where to mention OSQP and SCS.}

\paragraph{Task-specific metric.}
In Kalman filtering, traditional metrics such as the fixed-point residual, which encompasses both primal and dual variables, may not fully capture specific aspects of state recovery accuracy.
In light of this context, we propose a task-specific metric aimed at quantifying the fidelity of state estimation.
This metric measures the deviation of the algorithmically recovered states, $s_1, \dots, s_T$, (extracted from the fixed-point vector $z$) after $k$ iterations, from their corresponding optimal states, $s_1^\star(x), \dots, s_T^\star(x)$:
% To precisely assess this deviation, we introduce a metric defined as the maximum Euclidean distance between the recovered and optimal states across all time steps:
\begin{equation}\label{eq:max_Euclidean}
\phi(z, x) = \max_{t=1, \dots, T} \|s_t - s^\star_t(x)\|_2.
\end{equation}
% Here, $s_t$ is extracted from the candidate solution $z$, and $s_t^\star(x)$ is the optimal state at the $t$-th time index for the problem with parameter $x$.
% we make the dependence on the parameter $x$ clear with the notation $s_t^\star(x)$.
% Here, $s_t$ is extracted from the candidate solution $z$, and we make the dependence on the parameter $x$ clear with the notation $s_t^\star(x)$.
The associated error metric indicates success when each recovered state lies within an $\epsilon$-radius ball centered at its optimal counterpart.

\paragraph{Numerical example.}
% As in \citet{neural_fp_accel_amos}, we set $n_s=4$, $n_o=2$, $n_u=2$, $\mu=2$, $\rho=2$, and $T=50$.
We follow the setup from~\citet{neural_fp_accel_amos} where $n_s=4$, $n_o=2$, $n_u=2$, $\mu=2$, $\rho=2$, and $T=50$.
The dynamics matrices are
% \vspace{4mm}\\
% \vspace{4mm}
% \resizebox{\textwidth}{!}{%
\begin{equation*}
\small
    \label{eq:rkf_dynamics}
    A = \begin{bmatrix}
        1 & 0 & (1 - (\gamma/2)\Delta t) \Delta t & 0\\
        0 & 1 & 0 & (1 - (\gamma/2)\Delta t) \Delta t\\
        0 & 0 & 1 - \gamma \Delta t & 0\\
        0 & 0 & 0 & 1 - \gamma \Delta t
    \end{bmatrix},\hspace{0mm}
    B = \begin{bmatrix}
        1 / 2\Delta t^2 & 0 \\
        0 & 1 / 2\Delta t^2\\
        \Delta t & 0 \\
        0 & \Delta t
    \end{bmatrix}, \hspace{0mm}
    C = \begin{bmatrix}
        1 & 0 & 0 & 0\\
        0 & 1 & 0 & 0\\
    \end{bmatrix},
\end{equation*}
% }
where $\Delta t=0.5$ and $\gamma=0.05$ are fixed to be respectively the sampling time and the velocity dampening parameter.
% Here we set them to be $\Delta t = 0.5$ and $\gamma = 0.05$.
% We generate the problem instances in the following way.
We generate true trajectories $\{x_0^*, \dots, x_{T-1}^*\}$ of the vehicle by first letting $x_0^* = 0$.
Then we sample the inputs as $w_t \sim \mathcal{N}(0, 0.01)$ and $v_t \sim \mathcal{N}(0, 0.01)$.
The trajectories are then fully defined via the dynamics equations in \Eqn~\eqref{eq:rkf_dynamics_eqn} with the sampled $w_t$'s and $v_t$'s.

\paragraph{Results.}
To the best of our knowledge, the tightest guarantees that can be obtained for this problem on the fixed-point residual are given the bound from the averaged iterations~\eqref{eq:avg_rate}.
We visualize our bounds on the maximum Euclidean distance metric from \Eqn~\eqref{eq:max_Euclidean} in Figure~\ref{fig:rkf_visuals} with a ball of radius $0.1$ around the optimal state.
We obtain probabilistic guarantees on the error metric that says that all of the recovered states are within their respective balls.

\begin{figure}[!h]
  \begin{subfigure}[t]{.99\linewidth}
      \centering
      \includegraphics[width=\linewidth]{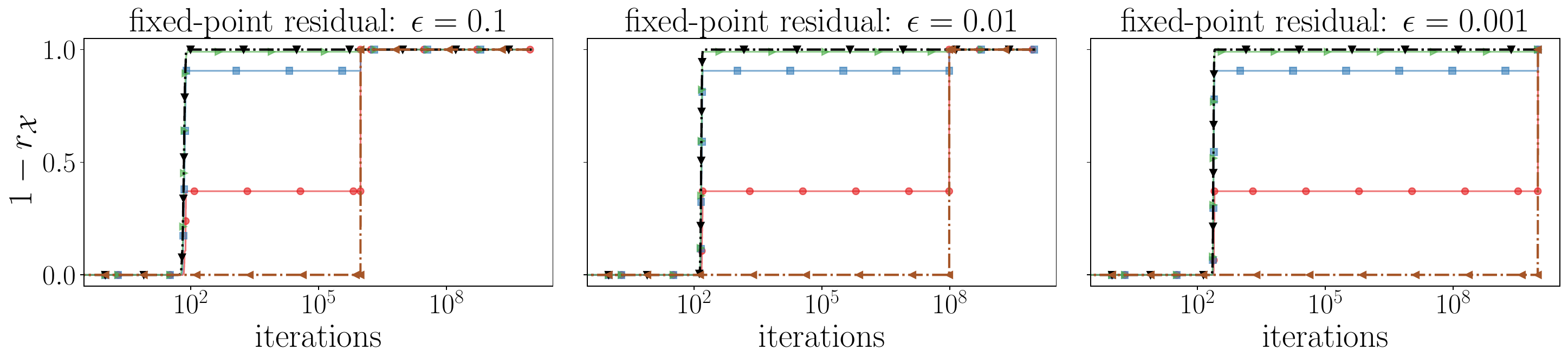}
    \end{subfigure}\\
    \begin{subfigure}[t]{.99\linewidth}
      \centering
      \includegraphics[width=\linewidth]{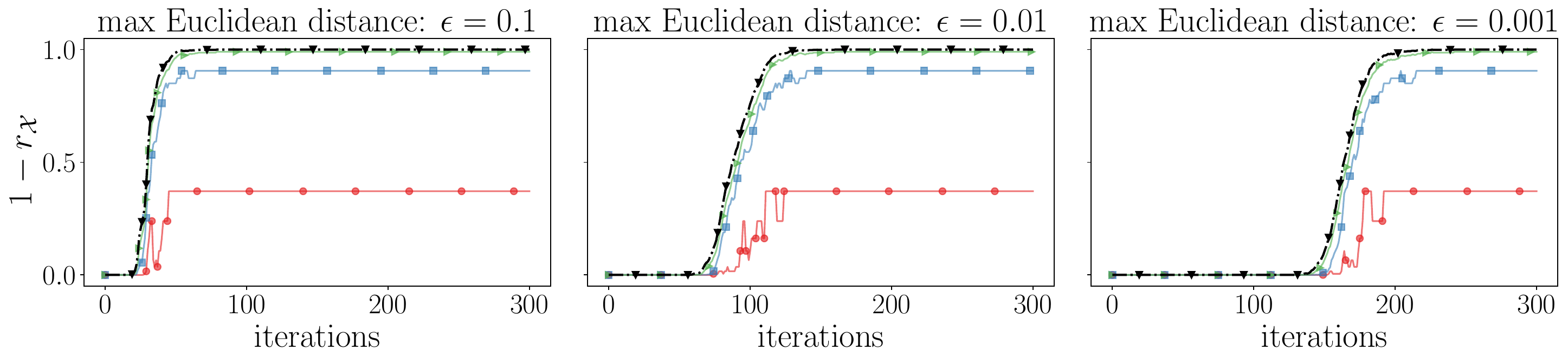}
    \end{subfigure}
  \centering
  \classicallegend
  \caption{
    Probabilistic lower bounds of the success rate for robust Kalman filtering.
    Top: fixed-point residual.
    Bottom: maximum Euclidean distance from \Eqn~\eqref{eq:max_Euclidean}.
    % The top row shows results for the fixed-point residual (fp. res.) and the bottom row shows bounds for the maximum Euclidean distance from \Eqn~\eqref{eq:max_Euclidean}.
    Note that the x-axes are different for the top and bottom rows.
    The bounds get tighter as the number of samples increases.
    % The black curve is the empirical average of problems that reach the tolerance with a given number of algorithm steps.
    % With probability at least $0.99$, the red, blue, and green curves lower bound $1 - \risk$ provided $10$, $100$, and $1000$ samples respectively.
    % Moreover, operator theory gives a guarantee that with enough iterations, all of the problems will reach the desired accuracy.
    }
    \label{fig:rkf_fp}
\end{figure}

\begin{figure}[!h]
    \begin{subfigure}[t]{.99\linewidth}
      \centering
      \includegraphics[width=\linewidth]{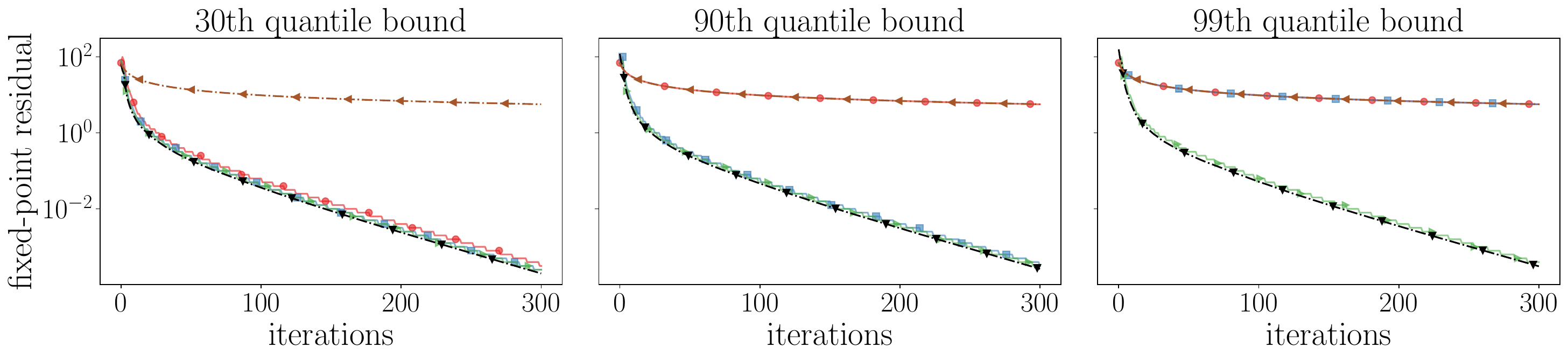}
    \end{subfigure}\\
  \begin{subfigure}[t]{.99\linewidth}
    \centering
    \includegraphics[width=\linewidth]{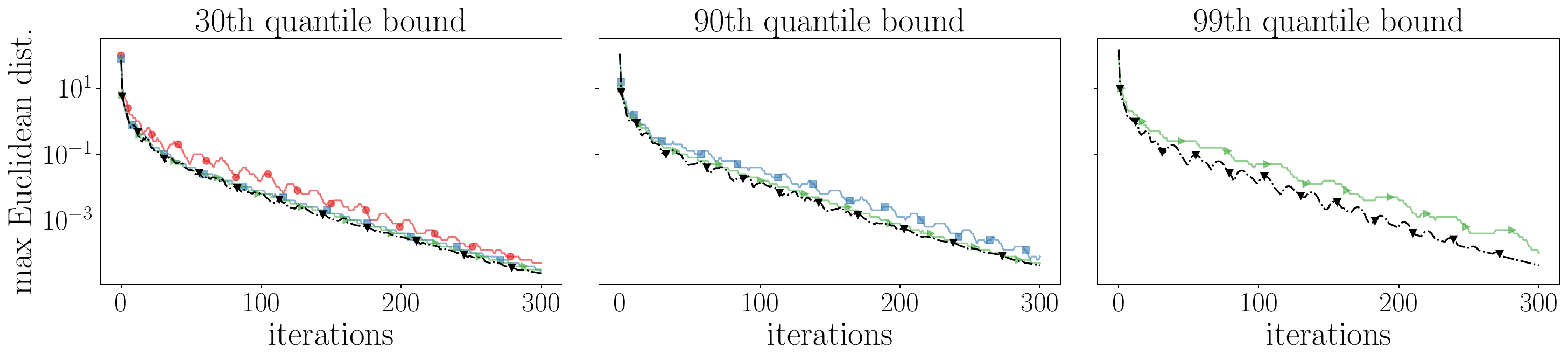}
  \end{subfigure}
  \centering
  \classicallegendquantile
  \caption{
    Probabilistic guarantees for SCS to solve the robust Kalman filtering problem.
    Top: fixed-point residual.
    Bottom: maximum Euclidean distance from \Eqn~\eqref{eq:max_Euclidean}.
    Our bounds resemble linear convergence, while the worst-case guarantee gives sublinear convergence.
  % The top row shows results for the fixed-point residual (fp. res.) and the bottom row shows results for the maximum Euclidean distance from~\Eqn~\eqref{eq:max_Euclidean}.
  % The black curve is the empirical average of problems that reach the tolerance with a given number of algorithm steps.
  % With probability at least $0.99$, the red, blue, and green curves lower bound $1 - \risk$ provided $10$, $100$, and $1000$ samples respectively.
  % As the number of samples gets larger, the expected number of problems that reach a given tolerance for the maximum Euclidean distance within a budget of iterations decreases.
  }
  \label{fig:rkf_custom}
\end{figure}

\begin{figure}[!h]
  \begin{subfigure}[t]{.33\linewidth}
    \centering
    \includegraphics[width=\linewidth]{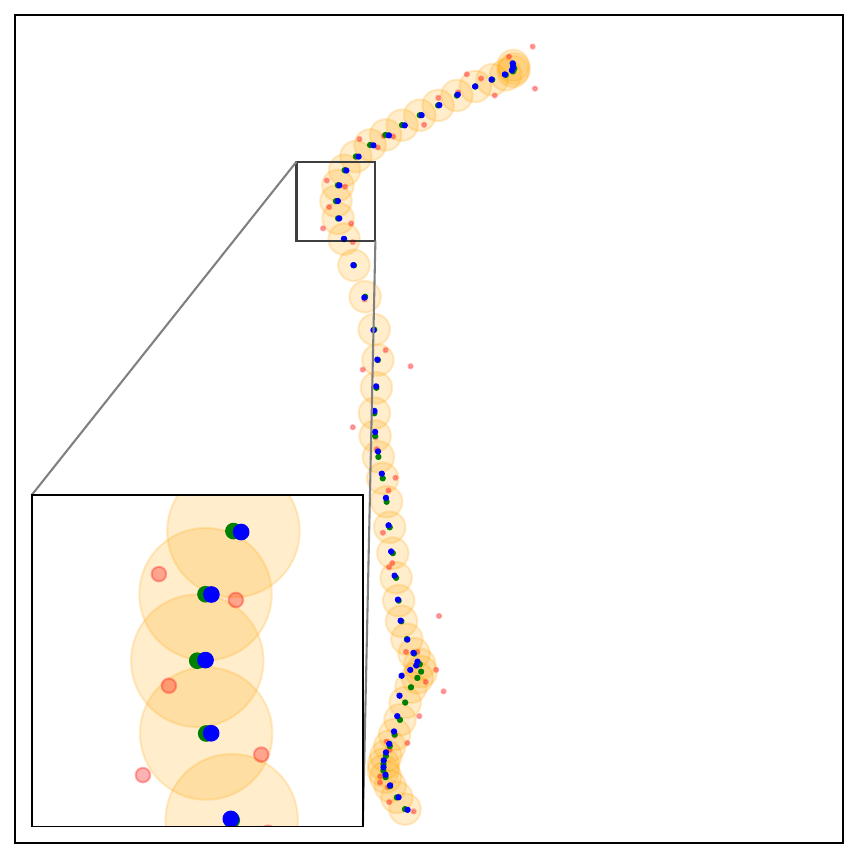}
  \end{subfigure}%
  \begin{subfigure}[t]{.33\linewidth}
    \centering
    \includegraphics[width=\linewidth]{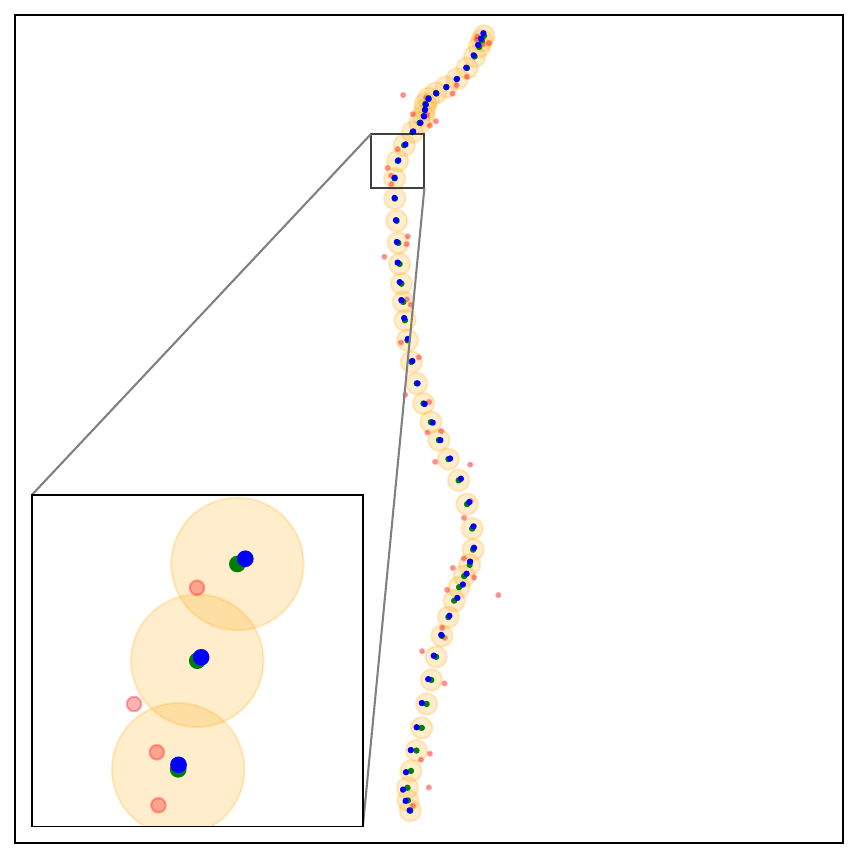}
  \end{subfigure}%
  \begin{subfigure}[t]{.33\linewidth}
    \centering
    \includegraphics[width=\linewidth]{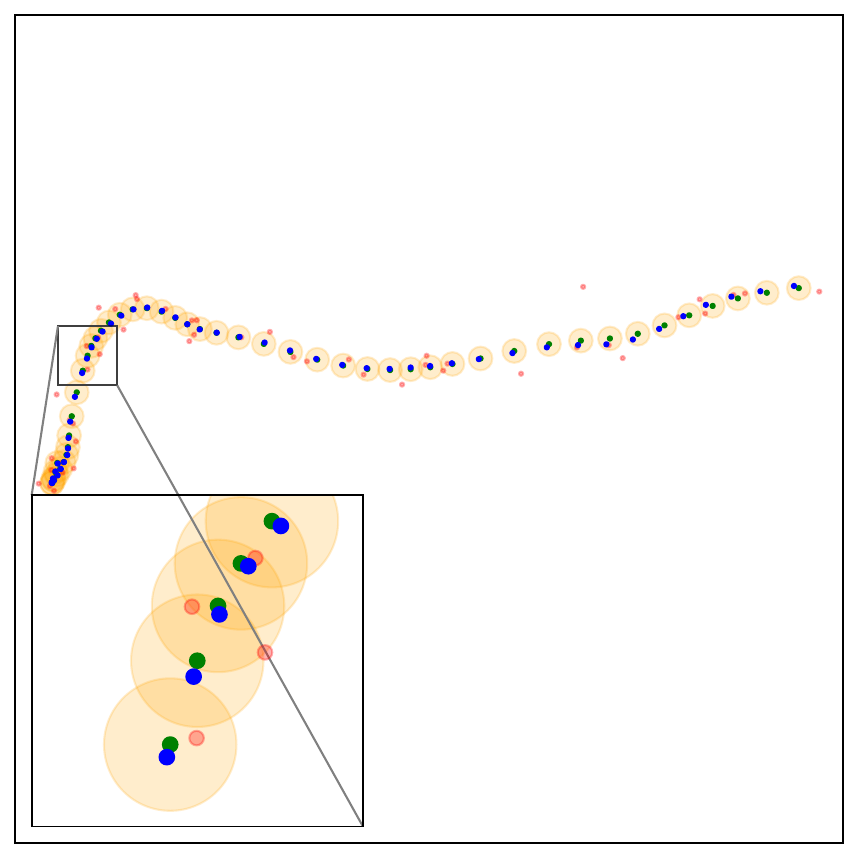}
  \end{subfigure}%
  \centering
  \rkfvisualslegend
  \caption{Visualizing the guarantees for robust Kalman filtering with the maximum Euclidean distance metric~\eqref{eq:max_Euclidean}.
    Each plot is a separate parametric problem.
    The noisy, observed trajectory is made up of the pink points which parametrize the problem.
    The robust Kalman filtering recovery, the optimal solution of problem~\eqref{prob:rkf}, is shown as green points.
    The shaded beige regions are centered at the optimal points with radius $0.1$.
    With high probability, all of the extracted states after $30$ steps will fall within their respective beige regions across $86 \%$ of problem instances.
    % We provide strong probabilistic guarantees that all of the states extracted from the solution after $15$ steps are within their respective beige regions.
    % With probability at least $0.99$, the risk with $15$ steps and radius $0.1$ is upper bounded by $0.1$.
    % With probability at least $0.99$, the solution returned after $25$ steps is within all of the balls of radius $0.1$ and the solution returned after $50$ steps is within all of the balls of radius $0.03$ \red{incorrect}.
  }
  \label{fig:rkf_visuals}
\end{figure}

\subsection{Learned optimizers}\label{subsec:experiments_learned}
In this subsection we apply our method to obtain generalization guarantees for a variety of learned optimizers: LISTA~\citep{lista} and several of its variants~\citep{alista,glista} in \Sec~\ref{subsec:lista}, learning warm starts (L2WS) for fixed-point optimization problems~\citep{l2ws} in \Sec~\ref{subsec:l2ws}, and model-agnostic meta-learning (MAML)~\citep{maml} in \Sec~\ref{subsec:maml}.

We implement our learning algorithm for the different learned optimizers in JAX~\citep{jax}, using the JAXOPT library~\citep{jaxopt_implicit_diff} with the ADAM optimizer~\citep{adam}.
To do the implicit differentiation during training, we use a bisection method to solve the KL inverse problems.
For each of the learned optimizers, we describe how we partition the weights into groups as mentioned in \Sec~\ref{sec:gen_l2o}, and report the number of partitions $J$.
We also describe how the prior mean is set for each learned optimizer.
We use $50000$ training samples and evaluate on $1000$ test problems in each example.

After training, we calibrate our bounds, and report a lower bound on the success rate $1 - \exprisk(P)$ for the learned optimizer with posterior distribution $P = \mathcal{N}_{w^\star,s^\star}$.
As in the results for the classical optimizers, we report bounds across many algorithm steps and tolerances, construct quantile bounds at each step, and report the results for task-specific metrics where applicable.
The probabilistic results hold with probability at least $0.99988$ for the risk and with probability at least $0.99028$ for the quantiles; see Appendix \Sec~\ref{sec:prob_details} for details.

% \blue{B: in this and the following examples, I would clearly state what is $x$, what is $z$, and how you generate the data (\eg, how many datapoints for training? how many for calibration?)}
% \blue{B: Also, shall we use iterations instead of evaluation steps?}
% \blue{B: I would use \emph{empirical quantile} instead of \emph{cold-start}}

\subsubsection{LISTA variants for sparse coding problems}\label{subsec:lista}
In the sparse coding problem, the goal is to recover a sparse vector $z \in \reals^n$ given a dictionary $D \in \reals^{m \times n}$ from noisy linear measurements
\begin{equation*}
  b = Dz + \epsilon,
\end{equation*}
where $b \in \reals^m$ is the noisy measurement and $\epsilon \in \reals^m$ is additive Gaussian white noise.
A popular approach to solve this problem is to formulate it as the lasso problem
\begin{equation}\label{prob:lasso}
  \begin{array}{ll}
\mbox{minimize} & (1/2) \|Dz - b\|_2^2 + \rho \|z\|_1,
\end{array}
\end{equation}
where $\rho \in \reals_{++}$ is a hyperparameter, and then run the iterative shrinkage thresholding algorithm (ISTA) with algorithm steps
\begin{equation*}
  z^{k+1} = \eta_{\rho / L} \left(z^k - \frac{1}{L}D^T(Dz^k - b)\right).
\end{equation*}
Here $\eta_\psi$ is the soft-thresholding function $\eta_\psi(z) = \textbf{sign}(z)\max(0, |z| - \psi)$
and $L \in \reals_{++}$ is less than or equal to the largest eigenvalue of $D^T D$.
Seeking faster convergence, learned ISTA (LISTA)~\citep{lista} and its variants learn some of the components of the update function.
% Several variants of LISTA, which learn different components of the update function, have since been proposed.
% Learned optimizers for sparse coding in the literature, originally called learned ISTA (LISTA)~\citep{lista}, replace (some of) the updates of ISTA with learned components.
All of these learned optimizers seek a good set of weights $\theta$ to solve problem~\eqref{prob:simplified_l2o} where the initial iterate is set to zero, \ie, $h_\theta(x) = 0$.
In this subsection, we apply our method to LISTA and several of its variants enumerated below, and compare the performance against classical algorithms: ISTA and its accelerated version, Fast ISTA (FISTA)~\citep{fista}.
% we use ISTA and its accelerated version, Fast ISTA (FISTA)~\citep{fista} as baselines.
% \begin{equation*}
%   h_\theta(x) = 0.
% \end{equation*}
% Thus the LISTA variants seek to optimize the algorithm steps themselves.
% The function $g_\theta : \reals^n \times \reals^d \rightarrow \reals^n$ is the learned algorithm step.
% For this example, we use ISTA and its accelerated version, Fast ISTA (FISTA)~\citep{fista} as baselines.
% FISTA introduces momentum and is known to converge more rapidly than ISTA both in theory and in practice.
% Here, $\alpha \in \reals^k$ are weights that determine the importance of accuracy at each iteration.
% In many cases, the step-T functional is used~\citep{lista,alista} which sets $\alpha_T=1$ and $\alpha_1 = \dots = \alpha_{T-1} = 0$.
% \red{potentially remove}

\paragraph{LISTA.}
The LISTA updates from the seminal work of~\citet{lista} are
\begin{equation}\label{eq:lista}
  z^{k+1} = \eta_{\psi^k} \left(W^k_1 z^k + W^k_2 b \right),
\end{equation}
where the learned parameters are $\theta=(\{\psi^k, W_1^k, W_2^k\}_{k=0}^{K-1}) \in \reals^{K(1 + mn + n^2)}$.
We partition the weights into $J=3$ groups: the shrinkage thresholds $\{\psi^k\}_{k=0}^{K-1}$, the first set of weight matrices $\{W_1^k\}_{k=0}^{K-1}$, and the second set of weight matrices $\{W_2^k\}_{k=0}^{K-1}$.
We set the prior means for the weights to the values of ISTA with $\rho = 0.1$.

\paragraph{TiLISTA.}
% \citet{lista_cpss} observed that there is a benefit in coupling the two weight matrices from the LISTA updates~\eqref{eq:lista}.
TiLISTA~\citep{alista}, a variant of LISTA, couples the two weight matrices and ties the matrix updates over the iterates so that they only differ by a learned scalar factor.
% In TiLISTA, the weight matrix $\tilde{W}$ is not pre-computed and is one of the parameters that is learned.
The TiLISTA updates are given by
\begin{equation*}
  z^{k+1} = \eta_{\psi^k} \left( z^k - \gamma^k \tilde{W}^T (Dz^k - b)\right),
\end{equation*}
where the weights are $\theta=(\tilde{W}, \{\psi^k, \gamma^k\}_{k=0}^{K-1}) \in \reals^{2K + mn}$.
We partition the weights into $J=3$ groups:  the shrinkage thresholds $\{\psi^k\}_{k=0}^{K-1}$, the step sizes $\{\gamma^k\}_{k=0}^{K-1}$, and the matrix $\tilde{W}$.
We set the prior mean for $\tilde{W}$ to be the pre-computed value given by solving problem~\eqref{prob:alista_stage1} and zero for the other groups.

\paragraph{ALISTA.}
\citet{alista} also propose ALISTA, which significantly reduces the number of learned algorithm parameters by determining the matrix $\tilde{W}$ from TiLISTA in a data-free manner.
The ALISTA updates are given by
\begin{equation}\label{eq:alista_datafree}
  z^{k+1} = \eta_{\psi^k} ( z^k - \gamma^k \tilde{W}^T (Dz^k - b)),
\end{equation}
where $\tilde{W} \in \reals^{m \times n}$ is pre-computed in a data-free manner by solving the convex QP
\begin{equation}\label{prob:alista_stage1}
  \begin{array}{ll}
  \mbox{minimize} & \|W^T D\|_F^2 \\
  \mbox{subject to} &W_{:,i}^T D_{:,i} = 1 \quad i=1, \dots, m.
\end{array}
\end{equation}
Then the parameters $\theta=(\{\psi^k, \gamma^k\}_{k=0}^{K-1}) \in \reals^{2K}$ are learned in an end-to-end fashion.
We partition the weights into $J=2$ groups: the shrinkage thresholds $\{\psi^k\}_{k=0}^{K-1}$ and the step sizes $\{\gamma^k\}_{k=0}^{K-1}$.
Since the matrix $\tilde{W}$ is pre-computed for ALISTA in a data-free manner, we do not train over this variable.
We set all of the prior means for the weights to zero.

\paragraph{GLISTA.}
In GLISTA~\citep{glista}, we incorporate gain gates and overshoot gates to the ALISTA model.
% GLISTA can be used in conjunction with other LISTA-variants.
% In this case, we use it with conjunction with ALISTA.
The GLISTA updates are given by
\begin{align*}
  \tilde{z}^{k+1} &= \eta_{\psi^k} \left(1 + \mu^k \psi^{k-1} \exp(\nu_k |z^k|) - \gamma^k \tilde{W}^T \left(D (1 + \mu^k \psi^{k-1} \exp(\nu_j |z^k|)) - b\right)\right)\\
  z^{k+1} &= \left(1 + \frac{a^k}{|\tilde{z}^{k+1} - z^k| + \epsilon}\right) \odot \tilde{z}^{k+1} - \left(\frac{a^k}{|\tilde{z}^{k+1} - z^k| + \epsilon}\right) \odot z^k,
\end{align*}
where $\epsilon \in \reals_{++}$ is a small positive value.
The weight matrix $\tilde{W}$ is pre-computed in a data-free manner as in ALISTA.
% The TiLISTA updates are given by
% \begin{equation*}
%   z^{i+1} = \eta_{\psi^i} \biggl( z^i - \gamma^i \tilde{W}^T (Dx^i - b)\biggr)
% \end{equation*}
% where the learned parameters are $w=(\tilde{W}, \{\psi^i, \gamma^i\}_{i=0}^{K-1}) \in \reals^{2K + mn}$.
We partition the weights into $J=5$ groups: the shrinkage thresholds $\{\psi^k\}_{k=0}^{K-1}$, the step sizes $\{\gamma^k\}_{k=0}^{K-1}$, the two sets of gated parameters $\{\mu^k\}_{k=0}^{K-1}$ and $\{\nu^k\}_{k=0}^{K-1}$, and the overshoot parameters $\{a^k\}_{k=0}^{K-1}$.
Thus the learned parameters are $\theta=(\{\psi^k, \gamma^k, \mu^k, \nu^k, a^k\}_{k=0}^{K-1}) \in \reals^{5K}$.
We set all of the prior means for the weights to zero.

\paragraph{Task-specific metric.}
% In sparse coding, it is common to report the normalized mean squared error in decibel units~\citep{l2o} given by
% \begin{equation*}
%   {\rm NMSE}(z,x) = 10 \log_{10} \frac{\|z - z^\star(\param)\|_2^2}{\|z^\star(\param)\|_2^2}.
% \end{equation*}
% Note that this metric is a normalized version of the regression loss from \Eqn~\eqref{eq:reg_loss} and we only report results on the NMSE for this example.
In this example, we only report normalized mean squared error in decibel (dB) units from \Eqn~\eqref{eq:nmse} as this is common in the literature for sparse coding~\citep{l2o}.
% Note that this metric is a normalized version of the regression loss from \Eqn~\eqref{eq:reg_loss}.
% We only report results on the NMSE for this example.

\paragraph{Numerical example.}
We follow the setup from~\citet{l2o} for this example.
We sample a dictionary $D \in \reals^{m \times n}$ with i.i.d. entries from the distribution $\mathcal{N}(0,1/m)$.
Then we normalize $D$ so that each column has Euclidean norm of one.
To generate each sample, we generate the ground truth from the distribution $\mathcal{N}(0,1)$ and zero out each entry with a probability of $0.9$.
The noise $\epsilon$ is set to a signal to noise ratio of $40$dB.
% We use a signal to noise ratio of $40dB$ for the noise $\epsilon$.
Then the measurement is $b = D z + \epsilon$.
We take a matrix with dimensions $m=256$, $n=512$, and pick the number of algorithm steps to be $K=10$.
We compare against ISTA and FISTA, setting $\rho = 0.1$, a value picked in~\citet{lista_cpss}.
We calibrate with $20000$ Monte Carlo samples of the weights.
% It is typical behavior that ISTA and FISTA do not perform well on a low number of iterations (in the range $10$ to $20$)~\citep{lista_cpss}.

\paragraph{Results.}
Figure~\ref{fig:alista_results1} along with Table~\ref{tab:sparse_coding} show the behavior of our method.
The classical optimizers ISTA and FISTA hardly make progress within $10$ iterations as is commonly observed in the literature~\citep{lista_cpss}.
% It is commonly observed that both ISTA and FISTA exhibit suboptimal performance when constrained to a low (less than $20$) number of iterations \citep{lista_cpss}.
% The results are in Figure~\ref{fig:alista_results1} and in Table~\ref{tab:sparse_coding}.
% All of the learned optimizers except for LISTA both empirically outperform the baselines and have bounds that are stronger.
Our method with all of the learned optimizers except for LISTA provides generalization guarantees that are much stronger than the baseline performance.
Moreover, the guarantees are close to the empirical results showing that our bounds are tight.
Our method with LISTA performs poorly because there are a very large number of weights, which in turn makes the regularizer $B(w,s,\lambda)$ significantly larger and more difficult to optimize.

\begin{table}[!h]
  \centering
  \footnotesize
    \renewcommand*{\arraystretch}{1.0}
  \caption{\reviewChanges{The quantile results for sparse coding on the NMSE after $10$ iterations.
  We report the empirical average of test problems (Emp.) and bounds (Bnd.) for each of the learned optimizers.}
  }
  \label{tab:sparse_coding}
  \vspace*{-3mm}
  \begin{tabular}{l}
  \end{tabular}
  \reviewChanges{
  \adjustbox{max width=.85\textwidth}{
    \begin{tabular}{ccccccccccc}
      \midrule
    Quantile&
    ISTA&
    FISTA&
    \multicolumn{2}{c}{LISTA}&
    \multicolumn{2}{c}{TiLISTA}&
    \multicolumn{2}{c}{ALISTA}&
    \multicolumn{2}{c}{GLISTA}
    \\
    {} & Emp. & Emp. & Emp. & Bnd. & Emp. & Bnd. & Emp. & Bnd. & Emp. & Bnd.\\
    \midrule
    % Add directly from csv reader
    % the csv file has names colnameA, colnameB
    \begin{filecontents*}{./data/sparse_coding.csv}
,quantiles,ista,fista,alista_emp,alista_bound,tilista_emp,tilista_bound,glista_emp,glista_bound,lista_emp,lista_bound
0,10,-0.93,-1.99,-38.09,-37.0,-36.79,-34.0,-43.96,-43.0,-2.2,-1.0
1,30,-0.53,-1.86,-36.38,-35.0,-34.48,-32.0,-42.74,-42.0,-1.82,-1.0
2,50,-0.39,-1.78,-35.06,-33.0,-33.05,-31.0,-41.78,-41.0,-1.57,-1.0
3,60,-0.31,-1.73,-34.52,-32.0,-32.31,-30.0,-41.33,-40.0,-1.46,0.0
4,80,-0.16,-1.66,-31.99,-30.0,-30.08,-27.0,-40.15,-39.0,-1.22,0.0
5,90,-0.09,-1.61,-29.82,-27.0,-28.58,-22.0,-39.19,-38.0,-1.04,0.0
6,95,-0.0,-1.54,-29.06,-23.0,-27.38,-18.0,-38.37,-36.0,-0.92,0.0
\end{filecontents*}
\csvreader[head to column names, late after line=\\]{./data/sparse_coding.csv}{
    quantiles=\colA,
    ista=\colB,
    fista=\colK,
    tilista_emp=\colC,
    tilista_bound=\colD,
    alista_emp=\colE,
    alista_bound=\colF,
    glista_emp=\colG,
    glista_bound=\colH,
    lista_emp=\colI,
    lista_bound=\colJ,
    }{\colA & \colB & \colK & \colI &\colJ&\colC &\colD &\colE &\colF &\colG &\colH}
    \bottomrule
  \end{tabular}}}
\end{table}

\begin{figure}[!h]
  % \begin{subfigure}[t]{.35\linewidth}
  %   \centering
  %   \includegraphics[width=\linewidth]{figures/alista/acc_-10.pdf}
  % \end{subfigure}%
  % \begin{subfigure}[t]{.35\linewidth}
  %   \centering
  %   \includegraphics[width=\linewidth]{figures/alista/acc_-20.pdf}
  % \end{subfigure}\\%
  \begin{subfigure}[t]{1.0\linewidth}
    \centering
    \includegraphics[width=\linewidth]{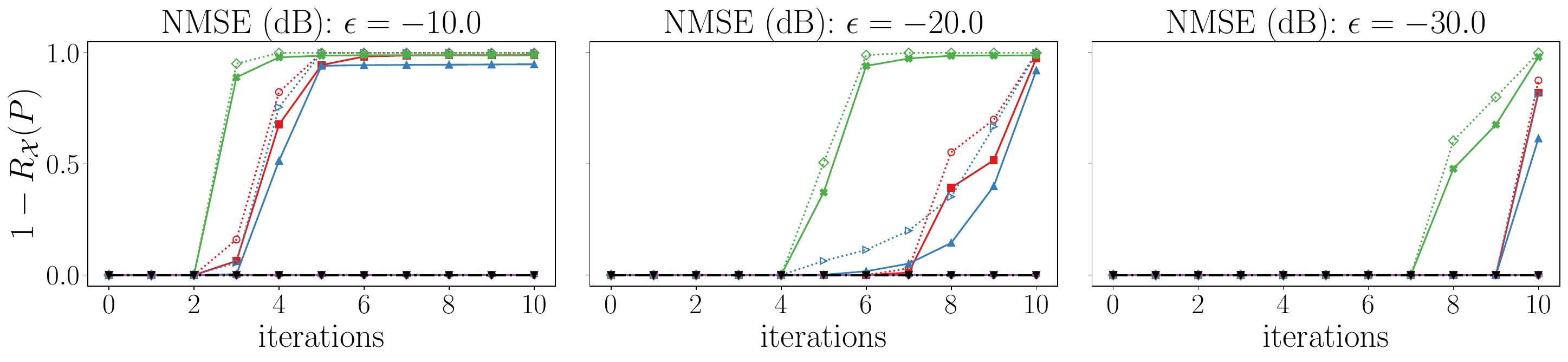}
  \end{subfigure}\\%
  \begin{subfigure}[t]{1.0\linewidth}
    \centering
    \includegraphics[width=\linewidth]{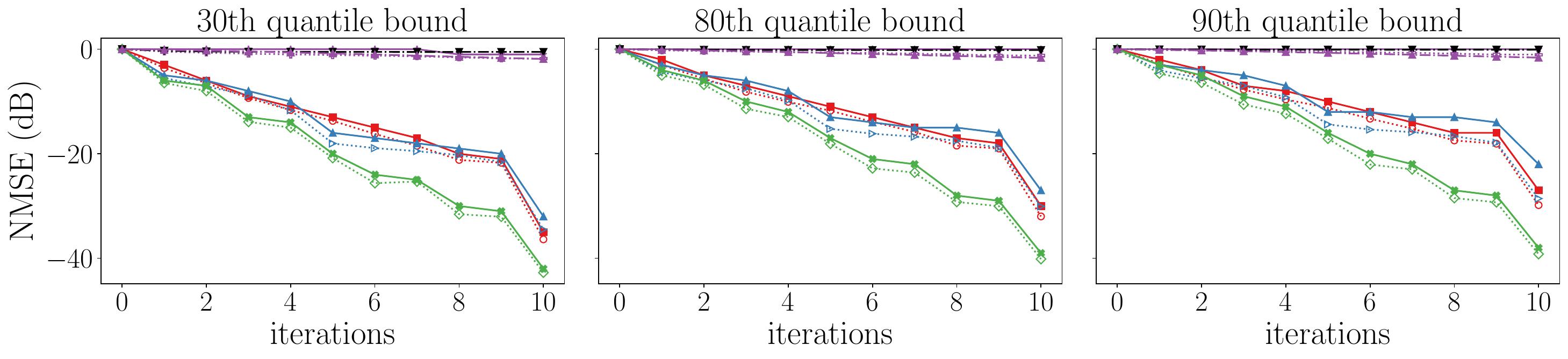}
  \end{subfigure}%
  \centering
  \sparsecodinglegend
  \caption{\reviewChanges{Sparse coding results.
  Top: lower bound on the success rate.
  Bottom: upper bound on the quantile.
    The PAC-Bayes guarantees for ALISTA, TiLISTA, and GLISTA significantly outperform the empirical results given by ISTA.
    The guarantees are close to the corresponding empirical values for the learned optimizers.}
    }
  \label{fig:alista_results1}
\end{figure}

% \red{cutoffs in titles}
% \red{same color, dotted and solid, bound is solid}
% \red{we can just do fista}
% \red{make it clear that acceleration doesnt show}
% \red{add in lista}
% \red{this is typical behavior for these examples -- cite lista-cpss}

% \begin{figure}[!h]
%   \centering
%     \includegraphics[width=\figsize\linewidth]{figures/alista/acc_-20.pdf}
%     \sparsecodinglegend
%     \\
%     \caption{ALISTA and TiLISTA results.
%     Plotting the fraction of test problems that reach $-20$ NMSE (dB) and the associated PAC-Bayes guarantee.
%     }
%     \label{fig:alista_results2}
% \end{figure}

% \begin{figure}[!h]
%   \centering
%     \includegraphics[width=\figsize\linewidth]{figures/alista/conv_rates.pdf}
%     \sparsecodinglegend
%     \\
%     \caption{ALISTA and TiLISTA results.
%     Plotting the fraction of iterations that satisfy the linear convergence rate and the associated PAC-Bayes bound.
%     }
%     \label{fig:alista_results3}
% \end{figure}

\subsubsection{Learning to warm starts for fixed-point problems}\label{subsec:l2ws}
In our second example of learned optimizers, we consider the L2WS framework~\citep{l2ws} which seeks to learn a high-quality initialization to solve the fixed-point problem~\eqref{prob:l2ws}.
% Specifically, we consider parametric fixed-point problems of the form
% \begin{equation}\label{prob:l2ws}
%   \begin{array}{ll}
% \mbox{find} & z \quad \mbox{subject to} \quad z = T(z,x) \\
% \end{array}
% \end{equation}
% where $T: \reals^n \times \reals^d \rightarrow \reals^n$ is the fixed-point operator.
% We denote the set of fixed-points for the fixed-point problem parametrized by $x$ to be $\fix T_x$.
% In this case, the ground truth solutions $z^\star(x)$ satisfy~\Eqn~\eqref{prob:l2ws} and therefore, $z^\star(x) \in \fix T_x$.
% One natural approach to solve this problem is through fixed-point iterations~\ref{eq:fp}.
% Under certain conditions on the operator $T$, the iterates given by~\Eqn~\eqref{eq:fp} converge.
% If the operator $T$ is contractive, linearly convergent, or averaged, and the set of fixed-points is non-empty, then the iterates are guaranteed to converge.
% See \Sec~\ref{sec:op_theory} for definitions of these terms.
% The work of~\citep{l2ws} seeks to learn a good initialization for fixed-point algorithms.
The training problem is problem~\eqref{prob:simplified_l2o} where the initialization $h_\theta(x)$ is learned rather than the algorithm steps (\ie, $T_\theta(z,x) = T(z,x)$).
The objective is the fixed-point residual $f(z,x) = \|z - T(z,x)\|_2$.
Here, $h_\theta : \reals^d \rightarrow \reals^n$ is a neural network with ReLU activation functions, and the warm start is computed as
\begin{equation*}
  h_\theta(x) = W_{L-1} \psi(W_{L-2} \psi(\dots \psi(W_0 x + b_0)) + b_{L-2}) + b_{L-1},
\end{equation*}
where $\psi(z) = \max(0,z)$ element-wise, the matrix in the $i$-th layer is $W_i \in \reals^{m_i \times n_i}$, and the bias term in the $i$-th layer is $b_i \in \reals^{n_i}$.
The learnable weights consist of all of the weight and bias terms in the neural network, \ie, $\theta = (W_0, b_0, \dots, W_{L-1}, b_{L-1})$.
For this approach, $\ell_\theta : \reals^n \rightarrow \reals_+$ can take either the form of the regression loss from~\Eqn~\eqref{eq:reg_loss}
% \begin{equation*}
%   f(z,x) = \|z - z^\star(\param)\|_2
% \end{equation*}
or objective loss from~\Eqn~\eqref{eq:obj_loss}.
% The fixed algorithm steps are differentiated through to update the weights $\theta$.
% For the learned warm starts framework, we set $w_0 = 0$.
% We use $J = 2L$ groups, allowing for a different prior variance for each bias term $b_i$ for $i=1, \dots, L$ and each weight matrix $W_i$ for $i=1, \dots, L$.
We partition the weights into $J=2L$ groups corresponding to each bias and weight term in each layer.
% We partition the weights into $J=2L$ groups, allowing for a different prior variance for each bias term $b_i$ for $i=0, \dots, L-1$ and each weight matrix $W_i$ for $i=1, \dots, L$.
We set the prior means to be zero for all of the weights.

% \blue{B: is to zero necessary here?}
\paragraph{Strengthening the bounds.}
As in the case of classical optimizers, one downside of our approach is that the bound on the expected risk cannot reach exactly zero (even for a very large number of iterations) due to a non-zero regularization term.
For the L2WS framework specifically, we bypass this problem and show how the expected risk can be bounded to zero with high probability.
We first bound the distance from the warm start to optimality $\dist_{\fix T_x}(h_\theta(x))$ with the following theorem.
\begin{theorem}\label{thm:l2ws}
  % Assume access to an $L$-layer stochastic neural network with mean weights $w = (W_1, b_1, \dots, W_L, b_L)$ and variance $s = (\Sigma_1, \sigma_1, \dots, \Sigma_L, \sigma_L)$.
  Let $w = (W_0, b_0, \dots, W_{L-1}, b_{L-1})$ and $s = (\Sigma_0, \sigma_0, \dots, \Sigma_{L-1}, \sigma_{L-1})$ be the mean and variance terms of the weights of an $L$-layer stochastic neural network.
  Let $\bar{x}$ and $\bar{z}$ be upper bounds on $\|x\|_2$ and $\|z^\star(x)\|_2$ for any $x$ drawn from the distribution $\mathcal{X}$.
  Let $a^\star_0 = \bar{x}$, 
  \begin{equation*}
    a^\star_{i+1} = \left(\|W_{i}\|_2 + \|b_{i}\|_2 + v_i \sqrt{2 (m_{i} + n_{i} + 1) \log((L - \delta) / (2Lh))}\right) (a^\star_{i}+1), \quad \tilde{\Sigma}_i = \begin{bmatrix}
      \Sigma_i \\
      \sigma_i^T
    \end{bmatrix},
  \end{equation*}
  for $i=0,\dots,L-1$, where $v_i^2 = \max \{\max_j\|(\tilde{\Sigma}_i)^{1/2}_{j:}\|^2_2, \max_k\|(\tilde{\Sigma}_i)^{1/2}_{:k}\|^2_2\}$.
  Then with probability at least $1 - \delta$ the following bound holds for any $x$ drawn from the distribution $\mathcal{X}$:
  \begin{equation*}
    \dist_{\fix T_x}(h_\theta(x)) \leq \bar{z} + a^\star_L.
  \end{equation*}
\end{theorem}
See Appendix \ref{proof:l2wsproof} for the proof.
This upper bound on the distance from the warm start to an optimal solution given by $\bar{z} + a^\star_L$ can be easily input into inequalities~\eqref{eq:lin_conv_rate} and~\eqref{eq:avg_rate} to bound the fixed-point residual for a given number of iterations.
To see this, recall that inequalities~\eqref{eq:lin_conv_rate} and~\eqref{eq:avg_rate} include a term $\|z^0(x) - z^\star(x)\|_2$ and that the bounds hold for \emph{any} $z^\star(x) \in \fix T_x$.
Therefore, replacing these terms with $\dist_{\fix T_x}(h_\theta(x))$ yields valid inequalities on the fixed-point residual.
% \blue{B: I would add a bit more details (maybe one more phrase) as those equations are not in terms of $\dist_{\fix T_x}$.}

\paragraph{Unconstrained quadratic optimization.}
% We first consider a stylized example to illustrate why unrolling fixed-point steps can significantly improve over a decoupled approach, where $k=0$.
The learned warm starts example we consider is an unconstrained quadratic optimization problem
\begin{equation*}
  \begin{array}{ll}
  \label{prob:gd_example}
  \mbox{minimize} & (1 / 2)z^T P z + c^T z,
  \end{array}
\end{equation*}
where $P \in \symm^n_{++}$, and $c \in \reals^n$ are the problem data and $z \in \reals^n$ is the decision variable.
The parameter is $x = c$ and the fixed-point algorithm is gradient descent.

\paragraph{Numerical example.}
We take the first example, from~\citet{l2ws} where $n=20$, and the neural network has a single hidden layer with $10$ neurons.
Let $P \in \symm_{++}^{n}$ be a diagonal matrix where the first $10$ diagonals take the value $100$ and the last ten take the value of $1$.
Let $x = c \in \reals^n$.
Here, the $i$-th index of $x$ is sampled according to the uniform distribution $\mu_i \mathcal{U}[-10,10]$, where $\mu_i = 10000 \text{ if } i \leq 10 \text{ else } 1$.
% The idea is that the first $10$ indices of the optimal solution $z^\star(\theta)$ vary much more than the last $10$, but the first $10$ indices of $z$ will converge much faster.
We pick $K=15$ steps for training and use the fixed-point residual loss.
We calibrate with $1000$ Monte Carlo samples of the weights.

\paragraph{Results.}
Figures~\ref{fig:l2ws_qp1} along with Table~\ref{tab:l2ws_qp} show the behavior of our method.
The PAC-Bayes guarantees outperform both the cold start and the nearest neighbor.

\begin{table}[!h]
  \centering
  \footnotesize
    \renewcommand*{\arraystretch}{1.0}
  \caption{\reviewChanges{The quantile results for L2WS on unconstrained QP results on the number of iterations required to reach a given tolerance.
    For different quantiles and tolerances (Tol.), we compare the cold start and nearest neighbor empirical performances against our learned warm starts for which we report the empirical (Emp.) quantile and the bound (Bnd.).}
  }
  \label{tab:l2ws_qp}
  \vspace*{-3mm}
  \begin{tabular}{c}
  \end{tabular}
  \reviewChanges{
  \adjustbox{max width=.52\textwidth}{
    \begin{tabular}{cccccc}
      \midrule
    Quantile&
    Tol.&
    \begin{tabular}{@{}c@{}}Cold \\ Start\end{tabular}&
    \begin{tabular}{@{}c@{}}Nearest \\ Neighbor\end{tabular}
    &\begin{tabular}{@{}c@{}}L2WS \\ Emp.\end{tabular}
    &\begin{tabular}{@{}c@{}}L2WS \\ Bnd.\end{tabular}\\
    \midrule
    \begin{filecontents*}{./data/unconstrained_qp.csv}
quantile,tol,cold_start,nearest_neighbor,l2ws_emp,l2ws_bound
30,0.01,280,234,1,1
,0.001,509,463,169,206
,0.0001,738,692,398,435
80,0.01,300,258,1,31
,0.001,529,487,207,260
,0.0001,758,716,436,490
90,0.01,304,264,1,158
,0.001,533,493,221,388
,0.0001,763,723,450,617

\end{filecontents*}
\csvreader[head to column names, late after line=\\]{./data/unconstrained_qp.csv}{
    quantile=\colQ,
    tol=\colA,
    cold_start=\colB,
    nearest_neighbor=\colC,
    l2ws_emp=\colD,
    l2ws_bound=\colE,
    }{\colQ & \colA & \colB & \colC &\colD &\colE}
    \bottomrule
  \end{tabular}}}
\end{table}

\begin{figure}[!h]
  \centering
    \includegraphics[width=\linewidth]{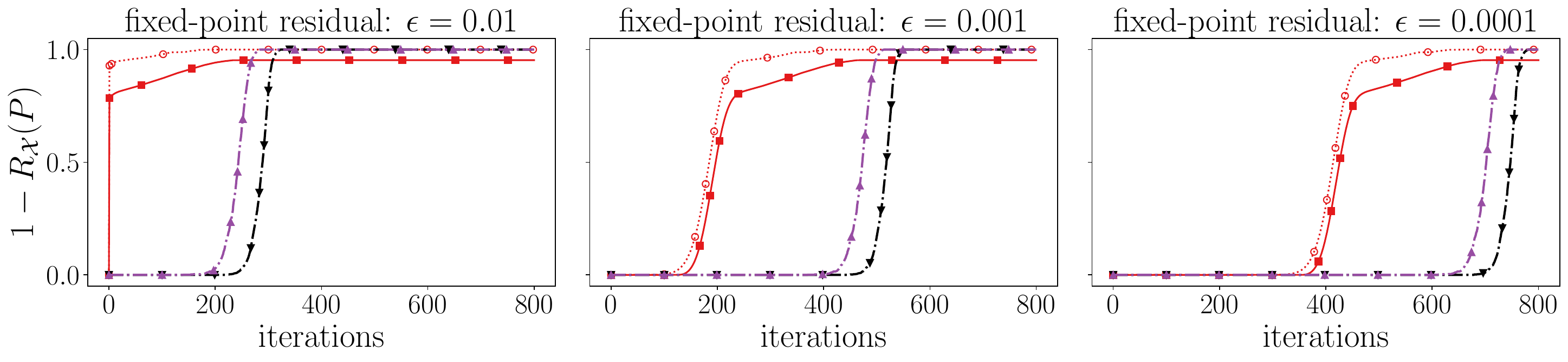}
    \includegraphics[width=\linewidth]{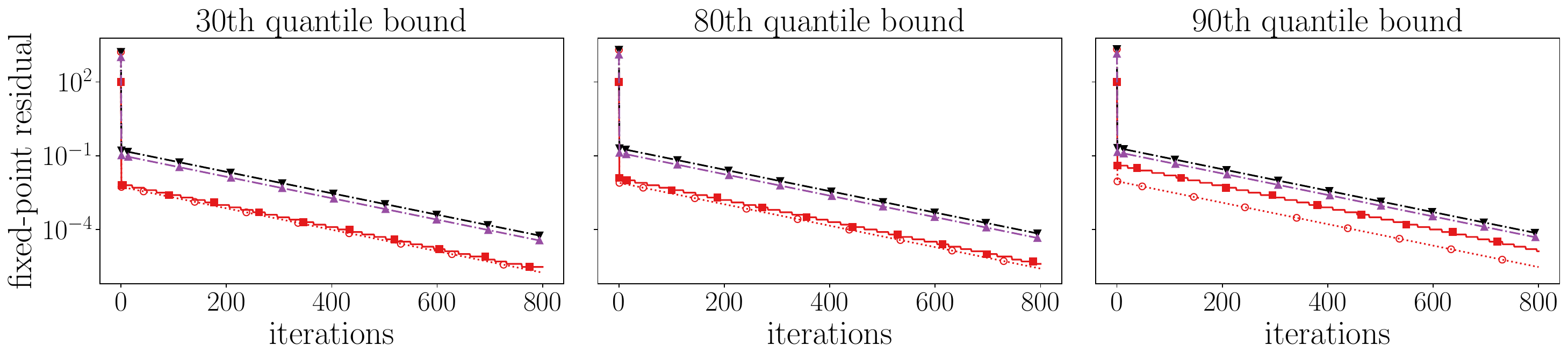}
    \lwslegend
    \caption{\reviewChanges{L2WS unconstrained QP fixed-point residual results.
    Top: lower bounds on the success rate.
    Bottom: upper bounds on the quantiles.
      % Plotting the fraction of test problems that reach various tolerances in terms of the fixed-point residual and the associated PAC-Bayes guarantee.
      The PAC-Bayes bound is very close to the empirical curve and outperforms both the cold start and the nearest neighbor curves.}
    }
    \label{fig:l2ws_qp1}
\end{figure}

\subsubsection{Model-agnostic meta-learning}\label{subsec:maml}
% In this subsection, we present our results for the Model-Agnostic Meta-Learning (MAML) approach~\citep{maml}, aimed at learning a model that quickly generalize to new tasks from minimal training examples.
In this subsection, we apply our method to obtain generalization guarantees for the MAML framework~\citep{maml}, which aims to learn a model that quickly generalize to new tasks from minimal training examples.
Each task $\mathcal{T}$ is associated with a dataset $\mathcal{D}$, split into two disjoint sets: the training set $\mathcal{D}^{\rm train}$ and the test set $\mathcal{D}^{\rm test}$.
The dataset $\mathcal{D}^{\rm train}$ consists of $K^{\rm train}$ input-output pairs $\{a_i, y_i\}_{i=1}^{K^{\rm train}}$.
Similarly, $\mathcal{D}^{\rm test}$ consists of $K^{\rm test}$ input-output pairs.
At its heart, MAML seeks to optimize a model's initial parameters $\theta$, so it can quickly adapt to unseen tasks.
MAML fits into the learning to optimize framework from \Sec~\ref{subsec:mechanics_l2o} where the parameter is the training set, \ie, $x = \mathcal{D}^{\rm train}$.
The pre-defined (\ie, not learned) update function is a step in the direction of the negative of the gradient of the loss over the training set $\mathcal{D}^{\rm train}$, \ie, 
\begin{equation*}
  z^{k+1}(x) = z^k(x)-\gamma\nabla_z \mathcal{L}(z^k(x), x).
\end{equation*}
Here, $\gamma$ is a pre-determined positive number indicating the step size.
% Note that the update function is pre-defined rather than learned.
MAML learns the initial parameters $h_\theta(x) = \theta$, which is shared across tasks.
% This approach differs from task-specific initialization schemes like the learned warm starts framework~\citep{l2ws_l4dc,l2ws}.
% MAML learns an initial iterate 
% \begin{equation*}
%   h_\theta(x) = \theta
% \end{equation*}
% that is shared across all of the tasks (unlike the learned warm starts).
The loss for the learned optimizer is computed on the test set and is calculated as
\begin{equation*}
  \ell_\theta(x) = \mathcal{L}(\hat{z}_\theta(x), \mathcal{D}^{\rm test}).
\end{equation*}
We consider regression tasks where $\mathcal{L}$ gives the mean squared error (MSE) in a dataset $\mathcal{D}$:
\begin{equation*}
\mathcal{L}(z, \mathcal{D}) = \frac{1}{|\mathcal{D}|} \sum_{i=1}^{|\mathcal{D}|} (g_z(a_i) - y_i)_2^2.
\end{equation*}
Here, $g_z$ is the neural network predictor with weights $z$.
% The test loss follows the same formula but is applied to the test dataset $\mathcal{D}^{\rm test}$.
We partition the weights into $2L$ groups as in \Sec~\ref{subsec:l2ws}.
We set the prior means to be zero for all of the weights.

\paragraph{Sinusoid curves.}
% In this example, each task involves regressing the input to the output of a sine wave as in~\citep{maml}.
We consider the meta-learning task of regressing inputs to outputs of sine waves using a few datapoints as in~\citet{maml}.
We generate each task by first sampling an amplitude $A$ and a phase $b$.
We then generate the datasets $\mathcal{D}^{\rm train}$ and $\mathcal{D}^{\rm test}$ in the following manner.
% The datasets $\mathcal{D}^{\rm train}$ and $\mathcal{D}^{\rm test}$ are generated by first sampling an amplitude $A$ and a phase $b$.
The inputs $a$ are uniformly sampled from an interval, and the corresponding outputs are given by $y = A \sin(a - b)$.
% Using the same amplitude and phase we generate samples to make up $\mathcal{D}^{\rm train}$ and $\mathcal{D}^{\rm test}$
% We use the same setup as in~\citep{maml}.
The neural network consists of two hidden layers of size $40$ each with ReLU activations.

\paragraph{Task-specific metric.}
To help visualize our results, we consider the task-specific metric that is the $\ell_{\infty}$ norm of the errors over the dataset $\mathcal{D}^{\rm test} = \{(a_i, y_i)\}_{i=1}^{K^{\rm test}}$:
\begin{equation}\label{eq:inf_norm}
  \max_{i=1, \dots, K^{\rm test}} |g_z(a_i) - y_i|.
\end{equation}
% This metric allows us to clearly visualize our results.
% In Figure~\ref{fig:maml_visuals} we visualize them.

\paragraph{Numerical example.}
We follow the exact setup from~\citet{maml}.
For each task $\mathcal{T}$, we sample an amplitude $A$ from the uniform distribution $\mathcal{U}[0.1, 5.0]$ and a phase from the uniform distribution $\mathcal{U}[0, \pi]$.
All of the $a$ datapoints are sampled i.i.d. from the uniformly from $[-5.0, 5.0]$.
We pick the number of datapoints in the training and test sets to be $K^{\rm train} = 5$ and $K^{\rm test} = 100$ respectively.
The step size $\gamma$ is $0.01$, and we unroll $2$ steps during training.
We calibrate with $20000$ Monte Carlo samples of the weights.

\paragraph{Results.}
Figures~\ref{fig:maml_results1} and~\ref{fig:maml_quantiles} along with Table~\ref{tab:maml} show the behavior of our method with two unrolled steps.
In this example, the baseline that we compare against is the pretrained model from~\citet{maml} which trains the network on the sinusoid curves without unrolling any algorithm steps. 
For both metrics, our bounds are much stronger than the pretrained model.
We visualize our results in Figure~\ref{fig:maml_visuals}.
% This metric facilitates a clear visualization of our results, as depicted in Figure~\ref{fig:maml_visuals}.
We obtain probabilistic guarantees that the solution returned after $10$ steps initialized with MAML will fall within a band of width two centered around the true sine surve.
The deterministic pretrained model fails to completely fall within the band of error for many of the problems.
% The PAC-Bayes guarantees ensure that probability at least $0.99$, our trained MAML produces a stochastic curve that remains entirely in the banded region $90 \%$ of the time after $10$ steps.
%     In contrast, the pretrained model produces a deterministic curve that only entirely lies in the banded region around $33 \%$ of the time.

\begin{figure}[!h]
  \centering
    \includegraphics[width=\linewidth]{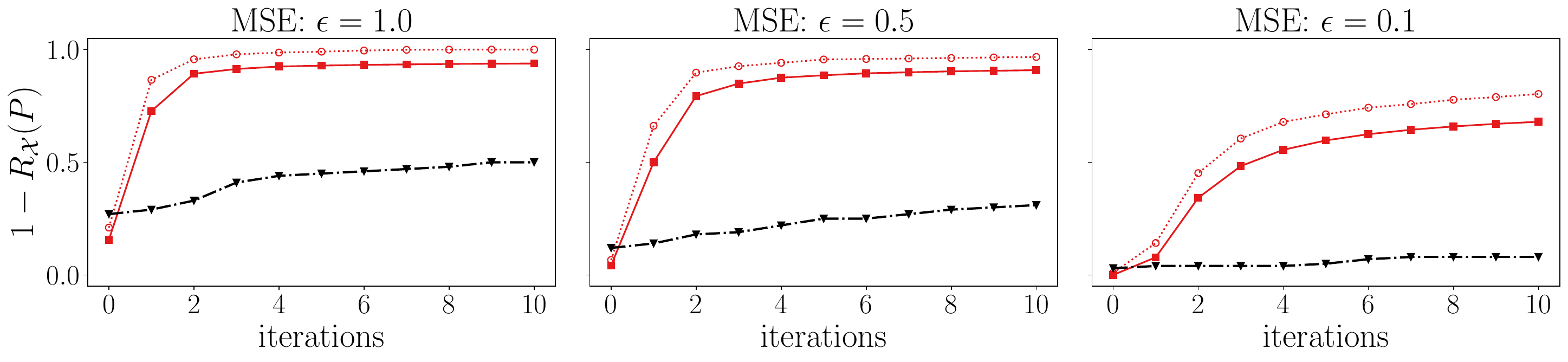}
    \includegraphics[width=\linewidth]{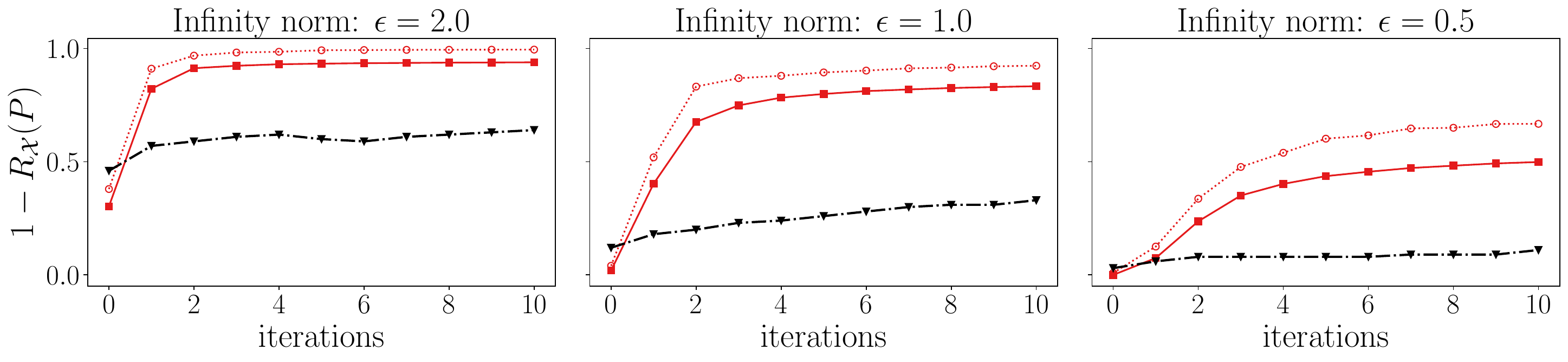}
  \centering
  \mamllegend
  \caption{\reviewChanges{MAML success rate results for sinusoid curves.
  Top: MSE.
  Bottom: infinity norm from \Eqn~\eqref{eq:inf_norm}.
    Our lower bounds on the success rate $1 - \exprisk(P)$ for both metrics are much higher than the empirical success rate of the pretrained model across many tolerances.}
  }
  \label{fig:maml_results1}
\end{figure}

\begin{figure}[!h]
  \centering
    \includegraphics[width=\linewidth]{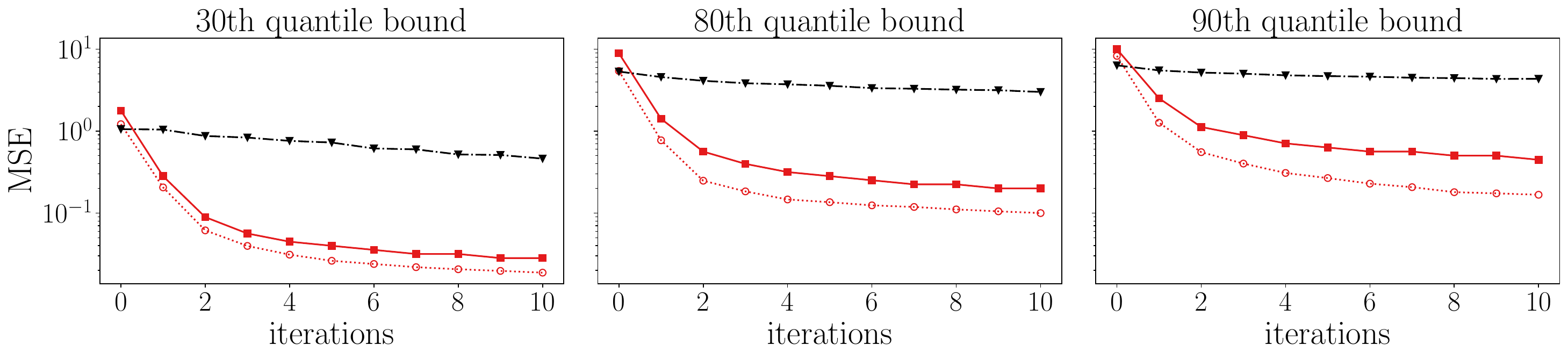}
    \includegraphics[width=\linewidth]{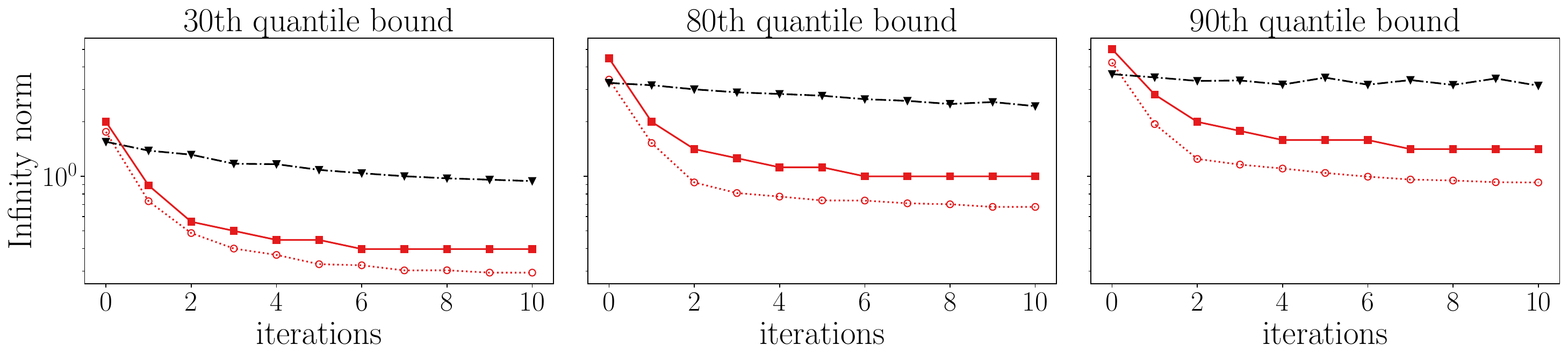}
  \centering
  \mamllegend
  \caption{\reviewChanges{
    MAML quantile results for sinusoid curves.
    %  for the MSE and the infinity norm from \Eqn~\eqref{eq:inf_norm}.
    Top: MSE.
    Bottom: infinity norm from \Eqn~\eqref{eq:inf_norm}.
  For the quantiles $30$, $80$, and $90$, our MAML upper bounds are significantly lower than the pretrained empirical curve after a few iterations.}
  }
  \label{fig:maml_quantiles}
\end{figure}

\begin{table}[!h]
  \centering
  \footnotesize
    \renewcommand*{\arraystretch}{1.0}
  \caption{\reviewChanges{The quantile results for MAML on sinusoidal regression tasks after $10$ iterations for both the mean square error (MSE) and the infinity norm.
  Since the expected risk is never bounded to a value below $0.05$, we cannot provide guarantees for the $95$th quantile.}
  % We report the empirical average of test problems (Emp.) and bounds (Bnd.) for MAML.
  }
  \label{tab:maml}
  \vspace*{-3mm}
  \begin{tabular}{l}
  \end{tabular}
  \reviewChanges{
  \adjustbox{max width=\textwidth}{
    \begin{tabular}{cccc}
      \midrule
      \multicolumn{4}{c}{MSE}\\
      \midrule
    Quantile&
    Pretrained&
    \multicolumn{2}{c}{MAML}
    \\
    {} & {} & Emp. & Bnd.\\
    \midrule
    % Add directly from csv reader
    % the csv file has names colnameA, colnameB
    \begin{filecontents*}{./data/maml.csv}
,quantiles,pretrain,maml_emp,maml_bound
0,10,0.496,0.207,0.251
1,30,0.942,0.295,0.398
2,50,1.429,0.391,0.562
3,60,1.829,0.452,0.631
4,80,2.434,0.679,1.0
5,90,3.155,0.925,1.413
6,95,3.849,1.191,-
\end{filecontents*}
\csvreader[head to column names, late after line=\\]{./data/maml.csv}{
    quantiles=\colA,
    pretrain=\colB,
    maml_emp=\colC,
    maml_bound=\colD,
    }{\colA & \colB & \colC &\colD}
    \bottomrule
  \end{tabular}}}
  \quad \quad 
  \reviewChanges{
  \adjustbox{max width=\textwidth}{
    \begin{tabular}{cccc}
      \midrule
      \multicolumn{4}{c}{Infinity norm}\\
      \midrule
    Quantile&
    Pretrained&
    \multicolumn{2}{c}{MAML}
    \\
    {} & {} & Emp. & Bnd.\\
    \midrule
    % Add directly from csv reader
    % the csv file has names colnameA, colnameB
    \begin{filecontents*}{./data/maml_inf.csv}
,quantiles,pretrain,maml_emp,maml_bound
0,10,0.153,0.009,0.014
1,30,0.461,0.019,0.028
2,50,0.972,0.035,0.056
3,60,1.429,0.051,0.079
4,80,3.004,0.1,0.2
5,90,4.342,0.167,0.447
6,95,6.6,0.351,-
\end{filecontents*}
\csvreader[head to column names, late after line=\\]{./data/maml_inf.csv}{
    quantiles=\colA,
    pretrain=\colB,
    maml_emp=\colC,
    maml_bound=\colD,
    }{\colA & \colB & \colC &\colD}
    \bottomrule
  \end{tabular}}}
\end{table}

\begin{figure}[!h]
  \begin{subfigure}[t]{.33\linewidth}
    \centering
    \includegraphics[width=\linewidth]{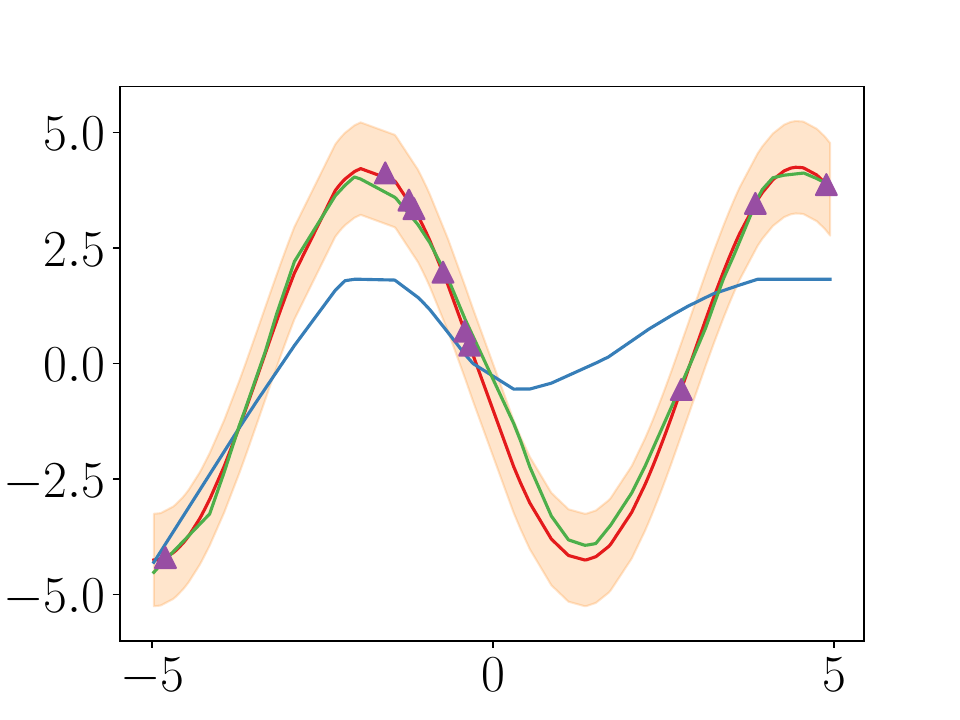}
  \end{subfigure}%
  \begin{subfigure}[t]{.33\linewidth}
    \centering
    \includegraphics[width=\linewidth]{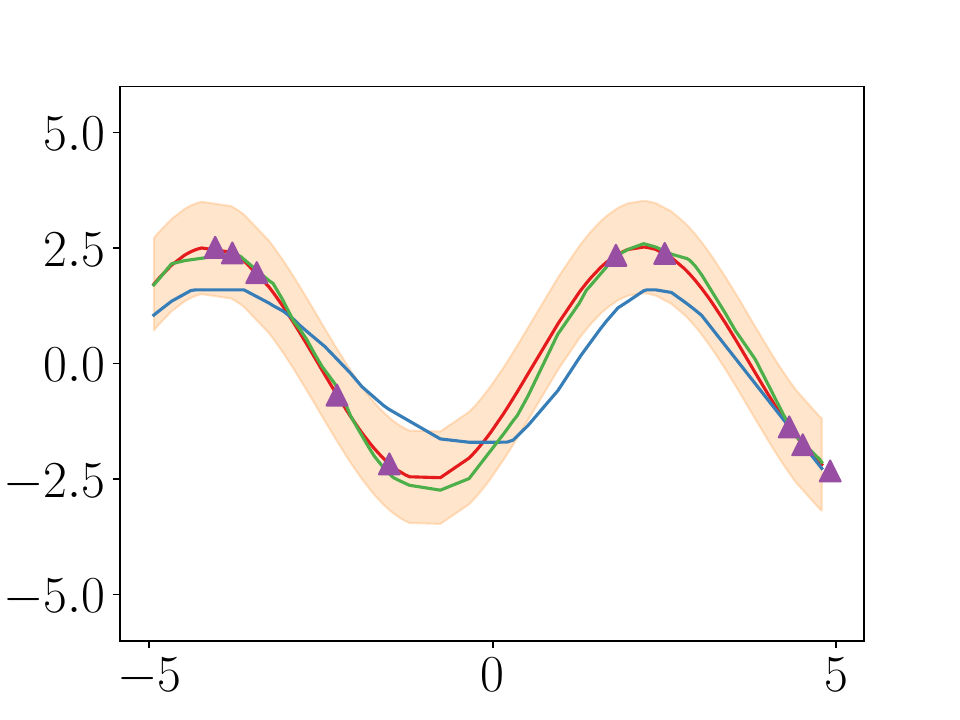}
  \end{subfigure}%
  \begin{subfigure}[t]{.33\linewidth}
    \centering
    \includegraphics[width=\linewidth]{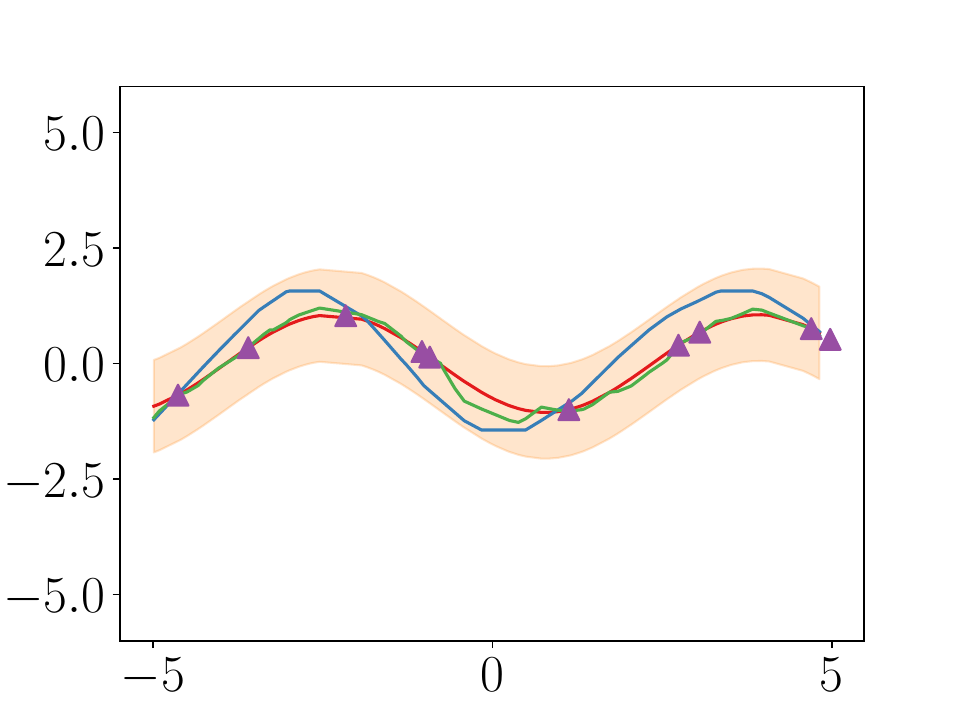}
  \end{subfigure}%
  \\
  \begin{subfigure}[t]{.33\linewidth}
    \centering
    \includegraphics[width=\linewidth]{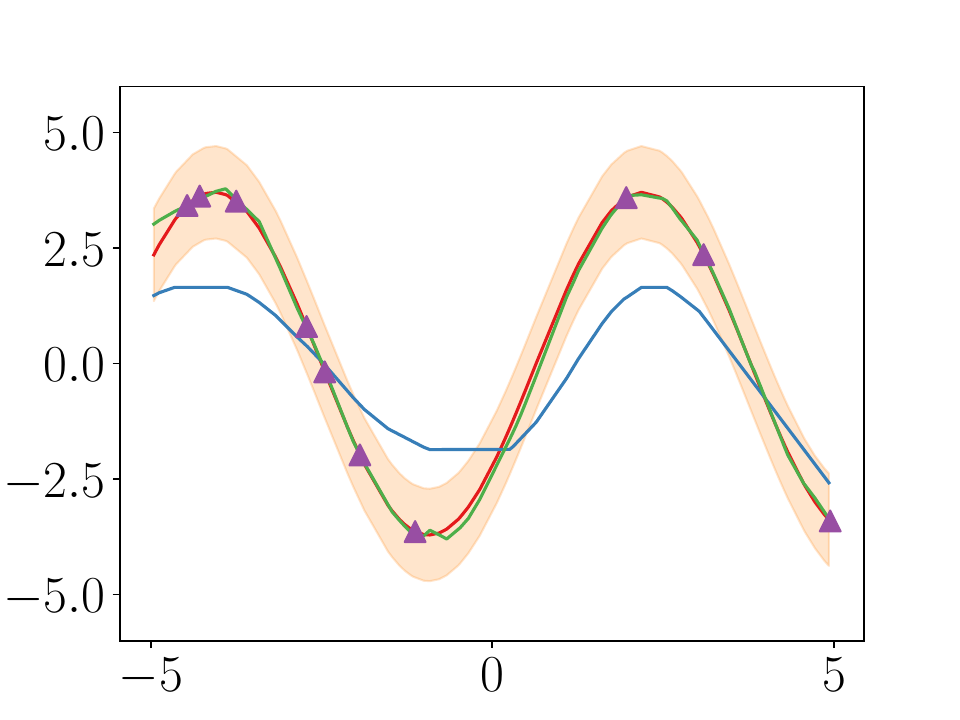}
  \end{subfigure}%
  \begin{subfigure}[t]{.33\linewidth}
    \centering
    \includegraphics[width=\linewidth]{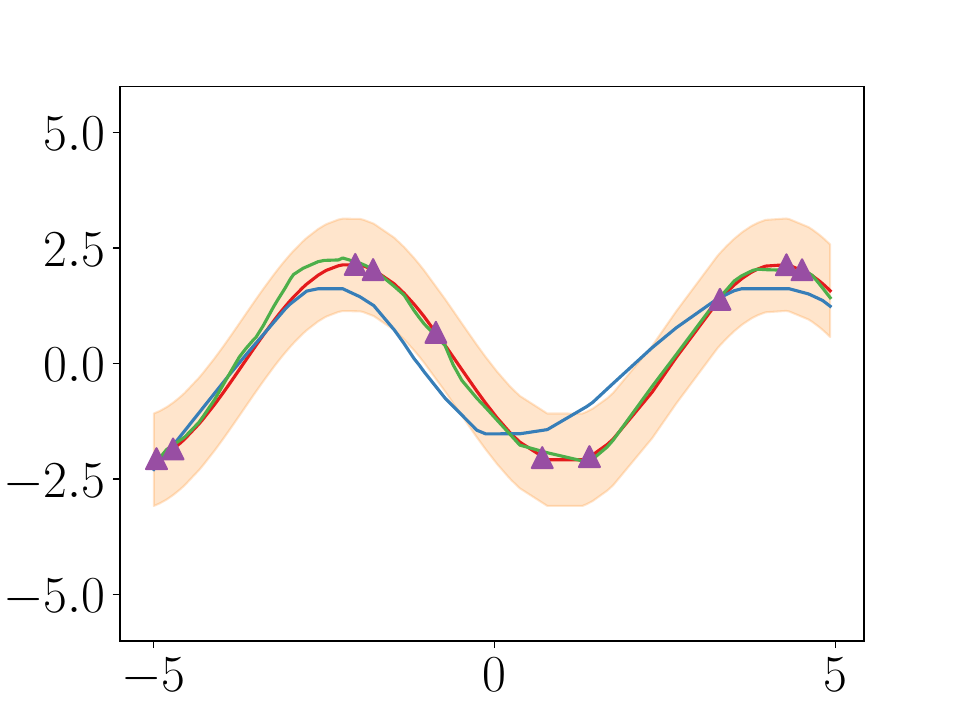}
  \end{subfigure}%
  \begin{subfigure}[t]{.33\linewidth}
    \centering
    \includegraphics[width=\linewidth]{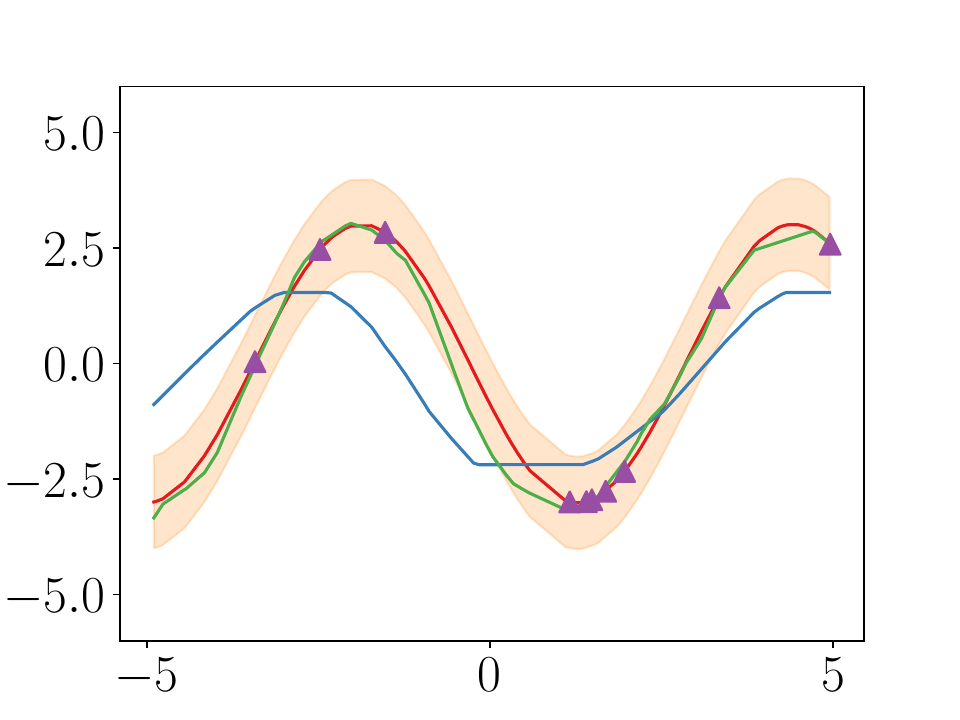}
  \end{subfigure}%
  \\
  \begin{subfigure}[t]{.33\linewidth}
    \centering
    \includegraphics[width=\linewidth]{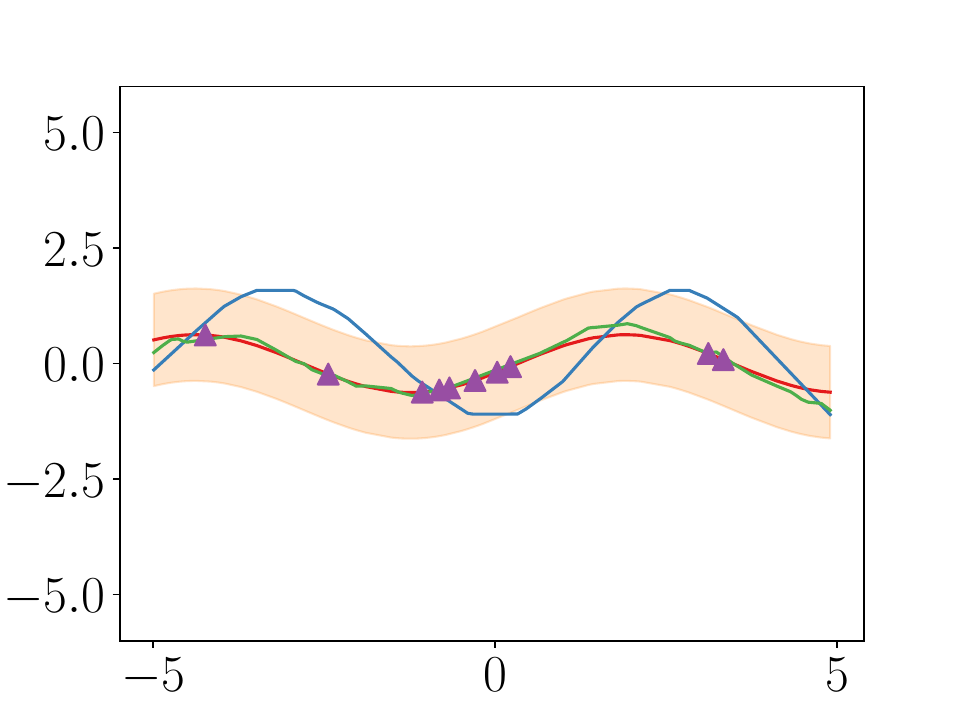}
  \end{subfigure}%
  \begin{subfigure}[t]{.33\linewidth}
    \centering
    \includegraphics[width=\linewidth]{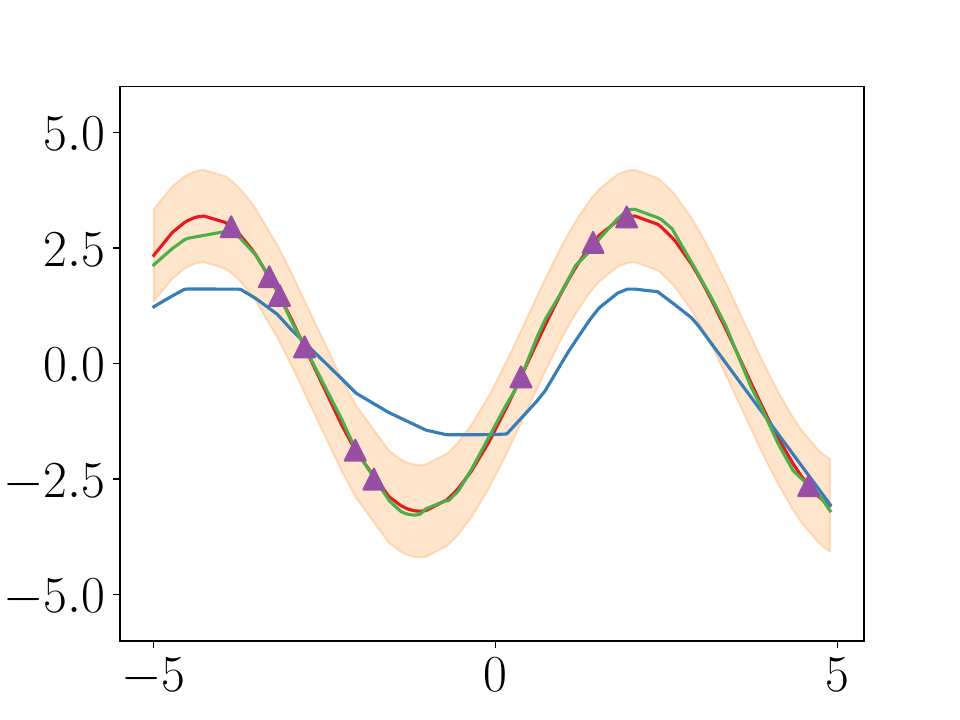}
  \end{subfigure}%
  \begin{subfigure}[t]{.33\linewidth}
    \centering
    \includegraphics[width=\linewidth]{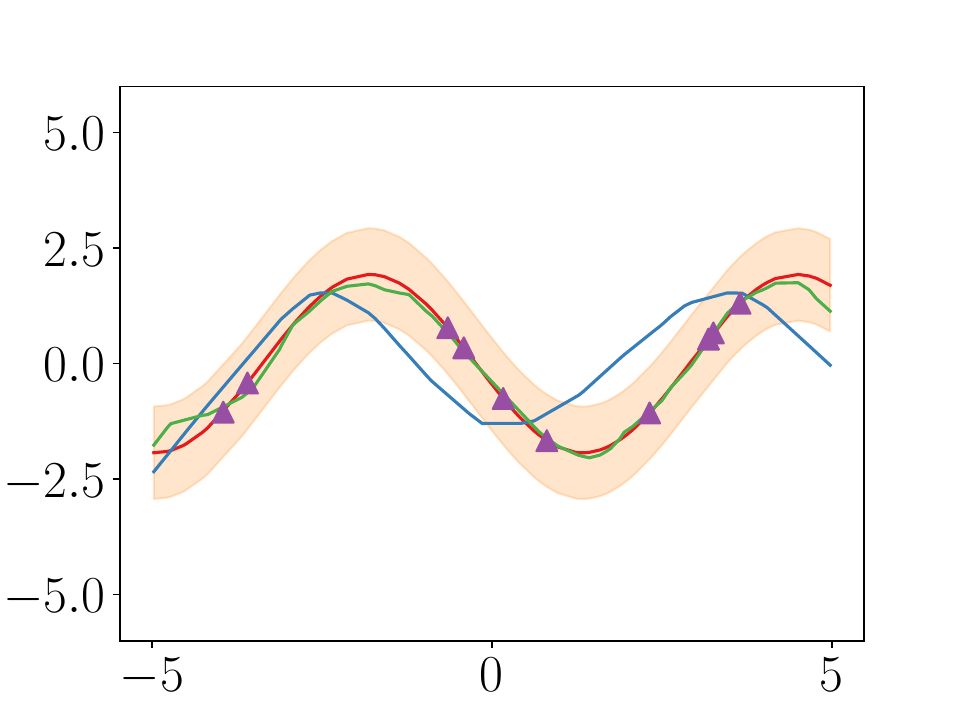}
  \end{subfigure}%
  \centering
  \mamlvisualslegend
  \caption{
    \reviewChanges{
    MAML visualizations for regressing on sine curves.
    The purple triangles are the $K^{\rm train}$ datapoints used for computing the gradients.
    With high probability, after $10$ steps, MAML is guaranteed to produce a curve that remains entirely in the banded region $81 \%$ of the time, while the pretrained model produces a curve that only entirely lies in the banded region around $33 \%$ of the time.}
    % The evaluation is used on $K^{\rm test} = 100$ datapoints for the entire curve.
    % The PAC-Bayes guarantees ensure that probability at least $0.99$, our trained MAML produces a stochastic curve that remains entirely in the banded region $90 \%$ of the time after $10$ steps.
    % In contrast, the pretrained model produces a deterministic curve that only entirely lies in the banded region around $33 \%$ of the time.
    }
  \label{fig:maml_visuals}
\end{figure}

\section{Conclusion}\label{sec:conclusion}
We present a data-driven framework to provide guarantees for the performance of classical and learned optimizers in the setting of parametric optimization.
For classical optimizers, we provide strong guarantees using a sample convergence bound.
For learned optimizers, we provide generalization guarantees using the PAC-Bayes framework and a learning algorithm designed to optimize these guarantees.
We showcase the effectiveness of our approach for both classical and learned optimizers on many examples including ones from control, signal processing, and meta-learning.
% \red{We see several directions for interesting future research: for classical optimizers, adapting our bounds to the setting where the i.i.d.\ assumption does not hold; for learned optimizers, investigating methods to scale our approach to tackle larger-scale problems.}

\reviewChanges{
We see a few directions for interesting future research.
The bounds in this paper all required the problem parameters to be drawn of a distribution in an i.i.d. fashion.
While in many applications this assumption is common, \eg, in sparse coding, or machine learning problems, there are other applications where this assumption is not met.
For example, in control problems, where the problems need to be solved sequentially.
Extending our work beyond the i.i.d.\ assumption is an avenue for future work.
Another avenue is to scale our approach to tackle larger-scale problems for learned optimizers.
}

% We see several future research directions.
% Our study opens several avenues for future research,
% For classical optimizers, we include tackling settings where the i.i.d. assumption is not valid.
% For learned optimizers, we include investigating methods to scale our method to larger instances.

% \red{B: conclusions. We could mention that this approach could provide guarantees for lots of model-free frameworks (\eg, LSTM)}

\acks{
  We thank Ernest Ryu, Anirudha Majumdar, Jingyi Huang, and two anonymous reviewers for helpful and detailed comments that improved the quality of this work.
  Bartolomeo Stellato and Rajiv Sambharya are supported by the NSF CAREER Award ECCS-223977. 
  Bartolomeo Stellato is also supported by the ONR YIP Award N000142512147.
  We are pleased to acknowledge that the work reported on in this paper was substantially performed using the Princeton Research Computing resources at Princeton University which is consortium of groups led by the Princeton Institute for Computational Science and Engineering (PICSciE) and Office of Information Technology's Research Computing.
}

\newpage
% \bibliography{bibliography}
\bibliography{bibliographynourl}

\begin{thebibliography}{104}
\providecommand{\natexlab}[1]{#1}
\providecommand{\url}[1]{\texttt{#1}}
\expandafter\ifx\csname urlstyle\endcsname\relax
  \providecommand{\doi}[1]{doi: #1}\else
  \providecommand{\doi}{doi: \begingroup \urlstyle{rm}\Url}\fi

\bibitem[Almeida et~al.(2021)Almeida, Winter, Tang, and Zaremba]{almeida2021generalizable}
D.~Almeida, C.~Winter, J.~Tang, and W.~Zaremba.
\newblock A generalizable approach to learning optimizers.
\newblock \emph{arXiv preprint arXiv:2106.00958}, 2021.

\bibitem[Alquier(2023)]{pac_bayes_intro}
P.~Alquier.
\newblock User-friendly introduction to {PAC}-{Bayes} bounds.
\newblock \emph{arXiv e-prints}, 2023.

\bibitem[Amit and Meir(2018)]{pmlr-v80-amit18a}
R.~Amit and R.~Meir.
\newblock Meta-learning by adjusting priors based on extended {PAC}-{B} ayes theory.
\newblock In \emph{International Conference on Machine Learning}, pages 205--214, 2018.

\bibitem[Amos(2023)]{amos_tutorial}
B.~Amos.
\newblock Tutorial on amortized optimization.
\newblock \emph{Foundations and Trends in Machine Learning}, 16\penalty0 (5):\penalty0 592--732, 2023.

\bibitem[Andrychowicz et~al.(2016)Andrychowicz, Denil, Colmenarejo, Hoffman, Pfau, Schaul, Shillingford, and de~Freitas]{learn_learn_gd_gd}
M.~Andrychowicz, M.~Denil, S.~G. Colmenarejo, M.~W. Hoffman, D.~Pfau, T.~Schaul, B.~Shillingford, and N.~de~Freitas.
\newblock Learning to learn by gradient descent by gradient descent.
\newblock In \emph{Neural Information Processing Systems}, 2016.

\bibitem[Bai et~al.(2022)Bai, Koltun, and Kolter]{bai2022neural}
S.~Bai, V.~Koltun, and J.~Z. Kolter.
\newblock Neural deep equilibrium solvers.
\newblock In \emph{International Conference on Learning Representations}, 2022.

\bibitem[Baker(2019)]{warm_start_power_flow}
K.~Baker.
\newblock Learning warm-start points for ac optimal power flow.
\newblock In \emph{IEEE International Workshop on Machine Learning for Signal Processing (MLSP)}, 2019.

\bibitem[Balatsoukas-Stimming and Studer(2019)]{deep_unfolding_wireless}
A.~Balatsoukas-Stimming and C.~Studer.
\newblock Deep unfolding for communications systems: A survey and some new directions.
\newblock In \emph{2019 IEEE International Workshop on Signal Processing Systems (SiPS)}, pages 266--271, 2019.

\bibitem[Balcan(2020)]{balcan2020data}
M.-F. Balcan.
\newblock Data-driven algorithm design.
\newblock \emph{arXiv preprint arXiv:2011.07177}, 2020.

\bibitem[Balcan et~al.(2017)Balcan, Nagarajan, Vitercik, and White]{balcan_partition}
M.-F. Balcan, V.~Nagarajan, E.~Vitercik, and C.~White.
\newblock Learning-theoretic foundations of algorithm configuration for combinatorial partitioning problems.
\newblock In \emph{Conference on Learning Theory}, pages 213--274. PMLR, 2017.

\bibitem[Balcan et~al.(2018)Balcan, Dick, and White]{balcan_cluster}
M.-F. Balcan, T.~Dick, and C.~White.
\newblock Data-driven clustering via parameterized lloyd's families.
\newblock \emph{Advances in Neural Information Processing Systems}, 31, 2018.

\bibitem[Balcan et~al.(2019)Balcan, Khodak, and Talwalkar]{pmlr-v97-balcan19a}
M.-F. Balcan, M.~Khodak, and A.~Talwalkar.
\newblock Provable guarantees for gradient-based meta-learning.
\newblock In \emph{International Conference on Machine Learning}, pages 424--433, 2019.

\bibitem[Balcan et~al.(2021)Balcan, DeBlasio, Dick, Kingsford, Sandholm, and Vitercik]{balcan_gen_guarantees}
M.-F. Balcan, D.~DeBlasio, T.~Dick, C.~Kingsford, T.~Sandholm, and E.~Vitercik.
\newblock How much data is sufficient to learn high-performing algorithms? generalization guarantees for data-driven algorithm design.
\newblock In \emph{Proceedings of the 53rd Annual ACM SIGACT Symposium on Theory of Computing}, pages 919--932, 2021.

\bibitem[Banert et~al.(2021)Banert, Rudzusika, {\"O}ktem, and Adler]{banert2021accelerated}
S.~Banert, J.~Rudzusika, O.~{\"O}ktem, and J.~Adler.
\newblock Accelerated forward-backward optimization using deep learning.
\newblock \emph{arXiv preprint arXiv:2105.05210}, 2021.

\bibitem[Banjac et~al.(2019)Banjac, Goulart, Stellato, and Boyd]{infeas_detection}
G.~Banjac, P.~Goulart, B.~Stellato, and S.~Boyd.
\newblock Infeasibility detection in the alternating direction method of multipliers for convex optimization.
\newblock \emph{Journal of Optimization Theory and Applications}, 183, 2019.

\bibitem[Bartlett et~al.(2022)Bartlett, Indyk, and Wagner]{bartlett2022generalization}
P.~L. Bartlett, P.~Indyk, and T.~Wagner.
\newblock Generalization bounds for data-driven numerical linear algebra.
\newblock In \emph{Conference on Learning Theory}, volume 178, pages 2013--2040, 2022.

\bibitem[Beck(2017)]{fom_book}
A.~Beck.
\newblock \emph{First-Order Methods in Optimization}.
\newblock Society for Industrial and Applied Mathematics, Philadelphia, PA, 2017.

\bibitem[Beck and Teboulle(2009)]{fista}
A.~Beck and M.~Teboulle.
\newblock A fast iterative shrinkage-thresholding algorithm with application to wavelet-based image deblurring.
\newblock In \emph{2009 IEEE International Conference on Acoustics, Speech and Signal Processing}, pages 693--696, 2009.

\bibitem[Bertsimas and Stellato(2022)]{online_milliseconds}
D.~Bertsimas and B.~Stellato.
\newblock Online {Mixed}-{Integer} {Optimization} in {Milliseconds}.
\newblock \emph{INFORMS Journal on Computing}, 34\penalty0 (4):\penalty0 2229--2248, 2022.

\bibitem[Blondel et~al.(2021)Blondel, Berthet, Cuturi, Frostig, Hoyer, Llinares-L{\'o}pez, Pedregosa, and Vert]{jaxopt_implicit_diff}
M.~Blondel, Q.~Berthet, M.~Cuturi, R.~Frostig, S.~Hoyer, F.~Llinares-L{\'o}pez, F.~Pedregosa, and J.-P. Vert.
\newblock Efficient and modular implicit differentiation.
\newblock \emph{arXiv preprint arXiv:2105.15183}, 2021.

\bibitem[Boley(2013)]{local_lin_conv}
D.~Boley.
\newblock Local linear convergence of the alternating direction method of multipliers on quadratic or linear programs.
\newblock \emph{SIAM Journal on Optimization}, 23\penalty0 (4):\penalty0 2183--2207, 2013.

\bibitem[Borrelli et~al.(2017)Borrelli, Bemporad, and Morari]{borrelli_mpc_book}
F.~Borrelli, A.~Bemporad, and M.~Morari.
\newblock \emph{Predictive Control for Linear and Hybrid Systems}.
\newblock Cambridge University Press, 2017.

\bibitem[Boyd et~al.(2007)Boyd, Kim, Vandenberghe, and Hassibi]{geometric_program}
S.~Boyd, S.-J. Kim, L.~Vandenberghe, and A.~Hassibi.
\newblock A tutorial on geometric programming.
\newblock \emph{Optimization and engineering}, 8:\penalty0 67--127, 2007.

\bibitem[Boyd et~al.(2011)Boyd, Parikh, wah Chu, Peleato, and Eckstein]{Boyd_admm}
S.~P. Boyd, N.~Parikh, E.~K. wah Chu, B.~Peleato, and J.~Eckstein.
\newblock Distributed optimization and statistical learning via the alternating direction method of multipliers.
\newblock \emph{Foundations and Trends in Machine Learning}, 3:\penalty0 1--122, 2011.

\bibitem[Bradbury et~al.(2018)Bradbury, Frostig, Hawkins, Johnson, Leary, Maclaurin, Necula, Paszke, Vander{P}las, Wanderman-{M}ilne, and Zhang]{jax}
J.~Bradbury, R.~Frostig, P.~Hawkins, M.~J. Johnson, C.~Leary, D.~Maclaurin, G.~Necula, A.~Paszke, J.~Vander{P}las, S.~Wanderman-{M}ilne, and Q.~Zhang.
\newblock {JAX}: composable transformations of {P}ython+{N}um{P}y programs, 2018.

\bibitem[Briden et~al.(2023)Briden, Choi, Yun, Linares, and Cauligi]{constraint_informed_traj_ws}
J.~Briden, C.~Choi, K.~Yun, R.~Linares, and A.~Cauligi.
\newblock Constraint-informed learning for warm starting trajectory optimization.
\newblock \emph{arXiv preprint arXiv:2312.14336}, 2023.

\bibitem[Chen et~al.(2018{\natexlab{a}})Chen, Saulnier, Atanasov, Lee, Kumar, Pappas, and Morari]{mpc_constrained_neural_nets}
S.~Chen, K.~Saulnier, N.~Atanasov, D.~D. Lee, V.~Kumar, G.~J. Pappas, and M.~Morari.
\newblock Approximating explicit model predictive control using constrained neural networks.
\newblock In \emph{American Control Conference}, pages 1520--1527, 2018{\natexlab{a}}.

\bibitem[Chen et~al.(2022)Chen, Chen, Chen, Heaton, Liu, Wang, and Yin]{l2o}
T.~Chen, X.~Chen, W.~Chen, H.~Heaton, J.~Liu, Z.~Wang, and W.~Yin.
\newblock Learning to optimize: A primer and a benchmark.
\newblock \emph{Journal of Machine Learning Research}, 23\penalty0 (189):\penalty0 1--59, 2022.

\bibitem[Chen et~al.(2018{\natexlab{b}})Chen, Liu, Wang, and Yin]{lista_cpss}
X.~Chen, J.~Liu, Z.~Wang, and W.~Yin.
\newblock Theoretical linear convergence of unfolded ista and its practical weights and thresholds.
\newblock \emph{Advances in Neural Information Processing Systems}, 31, 2018{\natexlab{b}}.

\bibitem[Chen et~al.(2020)Chen, Zhang, Reisinger, and Song]{reasoning_layer}
X.~Chen, Y.~Zhang, C.~Reisinger, and L.~Song.
\newblock Understanding deep architectures with reasoning layer.
\newblock In \emph{Neural Information Processing Systems}, 2020.

\bibitem[Chen et~al.(2021)Chen, Liu, Wang, and Yin]{hyperlista}
X.~Chen, J.~Liu, Z.~Wang, and W.~Yin.
\newblock Hyperparameter tuning is all you need for lista.
\newblock In \emph{Advances in Neural Information Processing Systems}, 2021.

\bibitem[Cohen et~al.(2017)Cohen, Afshar, Tapson, and van Schaik]{emnist}
G.~Cohen, S.~Afshar, J.~Tapson, and A.~van Schaik.
\newblock Emnist: an extension of mnist to handwritten letters, 2017.

\bibitem[Diamond et~al.(2017)Diamond, Sitzmann, Heide, and Wetzstein]{Diamond2017UnrolledOW}
S.~Diamond, V.~Sitzmann, F.~Heide, and G.~Wetzstein.
\newblock Unrolled optimization with deep priors.
\newblock \emph{arXiv preprint arXiv:1705.08041}, 2017.

\bibitem[Diehl et~al.(2009)Diehl, Ferreau, and Haverbeke]{nonlinear_mpc}
M.~Diehl, H.~J. Ferreau, and N.~Haverbeke.
\newblock \emph{Efficient Numerical Methods for Nonlinear MPC and Moving Horizon Estimation}.
\newblock 2009.

\bibitem[Dontchev and Rockafellar(2009)]{implicit_function}
A.~L. Dontchev and R.~T. Rockafellar.
\newblock \emph{Implicit functions and solution mappings: A view from variational analysis}, volume 616.
\newblock Springer, 2009.

\bibitem[Donti et~al.(2021)Donti, Rolnick, and Kolter]{donti2021dc3}
P.~Donti, D.~Rolnick, and Z.~Kolter.
\newblock Dc3: A learning method for optimization with hard constraints.
\newblock In \emph{International Conference on Learning Representations}, 2021.

\bibitem[Drori and Teboulle(2014)]{pep}
Y.~Drori and M.~Teboulle.
\newblock Performance of first-order methods for smooth convex minimization: a novel approach.
\newblock \emph{Mathematical Programming}, 145\penalty0 (1):\penalty0 451--482, 2014.

\bibitem[Duchi(2016)]{Duchi2016DerivationsFL}
J.~C. Duchi.
\newblock Derivations for linear algebra and optimization.
\newblock 2016.

\bibitem[Dziugaite and Roy(2017)]{nonvacuous_pac_bayes}
G.~K. Dziugaite and D.~M. Roy.
\newblock Computing nonvacuous generalization bounds for deep (stochastic) neural networks with many more parameters than training data.
\newblock \emph{arXiv preprint arXiv:1703.11008}, 2017.

\bibitem[Elad and Aharon(2006)]{learned_dictionaries_images}
M.~Elad and M.~Aharon.
\newblock Image denoising via sparse and redundant representations over learned dictionaries.
\newblock \emph{IEEE Transactions on Image Processing}, 15\penalty0 (12):\penalty0 3736--3745, 2006.

\bibitem[Farid and Majumdar(2021)]{farid_metalearning}
A.~Farid and A.~Majumdar.
\newblock Generalization bounds for meta-learning via pac-bayes and uniform stability.
\newblock In \emph{Neural Information Processing Systems}, 2021.

\bibitem[Fazlyab et~al.(2017)Fazlyab, Ribeiro, Morari, and Preciado]{Fazlyab2017AnalysisOO}
M.~Fazlyab, A.~Ribeiro, M.~Morari, and V.~M. Preciado.
\newblock Analysis of optimization algorithms via integral quadratic constraints: Nonstrongly convex problems.
\newblock \emph{SIAM Journal of Optimization}, 28:\penalty0 2654--2689, 2017.

\bibitem[Fazlyab et~al.(2022)Fazlyab, Morari, and Pappas]{fazlyabsdp}
M.~Fazlyab, M.~Morari, and G.~Pappas.
\newblock Safety verification and robustness analysis of neural networks via quadratic constraints and semidefinite programming.
\newblock \emph{IEEE Transactions on Automatic Control}, 67\penalty0 (1):\penalty0 1--15, 2022.

\bibitem[Finn et~al.(2017)Finn, Abbeel, and Levine]{maml}
C.~Finn, P.~Abbeel, and S.~Levine.
\newblock Model-agnostic meta-learning for fast adaptation of deep networks.
\newblock In \emph{International Conference on Machine Learning}, volume~70 of \emph{Proceedings of Machine Learning Research}, pages 1126--1135. PMLR, 2017.

\bibitem[Garstka et~al.(2019)Garstka, Cannon, and Goulart]{COSMO}
M.~Garstka, M.~Cannon, and P.~Goulart.
\newblock {COSMO}: A conic operator splitting method for large convex problems.
\newblock In \emph{European Control Conference}, 2019.

\bibitem[Giselsson and Boyd(2014)]{Giselsson2014LinearCA}
P.~Giselsson and S.~P. Boyd.
\newblock Linear convergence and metric selection for douglas-rachford splitting and admm.
\newblock \emph{IEEE Transactions on Automatic Control}, 62:\penalty0 532--544, 2014.

\bibitem[Goujaud et~al.(2022)Goujaud, Moucer, Glineur, Hendrickx, Taylor, and Dieuleveut]{pepit}
B.~Goujaud, C.~Moucer, F.~Glineur, J.~Hendrickx, A.~Taylor, and A.~Dieuleveut.
\newblock Pepit: computer-assisted worst-case analyses of first-order optimization methods in python.
\newblock \emph{arXiv preprint arXiv:2201.04040}, 2022.

\bibitem[Gregor and LeCun(2010)]{lista}
K.~Gregor and Y.~LeCun.
\newblock Learning fast approximations of sparse coding.
\newblock In \emph{International Conference on Machine Learning}, Madison, WI, USA, 2010. Omnipress.

\bibitem[Gupta and Roughgarden(2017)]{pac_bayes_gen_l2o}
R.~Gupta and T.~Roughgarden.
\newblock A {PAC} approach to application-specific algorithm selection.
\newblock \emph{SIAM Journal on Computing}, 46\penalty0 (3):\penalty0 992--1017, 2017.

\bibitem[Heaton et~al.(2023)Heaton, Chen, Wang, and Yin]{safeguard_convex}
H.~Heaton, X.~Chen, Z.~Wang, and W.~Yin.
\newblock Safeguarded learned convex optimization.
\newblock In \emph{Proceedings of the AAAI Conference on Artificial Intelligence}, 2023.

\bibitem[Hong and Luo(2012)]{Hong2012OnTL}
M.~Hong and Z.-Q.~T. Luo.
\newblock On the linear convergence of the alternating direction method of multipliers.
\newblock \emph{Mathematical Programming}, 162:\penalty0 165 -- 199, 2012.

\bibitem[Hospedales et~al.(2021)Hospedales, Antoniou, Micaelli, and Storkey]{hospedales2020metalearning}
T.~Hospedales, A.~Antoniou, P.~Micaelli, and A.~Storkey.
\newblock Meta-learning in neural networks: A survey.
\newblock \emph{IEEE transactions on pattern analysis and machine intelligence}, 44\penalty0 (9):\penalty0 5149--5169, 2021.

\bibitem[Huber(1964)]{huber}
P.~J. Huber.
\newblock Robust estimation of a location parameter.
\newblock \emph{The Annals of Mathematical Statistics}, 35\penalty0 (1):\penalty0 73--101, 1964.

\bibitem[Ichnowski et~al.(2021)Ichnowski, Jain, Stellato, Banjac, Luo, Borrelli, Gonzales, Stoica, and Goldberg]{rlqp}
J.~Ichnowski, P.~Jain, B.~Stellato, G.~Banjac, M.~Luo, F.~Borrelli, J.~E. Gonzales, I.~Stoica, and K.~Goldberg.
\newblock Accelerating quadratic optimization with reinforcement learning.
\newblock In \emph{Advances in Neural Information Processing Systems 35}, 2021.

\bibitem[Jung et~al.(2022)Jung, Park, and Park]{qp_accelerate_rho}
H.~Jung, J.~Park, and J.~Park.
\newblock Learning context-aware adaptive solvers to accelerate quadratic programming.
\newblock \emph{arXiv preprint arXiv:2211.12443}, 2022.

\bibitem[Kalman(1960)]{kalman_filter}
R.~E. Kalman.
\newblock A new approach to linear filtering and prediction problems.
\newblock \emph{Transactions of the ASME--Journal of Basic Engineering}, 82\penalty0 (Series D):\penalty0 35--45, 1960.

\bibitem[Karg and Lucia(2020)]{deep_learning_mpc_karg}
B.~Karg and S.~Lucia.
\newblock Efficient representation and approximation of model predictive control laws via deep learning.
\newblock \emph{IEEE Transactions on Cybernetics}, PP, 2020.

\bibitem[King et~al.(2024)King, Kotary, Fioretto, and Drgona]{metric_learning}
E.~King, J.~Kotary, F.~Fioretto, and J.~Drgona.
\newblock Metric learning to accelerate convergence of operator splitting methods for differentiable parametric programming.
\newblock \emph{CoRR}, abs/2404.00882, 2024.

\bibitem[Kingma and Ba(2015)]{adam}
D.~P. Kingma and J.~Ba.
\newblock Adam: {A} method for stochastic optimization.
\newblock In \emph{International Conference on Learning Representations}, 2015.

\bibitem[Kotary et~al.(2021)Kotary, Fioretto, Van~Hentenryck, and Wilder]{e2e_survey}
J.~Kotary, F.~Fioretto, P.~Van~Hentenryck, and B.~Wilder.
\newblock End-to-end constrained optimization learning: A survey.
\newblock In \emph{International Joint Conference on Artificial Intelligence, { IJCAI-21}}, 2021.

\bibitem[Kullback and Leibler(1951)]{kl_div}
S.~Kullback and R.~A. Leibler.
\newblock {On Information and Sufficiency}.
\newblock \emph{The Annals of Mathematical Statistics}, 22\penalty0 (1):\penalty0 79 -- 86, 1951.

\bibitem[Langford and Caruana(2001)]{langford_union_prior}
J.~Langford and R.~Caruana.
\newblock (not) bounding the true error.
\newblock In \emph{Advances in Neural Information Processing Systems}, volume~14. MIT Press, 2001.

\bibitem[Langford and Seeger(2001)]{langford_seeger}
J.~Langford and M.~Seeger.
\newblock \emph{Bounds for averaging classifiers}.
\newblock School of Computer Science, Carnegie Mellon University, 2001.

\bibitem[Lessard et~al.(2016)Lessard, Recht, and Packard]{lessard2016IQC}
L.~Lessard, B.~Recht, and A.~Packard.
\newblock Analysis and design of optimization algorithms via integral quadratic constraints.
\newblock \emph{{SIAM} Journal on Optimization}, 26\penalty0 (1):\penalty0 57--95, jan 2016.

\bibitem[Li and Malik(2016)]{learning_to_optimize_malik}
K.~Li and J.~Malik.
\newblock Learning to optimize.
\newblock \emph{arXiv preprint arXiv:1606.01885}, 2016.

\bibitem[Lieder(2018)]{Lieder2018ProjectionBM}
F.~Lieder.
\newblock \emph{Projection Based Methods for Conic Linear Programming — Optimal First Order Complexities and Norm Constrained Quasi Newton Methods}.
\newblock Phd thesis, HHU Düsseldorf, 2018.

\bibitem[Liu et~al.(2019)Liu, Chen, Wang, and Yin]{alista}
J.~Liu, X.~Chen, Z.~Wang, and W.~Yin.
\newblock {ALISTA}: Analytic weights are as good as learned weights in { LISTA}.
\newblock In \emph{International Conference on Learning Representations}, 2019.

\bibitem[Majumdar et~al.(2021)Majumdar, Farid, and Sonar]{majumdar2021pac}
A.~Majumdar, A.~Farid, and A.~Sonar.
\newblock {PAC}-{Bayes} control: learning policies that provably generalize to novel environments.
\newblock \emph{The International Journal of Robotics Research}, 40\penalty0 (2-3):\penalty0 574--593, 2021.

\bibitem[Mak et~al.(2023)Mak, Chatzos, Tanneau, and Van~Hentenryck]{mak2023learning}
T.~W. Mak, M.~Chatzos, M.~Tanneau, and P.~Van~Hentenryck.
\newblock Learning regionally decentralized ac optimal power flows with { ADMM}.
\newblock \emph{IEEE Transactions on Smart Grid}, 2023.

\bibitem[Maurer(2004)]{maurer2004note}
A.~Maurer.
\newblock A note on the {PAC} {Bayesian} theorem.
\newblock \emph{arXiv preprint cs/0411099}, 2004.

\bibitem[McAllester(1998)]{McAllester_PAC_Bayes}
D.~A. McAllester.
\newblock Some {PAC}-{Bayesian} theorems.
\newblock In \emph{Conference on Computational Learning Theory}. Association for Computing Machinery, 1998.

\bibitem[Metz et~al.(2022)Metz, Harrison, Freeman, Merchant, Beyer, Bradbury, Agrawal, Poole, Mordatch, Roberts, et~al.]{VeLO}
L.~Metz, J.~Harrison, C.~D. Freeman, A.~Merchant, L.~Beyer, J.~Bradbury, N.~Agrawal, B.~Poole, I.~Mordatch, A.~Roberts, et~al.
\newblock Velo: Training versatile learned optimizers by scaling up.
\newblock \emph{arXiv preprint arXiv:2211.09760}, 2022.

\bibitem[Misra et~al.(2022)Misra, Roald, and Ng]{learn_active_sets}
S.~Misra, L.~Roald, and Y.~Ng.
\newblock Learning for constrained optimization: Identifying optimal active constraint sets.
\newblock \emph{INFORMS Journal on Computing}, 34\penalty0 (1):\penalty0 463--480, 2022.

\bibitem[Monga et~al.(2021)Monga, Li, and Eldar]{algo_unrolling}
V.~Monga, Y.~Li, and Y.~C. Eldar.
\newblock Algorithm unrolling: Interpretable, efficient deep learning for signal and image processing.
\newblock \emph{IEEE Signal Processing Magazine}, 38\penalty0 (2):\penalty0 18--44, 2021.

\bibitem[Nesterov(1983)]{nesterov}
Y.~Nesterov.
\newblock A method for unconstrained convex minimization problem with the rate of convergence $o(1/k^2)$.
\newblock 1983.

\bibitem[O'Donoghue(2021)]{scs_quadratic}
B.~O'Donoghue.
\newblock Operator splitting for a homogeneous embedding of the linear complementarity problem.
\newblock \emph{SIAM Journal on Optimization}, 31\penalty0 (3):\penalty0 1999--2023, 2021.

\bibitem[Paquette et~al.(2022)Paquette, van Merri\"{e}nboer, Paquette, and Pedregosa]{avg_case_halting}
C.~Paquette, B.~van Merri\"{e}nboer, E.~Paquette, and F.~Pedregosa.
\newblock Halting time is predictable for large models: A universality property and average-case analysis.
\newblock \emph{Foundations of Computational Mathematics}, 23\penalty0 (2):\penalty0 597--673, feb 2022.

\bibitem[Parikh and Boyd(2014)]{prox_algos}
N.~Parikh and S.~Boyd.
\newblock Proximal algorithms.
\newblock \emph{Foundations and Trends in Optimization}, 1\penalty0 (3):\penalty0 127–239, 2014.

\bibitem[Pedregosa and Scieur(2020)]{average_case}
F.~Pedregosa and D.~Scieur.
\newblock Acceleration through spectral density estimation.
\newblock In \emph{International Conference on Machine Learning}, 2020.

\bibitem[Pr{\'{e}}mont{-}Schwarz et~al.(2022)Pr{\'{e}}mont{-}Schwarz, Vitku, and Feyereisl]{safeguard_l2o}
I.~Pr{\'{e}}mont{-}Schwarz, J.~Vitku, and J.~Feyereisl.
\newblock A simple guard for learned optimizers.
\newblock In \emph{International Conference on Machine Learning}, 2022.

\bibitem[Ranjan and Stellato(2024)]{perfverifyqp}
V.~Ranjan and B.~Stellato.
\newblock Verification of first-order methods for parametric quadratic optimization.
\newblock \emph{arXiv preprint arXiv:2403.03331}, 2024.

\bibitem[Reeb et~al.(2018)Reeb, Doerr, Gerwinn, and Rakitsch]{reeb2018learning}
D.~Reeb, A.~Doerr, S.~Gerwinn, and B.~Rakitsch.
\newblock Learning gaussian processes by minimizing {PAC}-{Bayesian} generalization bounds.
\newblock \emph{Advances in Neural Information Processing Systems}, 31, 2018.

\bibitem[Ryu et~al.(2019)Ryu, Liu, Wang, Chen, Wang, and Yin]{plug_and_play_ryu}
E.~Ryu, J.~Liu, S.~Wang, X.~Chen, Z.~Wang, and W.~Yin.
\newblock Plug-and-play methods provably converge with properly trained denoisers.
\newblock In \emph{International Conference on Machine Learning}, 2019.

\bibitem[Ryu et~al.(2020)Ryu, Taylor, Bergeling, and Giselsson]{ryu_ospep}
E.~Ryu, A.~B. Taylor, C.~Bergeling, and P.~Giselsson.
\newblock Operator splitting performance estimation: Tight contraction factors and optimal parameter selection.
\newblock \emph{SIAM Journal on Optimization}, 30\penalty0 (3):\penalty0 2251--2271, 2020.

\bibitem[Ryu and Yin(2022)]{lscomo}
E.~K. Ryu and W.~Yin.
\newblock \emph{Large-Scale Convex Optimization: Algorithms amp; Analyses via Monotone Operators}.
\newblock Cambridge University Press, 2022.

\bibitem[Sambharya et~al.(2023)Sambharya, Hall, Amos, and Stellato]{l2ws_l4dc}
R.~Sambharya, G.~Hall, B.~Amos, and B.~Stellato.
\newblock End-to-{End} {Learning} to {Warm}-{Start} for {Real}-{Time} { Quadratic} {Optimization}.
\newblock In \emph{Proceedings of the 5th {Annual} {Learning} for {Dynamics} and { Control} {Conference}}, 2023.

\bibitem[Sambharya et~al.(2024)Sambharya, Hall, Amos, and Stellato]{l2ws}
R.~Sambharya, G.~Hall, B.~Amos, and B.~Stellato.
\newblock Learning to warm-start fixed-point optimization algorithms.
\newblock \emph{Journal of Machine Learning Research}, 25\penalty0 (166):\penalty0 1--46, 2024.

\bibitem[Sj{\"o}lund and B{\aa}nkestad(2022)]{sjolund2022graphbased}
J.~Sj{\"o}lund and M.~B{\aa}nkestad.
\newblock Graph-based neural acceleration for nonnegative matrix factorization.
\newblock \emph{arXiv preprint arXiv:2202.00264}, 2022.

\bibitem[Stellato et~al.(2020)Stellato, Banjac, Goulart, Bemporad, and Stephen]{osqp}
B.~Stellato, G.~Banjac, P.~Goulart, A.~Bemporad, and B.~Stephen.
\newblock {OSQP: An Operator Splitting Solver for Quadratic Programs}.
\newblock \emph{Mathematical Programming Computation}, 12\penalty0 (4):\penalty0 637--672, 2020.

\bibitem[Sucker and Ochs(2023)]{pmlr-v206-sucker23a}
M.~Sucker and P.~Ochs.
\newblock {PAC}-{B}ayesian learning of optimization algorithms.
\newblock In \emph{International Conference on Artificial Intelligence and Statistics}, 2023.

\bibitem[Sucker et~al.(2024)Sucker, Fadili, and Ochs]{Sucker2024LearningtoOptimizeWP}
M.~Sucker, J.~Fadili, and P.~Ochs.
\newblock Learning-to-optimize with {PAC}-{B}ayesian guarantees: Theoretical considerations and practical implementation.
\newblock \emph{arXiv preprint arXiv:2404.03290}, 2024.

\bibitem[Tan et~al.(2023)Tan, Mukherjee, Tang, and Sch\" {o}~nlieb]{learn_mirror}
H.~Y. Tan, S.~Mukherjee, J.~Tang, and C.-B. Sch\" {o}~nlieb.
\newblock Data-driven mirror descent with input-convex neural networks.
\newblock \emph{SIAM Journal on Mathematics of Data Science}, 5\penalty0 (2):\penalty0 558--587, 2023.

\bibitem[Taylor et~al.(2015)Taylor, Hendrickx, and Glineur]{pep2}
A.~Taylor, J.~Hendrickx, and F.~Glineur.
\newblock Smooth strongly convex interpolation and exact worst-case performance of first-order methods.
\newblock \emph{Mathematical Programming}, 161, 02 2015.

\bibitem[Taylor et~al.(2018)Taylor, Scoy, and Lessard]{Taylor2018LyapunovFF}
A.~B. Taylor, B.~V. Scoy, and L.~Lessard.
\newblock Lyapunov functions for first-order methods: Tight automated convergence guarantees.
\newblock In \emph{International Conference on Machine Learning}, 2018.

\bibitem[Tropp(2011)]{Tropp_2011}
J.~A. Tropp.
\newblock User-friendly tail bounds for sums of random matrices.
\newblock \emph{Foundations of Computational Mathematics}, 12\penalty0 (4):\penalty0 389--434, 2011.

\bibitem[Venkataraman and Amos(2021)]{neural_fp_accel_amos}
S.~Venkataraman and B.~Amos.
\newblock Neural fixed-point acceleration for convex optimization.
\newblock \emph{arXiv preprint arXiv:2107.10254}, 2021.

\bibitem[Vilalta and Drissi(2001)]{meta_learning_survey}
R.~Vilalta and Y.~Drissi.
\newblock A perspective view and survey of meta-learning.
\newblock \emph{Artificial Intelligence Review}, 18, 2001.

\bibitem[Wright(1997)]{PrimalDualIntWright1997}
S.~J. Wright.
\newblock \emph{Primal-dual interior-point methods}.
\newblock SIAM, 1997.

\bibitem[Wu et~al.(2020)Wu, Guo, Li, and Zhang]{glista}
K.~Wu, Y.~Guo, Z.~Li, and C.~Zhang.
\newblock Sparse coding with gated learned ista.
\newblock In \emph{International Conference on Learning Representations}, 2020.

\bibitem[Xie and Soh(1994)]{rkf}
L.~Xie and Y.~C. Soh.
\newblock Robust kalman filtering for uncertain systems.
\newblock \emph{Systems \& Control Letters}, 22\penalty0 (2):\penalty0 123--129, 1994.

\bibitem[Yang et~al.(2022)Yang, Chen, Chen, Wang, and Liang]{mL2O}
J.~Yang, X.~Chen, T.~Chen, Z.~Wang, and Y.~Liang.
\newblock M-l2o: Towards generalizable learning-to-optimize by test-time fast self-adaptation.
\newblock In \emph{International Conference on Learning Representations}, 2022.

\bibitem[Yang et~al.(2023)Yang, Chen, Zhu, He, Tao, Liang, and Wang]{l2gen_l2o}
J.~Yang, T.~Chen, M.~Zhu, F.~He, D.~Tao, Y.~Liang, and Z.~Wang.
\newblock Learning to generalize provably in learning to optimize.
\newblock In \emph{International Conference on Artificial Intelligence and Statistics}, 2023.

\bibitem[Yuan et~al.(2020)Yuan, Zeng, and Zhang]{discern_lin_rate}
X.~Yuan, S.~Zeng, and J.~Zhang.
\newblock Discerning the linear convergence of admm for structured convex optimization through the lens of variational analysis.
\newblock \emph{Journal of Machine Learning Research}, 21\penalty0 (83):\penalty0 1--75, 2020.

\bibitem[Zhang et~al.(2020)Zhang, O'Donoghue, and Boyd]{zhang2020globally}
J.~Zhang, B.~O'Donoghue, and S.~Boyd.
\newblock Globally convergent type-{I} anderson acceleration for nonsmooth fixed-point iterations.
\newblock \emph{SIAM Journal on Optimization}, 30\penalty0 (4):\penalty0 3170--3197, 2020.

\end{thebibliography}


\begin{thebibliography}{8}
\providecommand{\natexlab}[1]{#1}
\providecommand{\url}[1]{\texttt{#1}}
\expandafter\ifx\csname urlstyle\endcsname\relax
  \providecommand{\doi}[1]{doi: #1}\else
  \providecommand{\doi}{doi: \begingroup \urlstyle{rm}\Url}\fi

\bibitem[Alquier(2023)]{pac_bayes_intro}
P.~Alquier.
\newblock User-friendly introduction to {PAC}-{Bayes} bounds.
\newblock \emph{arXiv e-prints}, 2023.

\bibitem[Gregor and LeCun(2010)]{lista}
K.~Gregor and Y.~LeCun.
\newblock Learning fast approximations of sparse coding.
\newblock In \emph{International Conference on Machine Learning}, Madison, WI, USA, 2010. Omnipress.

\bibitem[Langford and Caruana(2001)]{langford_union_prior}
J.~Langford and R.~Caruana.
\newblock (not) bounding the true error.
\newblock In \emph{Advances in Neural Information Processing Systems}, volume~14. MIT Press, 2001.

\bibitem[Langford and Seeger(2001)]{langford_seeger}
J.~Langford and M.~Seeger.
\newblock \emph{Bounds for averaging classifiers}.
\newblock School of Computer Science, Carnegie Mellon University, 2001.

\bibitem[Liu et~al.(2019)Liu, Chen, Wang, and Yin]{alista}
J.~Liu, X.~Chen, Z.~Wang, and W.~Yin.
\newblock {ALISTA}: Analytic weights are as good as learned weights in { LISTA}.
\newblock In \emph{International Conference on Learning Representations}, 2019.

\bibitem[Majumdar et~al.(2021)Majumdar, Farid, and Sonar]{majumdar2021pac}
A.~Majumdar, A.~Farid, and A.~Sonar.
\newblock {PAC}-{Bayes} control: learning policies that provably generalize to novel environments.
\newblock \emph{The International Journal of Robotics Research}, 40\penalty0 (2-3):\penalty0 574--593, 2021.

\bibitem[Maurer(2004)]{maurer2004note}
A.~Maurer.
\newblock A note on the {PAC} {Bayesian} theorem.
\newblock \emph{arXiv preprint cs/0411099}, 2004.

\bibitem[Pedregosa and Scieur(2020)]{average_case}
F.~Pedregosa and D.~Scieur.
\newblock Acceleration through spectral density estimation.
\newblock In \emph{International Conference on Machine Learning}, 2020.

\end{thebibliography}
% \ifpreprint \else
% \bibliographystyle{icml2024}
% \fi

\newpage
\appendix
\onecolumn
% \section{KL inverse derivative}\label{}

\reviewChanges{
\section{PAC-Bayes background}\label{sec:prob_background}}
\reviewChanges{
In this section, we introduce the PAC-Bayes background needed to construct generalization guarantees given a set of \reviewChanges{$N$ i.i.d. samples $S$}.
We first introduce the Kullback-Leibler (KL) divergence, an important component in our bounds, and show how to compute its inverse in \Sec~\ref{subsec:kl}.
In \Sec~\ref{subsec:background_bounds}, we present two PAC-Bayes bounds: a sample convergence bound and Maurer's bound.
Specifically, for classical optimizers, we will use the sample convergence bound to bound the \emph{risk}
\begin{equation*}
  \risk = \mathbf{E}_{x \sim \mathcal{X}} e(x),
\end{equation*}
in terms of the \emph{empirical risk}
\begin{equation*}
  \emprisk = \frac{1}{N} \sum_{i=1}^N e(x_i).
\end{equation*}
Recall from \Eqn~\eqref{eq:zero_one_classical} that the error term $e(x)$ for a given parameter $x$ is always equal to $0$ or $1$.
For learned optimizers, we consider weights $\theta$ drawn from a distribution $P$ and use Maurer's bound to bound the \emph{expected risk}
\begin{equation}\label{eq:exp_risk}
  \exprisk(P) = \mathbf{E}_{\theta \sim P} \mathbf{E}_{x \sim \mathcal{X}}  e_\theta(x),
\end{equation}
in terms \reviewChanges{of} its \emph{expected empirical risk}
\begin{equation*}
  \expemprisk(P) =  \mathbf{E}_{\theta \sim P} \frac{1}{N} \sum_{i=1}^N e_\theta(x_i).
\end{equation*}
We use randomized weights to represent a distribution of learned optimizers, which is a key component of the  PAC-Bayes methods.}
\reviewChanges{Using randomized weights does not limit which types of learned optimizers we can apply our method to.}

\reviewChanges{
\subsection{KL divergence}\label{subsec:kl}
The KL divergence, a measure of distance between two probability distributions, features prominently in the PAC-Bayes guarantees that we use.
To derive our generalization bounds, it is sufficient to examine the KL divergence in two scenarios: between Normal distributions and between Bernoulli distributions.
\paragraph{Normal distributions.}
The KL-divergence between continuous distributions with density functions $q$ and $p$ over the Euclidean space $\reals^m$ is defined as
\begin{equation*}
  {\rm KL}(q \parallel p) = \int_{-\infty}^{\infty}q(y)\log \left(\frac{q(y)}{p(y)}\right) dy.
\end{equation*}
We are particularly interested in the case where both $p$ and $q$ are densities of multivariate normal distributions: $\mathcal{N}_p = \mathcal{N}(\mu_p, \Sigma_p)$ and $\mathcal{N}_q = \mathcal{N}(\mu_q, \Sigma_q)$ over $\reals^m$.
In this case, the KL divergence can be obtained in closed-form~\citep{Duchi2016DerivationsFL}:
\begin{equation}\label{eq:kl_normal}
   {\rm KL}(\mathcal{N}_q\parallel\mathcal{N}_p) = \frac{1}{2} \left(\Tr(\Sigma_p^{-1} \Sigma_q) + (\mu_q - \mu_p)^T \Sigma_p^{-1} (\mu_q - \mu_p) + \log \frac{{\det} \Sigma_p}{{\det} \Sigma_q} - m \right).
\end{equation}
}
\reviewChanges{
\paragraph{Bernoulli distributions.}
Our goal is to bound the risk $\risk$ for classical optimizers and the expected risk $\exprisk(P)$ for learned optimizers with posterior distribution $P$ in terms of their empirical counterparts.
Importantly, we remark that these quantities are the expected values of 0--1 error functions in equations~\eqref{eq:zero_one_classical} and~\eqref{eq:zero_one_loss_l2o}, which correspond to the key parameters of Bernoulli distributions.
We denote the KL divergence between two Bernoulli distributions, $\mathcal{B}(q)$ with mean $q$ and $\mathcal{B}(p)$ with mean $p$, as~\citep{kl_div}
\begin{equation*}
  {\rm kl}(q \parallel p) = {\rm KL}(\mathcal{B}(q) \parallel \mathcal{B}(p)) = q \log \frac{q}{p} + (1 - q) \log \frac{1 - q}{1 - p}.
\end{equation*}
In the next subsection, we will bound the gap between the key parameter of a Bernoulli distribution denoted as $p$ and its estimated value $q \in [0,1]$ as
\begin{equation*}
  {\rm kl}(q \parallel p) \leq c,
\end{equation*}
where $c > 0$.
This implies the inequality
\begin{equation*}
  p \leq {\rm kl}^{-1}(q ~|~ c) = \sup \{p \in [0,1] \mid {\rm kl}(q \parallel p) \leq c \},
\end{equation*}
where ${\rm kl}^{-1}(q ~|~ c)$  can be computed by solving the following one-dimensional convex geometric program~\citep{geometric_program},
\begin{equation}\label{prob:kl_inv}
  \begin{array}{ll}
  \mbox{maximize} & p \\
  \mbox{subject to} &\displaystyle q \log\left(\frac{q}{p}\right) + (1 - q) \log\left(\frac{1 - q}{1 - p}\right) \leq c \\
  &0 \leq p \leq 1.\\
  \end{array}
\end{equation}
Precise solutions to this problem can be obtained through convex optimization algorithms (\eg, through interior point methods~\citep{PrimalDualIntWright1997}).
\reviewChanges{Problem~\eqref{prob:kl_inv} is used to compute our performance guarantees for both classical and learned optimizers.}
We note that an upper bound to the KL inverse can be explicitly computed using Pinsker's inequality
\begin{equation}\label{eq:pinsker}
  {\rm kl}^{-1}(q ~|~ c) \leq q + \sqrt{c / 2}.
\end{equation}
\reviewChanges{However, the gap between Pinsker's bound and the KL inverse can be large; the KL inverse is always upper bounded by one, but Pinsker's bound can be infinitely large.}
}

\reviewChanges{
\subsection{Probabilistic Bounds}\label{subsec:background_bounds}
In this subsection, we present the probabilistic bounds that we use to obtain our generalization guarantees: a sample convergence bound and \reviewChanges{Maurer's} bound.}
\reviewChanges{
\paragraph{Sample convergence bound.}
The sample convergence bound below will be used to bound the risk $\risk$ in terms of the empirical risk $\emprisk$ for classical optimizers.}
\reviewChanges{
\begin{theorem}\label{thm:sample_conv_bound}
  \citep{langford_union_prior}.
  Given $\delta \in (0,1)$ and $N$ samples \reviewChanges{$S$}, with probability at least $1 - \delta$ the following bound holds:
\begin{equation}\label{eq:langford_bound}
  {\rm kl} (\emprisk \parallel  \risk ) \leq \frac{\log (2 / \delta)}{N}.
\end{equation}
\end{theorem}
}

\reviewChanges{
\begin{proof}
  We seek to prove the following inequality for $\epsilon > 0$
  \begin{equation*}
    \mathbf{P}(\kl(\emprisk~||~\risk) \geq \epsilon) \leq 2 e^{-N \epsilon}.
  \end{equation*}
  To get the final bound, we set $\delta = 2 \exp(-N \epsilon)$.
  The proof proceeds by breaking the case that the KL divergence exceeds $\epsilon$ into the case where $\emprisk > \risk$ and the other case where $\emprisk < \risk$.
  For a given value of the risk $\risk$, we define the quantity $r_\epsilon^1 > \risk$ implicitly so that $\kl(r_\epsilon^1~||~\risk) = \epsilon$.
  Similarly, we define $r_\epsilon^2 < \risk$ implicitly so that $\kl(r_\epsilon^1~||~\risk) = \epsilon$.
  We continue as follows for the case where $\emprisk > \risk$:
  \begin{align*}
    \mathbf{P}(\kl(\emprisk~||~\risk) \geq \epsilon, \emprisk > \risk) &= \mathbf{P}(\emprisk \geq r^1_\epsilon) \leq e^{-N \epsilon}.
  \end{align*}
  The equality comes from the definition of $r_\epsilon^1$.
  The inequality follows from the upper tail Chernoff bound for $\Delta > 0$:
  \begin{equation*}
    \mathbf{P}(\emprisk \geq \risk + \Delta) \leq \exp(-N~\kl(\risk+\Delta~||~\risk)).
  \end{equation*}
  For the case where $\emprisk < \risk$ we have
  \begin{align*}
    \mathbf{P}(\kl(\emprisk~||~\risk) \geq \epsilon, \emprisk < \risk) &= \mathbf{P}(\emprisk \geq r^2_\epsilon) \leq e^{-N \epsilon}.
  \end{align*}
  The equality comes from the definition of $r_\epsilon^2$.
  The inequality follows from the lower tail Chernoff bound for $\Delta > 0$:
  \begin{equation*}
    \mathbf{P}(\emprisk \geq \risk - \Delta) \leq \exp(-N~\kl(\risk-\Delta~||~\risk)).
  \end{equation*}
  The proof concludes by summing the probabilities over both cases
  \begin{equation*}
    \mathbf{P}(\kl(\emprisk~||~\risk) \geq \epsilon) = \mathbf{P}(\kl(\emprisk~||~\risk) \geq \epsilon, \emprisk > \risk) + \mathbf{P}(\kl(\emprisk~||~\risk) \geq \epsilon, \emprisk < \risk) \leq 2 e^{-N \epsilon}.
  \end{equation*}
\end{proof}
% \begin{proof}
%   We provide the proof written in \citet[Appendix C Lemma 8]{mardia2018concentration} for the reader's convenience.
%   We define the sets
%   \begin{equation*}
%     M_1 = \{\emprisk \mid  \kl(\emprisk~||~\risk) \geq \epsilon, \emprisk > \risk\} , \quad M_2 = \{\emprisk \mid  \kl(\emprisk~||~\risk) \geq \epsilon, \emprisk < \risk\}.
%   \end{equation*}
%   \begin{align*}
%     \mathbf{P}(\kl(\emprisk~||~\risk) \geq \epsilon) &= \mathbf{P}(M_1 \cup M_2) \\ 
%     & \leq \mathbf{P}(M_1) + \mathbf{P}(M_2) \\
%     & \leq e^{-N \epsilon} + e^{-N \epsilon} = 2 e^{-N \epsilon}.
%   \end{align*}
%   The first inequality uses a union bound.
%   To bound $\mathbf{P}(M_1)$ and $\mathbf{P}(M_2)$ in the next line, we use the fact that $M_1$ and $M_2$ are convex sets, $\inf_{r \in M_1} \kl(r~||~\risk) = \inf_{r \in M_2} \kl(r~||~\risk) = \epsilon$, and \citet[inequality $2.16$]{Dembo2010}.
%   The proof concludes by setting $\delta = 2 \exp(-N \epsilon)$.
% \end{proof}
}

\reviewChanges{
\paragraph{Maurer's bound.}
The derivation of our generalization bounds for learned optimizers is based on Maurer's bound (itself an adaptation of Seeger's bound~\citep{langford_seeger}), which allows us to provide bounds when the weights $\theta$ are drawn from a distribution $P \in \mathcal{P}$.
Here, $\mathcal{P}$ is the space of all probability distributions in $\reals^p$.
Specifically, Maurer's bound provides a bound on the expected risk $\exprisk(P)$ in terms of its expected empirical risk $\expemprisk(P)$.
\begin{theorem}\label{thm:maurer}
  \citep{maurer2004note}. 
  Given a set of \reviewChanges{$N$ samples $S$} where $N \geq 8$, a prior distribution independent of the training data $P_0 \in \mathcal{P}$, and $\delta \in (0, 1)$, with probability at least $1 - \delta$ the following bound holds for all distributions $P \in \mathcal{P}$:
\reviewChanges{
\begin{equation}\label{eq:pac_bayes_kl_inverse}
  \exprisk(P) \leq {\rm kl}^{-1}\left(\expemprisk(P) ~\bigg|~ \frac{1}{N}\left({\rm KL}(P \parallel P_0) + \log \frac{2\sqrt{N}}{\delta}\right)\right).
\end{equation}}
\end{theorem}
The PAC-Bayes framework typically adopts the following steps.
First, we select the prior $P_0 \in \mathcal{P}$ before observing any training data.
Then, we observe the training data \reviewChanges{$S$} and we choose the posterior distribution $P$ (\eg, through a learning algorithm~\citep{nonvacuous_pac_bayes}).
Lastly, we use the inequality~\eqref{eq:pac_bayes_kl_inverse} to bound the expected risk of the posterior distribution $\exprisk(P)$.
This posterior is allowed to depend on the prior and the samples.
}

\reviewChanges{
\begin{proof}
  The relative entropy $\KL(P~||~P_0)$ of two probability measures $P$ and $P_0$ on a set $\mathcal{H}$ is defined to be infinite if $P$ is not absolutely continuous with respect to $P$.
  Otherwise, $\KL(P~||~P_0) = \mathbf{E}_P [\log \frac{dP}{dP_0}]$ where $dP / dP_0$ is the density of $P$ with respect to $P_0$.
  In the proof, we let $S$ be a set of $N$ training samples of the parameters drawn i.i.d. from the distribution $\mathcal{X}$.
  The proof continues as
  \begin{align*}
    &\mathbf{E}_S [\exp(N~ \kl(\exprisk(P)~||~ \expemprisk(P)) - \KL(P~||~ P_0))]\\
    & \leq \mathbf{E}_S \left[\exp \left(\mathbf{E}_{\theta \sim P} \left[N~ \kl \left(\frac{1}{N}\sum_{i=1}^N e_\theta(x_i)~\bigg|\bigg|~ \mathbf{E}_{x \sim \mathcal{X}} e_\theta(x)\right) - \log \frac{dP}{dP_0} (\theta) \right]\right)\right] \\
    & \leq \mathbf{E}_S \left[ \mathbf{E}_{\theta \sim P} \left[ \exp \left(N~ \kl \left(\frac{1}{N}\sum_{i=1}^N e_\theta(x_i)~\bigg|\bigg|~ \mathbf{E}_{x \sim \mathcal{X}} e_\theta(x)\right) - \log \frac{dP}{dP_0} (\theta) \right)\right]\right] \\
    & = \mathbf{E}_{S} \mathbf{E}_{\theta \sim P_0} \exp \left(N~ \kl \left(\frac{1}{N}\sum_{i=1}^N e_\theta(x_i)~\bigg|\bigg|~ \mathbf{E}_{x \sim \mathcal{X}} e_\theta(x)\right)  \right) \left( \frac{dP}{dP_0}\right)^{-1} \left( \frac{dP}{dP_0}\right)\\
    & = \mathbf{E}_{\theta \sim P_0} \mathbf{E}_S \exp \left(N~ \kl\left(\frac{1}{N}\sum_{i=1}^N e_\theta(x_i)~\bigg|\bigg|~ \mathbf{E}_{x \sim \mathcal{X}} e_\theta(x)\right)\right) \\
    & \leq 2 \sqrt{N}.
  \end{align*}
  The first inequality follow's from Jensen's inequality and the convexity of the KL divergence.
  The second inequality follow's from Jensen's inequality and the convexity of the exponential function.
  The third line applies the Radon-Nikodyn derivative to change the expectation of $\theta$ over the posterior $P$ to be the expectation of $\theta$ over the prior $P_0$.
  The second to last line uses Tonelli's theorem to switch the order of the expectations of a non-negative random variable.
  The last inequality uses inequality $1$ in~\citet{maurer2004note}.
  Then, by Markov's inequality the proof finishes with
  \begin{align*}
    \delta & \geq \mathbf{P}_S \left(\exp(N~ \kl(\exprisk(P)~||~  \expemprisk(P)) - \KL(P~||~ P_0)) > \frac{2 \sqrt{N}}{\delta} \right) \\
    &= \mathbf{P}_S  \left(\kl(\exprisk(P)~||~  \expemprisk(P)) > \frac{KL(P~||~  P_0) + \log \left(\frac{2 \sqrt{N}}{\delta}\right)}{N} \right).
  \end{align*}
\end{proof}
}

\section{Experimental details}
\subsection{Cross-validating $B^{\rm target}$}\label{subsec:crossval}
% We cross-validate with $B^{\rm target}$ across $(0.01, 0.05, 0.1, 0.2, 0.3)$.
In our experiments, we cross-validate over six $B^{\rm target}$ hyperparameter values.
% We cross-validate over $5$ values.
If a particular bound on the expected risk holds with probability $1 - \delta$ for a given $B^{\rm target}$, then all six bounds hold with probability $1 - 6 \delta$ by a union bound.
Table~\ref{tab:cross_val} enumerates the $B^{\rm target}$ values chosen for cross-validation, alongside the corresponding upper bound on the generalization gap as determined by Pinsker's inequality from \Eqn~\eqref{eq:pinsker}.
After training, if $B(w^\star,s^\star,\lambda^\star)=B^{\rm target}$ (which we observe, approximately holds true due to the penalty form from problem~\eqref{prob:training_prob}), then we can bound the generalization gap: $\exprisk(P) - \expemprisk(P) \leq \sqrt{B^{\rm target} / 2}$ where the posterior is $P = \mathcal{N}_{w^\star,s^\star}$.

\begin{table}[!h]
  \centering
  \footnotesize
    \renewcommand*{\arraystretch}{1.0}
  \caption{
    The different $B^{\rm target}$ values used during cross-validation and their associated upper bounds on the generalization gap.
  }
  \label{tab:cross_val}
  \vspace*{-3mm}
  \begin{tabular}{l}
  \end{tabular}
  \adjustbox{max width=\textwidth}{
    \begin{tabular}{ll}
    $B^{\rm target}$&
    $\sqrt{B^{\rm target}/2}$
    \\
    \midrule
    % Add directly from csv reader
    % the csv file has names colnameA, colnameB
    \begin{filecontents*}{./data/cross_validation.csv}
Btarget,Pinsker
0.01,0.071
0.03,0.122
0.05,0.158
0.1,0.223
0.2,0.316
0.3,0.387
\end{filecontents*}
\csvreader[head to column names, late after line=\\]{./data/cross_validation.csv}{
    Btarget=\colA,
    Pinsker=\colB,
    }{\colA & \colB}
    \bottomrule
  \end{tabular}}
\end{table}

\subsection{Quantile bounds}\label{sec:quantiles}
The results from Sections~\ref{sec:classical} and \ref{sec:gen_l2o} provide probabilistic bounds on the risk and the expected risk respectively for a number of algorithm steps $k$ and tolerance $\epsilon$.
Recall that the (expected) risk is equivalent to the probability of failing to reach a given tolerance (due to the use of the error function).
Therefore an upper bound on the (expected) risk corresponds to an upper bound on the quantile.
For instance, if after $k$ steps, the risk is bounded with probability (w.p.) $1 - \delta$ by $0.1$ with some underlying metric $\phi$ and tolerance $\epsilon$, then, w.p. $1 - \delta$, the tolerance $\epsilon$ upper bounds the metric $\phi$ after $k$ steps at least $90 \%$ of the time.
Using our notation, this is equivalent to the following statement; if $r_{\mathcal{X}} = \mathbf{E}_{x \sim \mathcal{X}} [\mathbf{1}(\phi(z^k(x), x) \geq \epsilon)] \leq 0.1$ w.p. $1-\delta$, then $\mathbf{P}_{x \sim \mathcal{X}}(\phi(z^k(x), x) \geq \epsilon) \leq 0.1$ w.p. $1-\delta$.
Therefore the tolerance $\epsilon$ is a valid, probabilistic $90$-th quantile bound.
% For instance, if the (expected) risk is upper bounded by $0.1$ after $k$ steps for a given metric, then the quantile is upper bounded

To obtain the \emph{tightest} quantile bounds in \Sec~\ref{sec:numerical_experiments} for a given number of steps $k$, we proceed as follows.
We first obtain bounds on the risk for $N^{\rm tol}$ pre-determined tolerances.
If each bound on the risk with a specific tolerance holds with probability $1 - \delta$, then all of the bounds across all of the tolerances hold simultaneously with probability $1 - \delta N^{\rm tol}$ by virtue of a union bound.
Then for a given $k$ and quantile $Q$, we find the lowest tolerance such that the bound on the (expected) risk is at most $1 - Q$.
For example, say we want to bound the  $90$th quantile bound of the fixed-point residual at $k$ steps.
We first take all of the bounds on the risk for $\epsilon_1, \dots, \epsilon_{N^{\rm tol}}$.
Then we find the lowest value $\epsilon_i$ such that the (expected) risk with tolerance $\epsilon_i$ is at most $0.1$; this value of $\epsilon_i$ bounds the $90$th quantile with probability at least $1 - \delta N^{\rm tol}$.
% In all of the experiments, we discretize the metrics into $N^{\rm tol} = 81$ pre-determined tolerances.
Note that the bounds do not hold simultaneously across different values of $k$.
However, if desired, they can be obtained by applying another union bound over the algorithm steps.
% It is possible, but requires another union bound.
% Then the quantile number is equal to the (expected) risk. 

\subsection{Other numerical details}\label{sec:prob_details}
To obtain the quantile bounds, we discretize the metric into $81$ pre-determined tolerances.
For the metrics in the MAML problem, the discretization is $81$ points evenly spaced out on a log scale between $10^{-3}$ and $10^1$.
For the NMSE metric, we discretize between $-80$ and $0$ evenly on a linear scale. 
For all other metrics, the discretization is $81$ points evenly spaced out on a log scale between $10^{-6}$ and $10^2$.
For classical optimizers, we set the desired probability value to be $\delta = 10^{-4}$.
The bounds on the risk for the classical optimizers holds with probability $1 - \delta = 0.9999$ and each of the quantile bounds holds with probability $0.9919$ due to the union bound over the $81$ tolerances.
For learned optimizers, the desired probability values are $\delta = 10^{-5}$ and $\omega = 10^{-5}$.
For the learned optimizers, there are two additional considerations: the additional sample convergence bounds which holds with probability $1 - \omega$ and the cross-validation over the set of $B^{\rm target}$ values which requires a union bound.
After taking a union bound over the cross-validated $B^{\rm target}$ values, the bound on the expected risk holds with probability at least $1 - 6 (\delta + \omega) = 0.99988$.
The bounds on the quantiles each hold with probability at least $0.99028$.
% We calibrate the bounds using $20000$ samples for LISTA and its variants and MAML.
% For L2WS, we use $1000$ samples.
For all learned optimizers, we set the prior hyperparameters to be $\lambda^{\rm max} = 100$ and $b=100$.
% We use $50000$ training samples and evaluate on $1000$ test problems.

% \subsection{Probability numerical details}\label{sec:p_details}
% we multiply each of these probabilities by $5$ since we also need to take the union bound over the cross-validated $B^{\rm target}$ values.
% Thus the bounds on the expected risk hold with probability

% \subsection{Differentiation through the KL inverse}

\section{Proofs}
\subsection{Proof of \Thm~\texorpdfstring{\ref{thm:gen_thm}}{}}\label{proof:gen_thmproof}
% First, let $\delta \in \reals^J$
The vector $a \in \mathbf{N}^J_+$ and constant $\delta \in (0,1)$ defines the quantity $\delta_a$ from \Eqn~\eqref{eq:delta_a}, $\delta_a = \delta \left(6 / (\pi^2)\right)^J (\prod_{j=1}^J a_j^2)^{-1}$.
% \begin{equation*}
%   \delta_a = \left(\frac{6}{\pi^2}\right)^J \frac{\delta}{\prod_{j=1}^J a_j^2}.
% \end{equation*}
Note that $\delta_a$ is in the range $(0,1)$ since all of $a_j$ terms and $J$ are at least one and $\delta \in (0,1)$.
Next, we apply \reviewChanges{Maurer's} bound from \Thm~\ref{thm:maurer} which states that with probability at least $1 - \delta_a$, the following inequalities hold:

\noindent
\resizebox{1.00\linewidth}{!}{
\begin{minipage}{\linewidth}
\begin{align}
  \reviewChanges{{\rm kl}}(\hat{R}_S(\mathcal{N}_{w,s})~|\reviewChanges{|}~ R_\mathcal{X}(\mathcal{N}_{w,s})) &\leq \frac{1}{N} \left(\KL(\mathcal{N}_{w,s}\parallel\mathcal{N}(w_0,\Lambda)) + \log \frac{\reviewChanges{2 \sqrt{N}}}{\delta_a}\right) \notag\\
  &= \frac{1}{N} \left(\KL(\mathcal{N}_{w,s}\parallel\mathcal{N}(w_0,\Lambda)) + \log \left(\prod_{j=1}^J a_j^2 \right) +  J \log \frac{\pi^2}{6} + \log \frac{\reviewChanges{2 \sqrt{N}}}{\delta}\right) \notag\\
  &= \frac{1}{N} \left(\KL(\mathcal{N}_{w,s}\parallel\mathcal{N}(w_0,\Lambda)) + 2 \sum_{j=1}^J  \log \left(b \log \frac{\lambda^{\rm max}}{\lambda_j}\right)  +  J \log \frac{\pi^2}{6} + \log \frac{\reviewChanges{2 \sqrt{N}} }{\delta}\right).\label{eq:ineq_delta_a}
\end{align}
\end{minipage}
}
In the final line, we use the equality from the theorem $\lambda_j = \lambda^{\rm max} \exp (- a_j / b)$
where $b$ and $\lambda^{\rm max}$ are pre-defined.
% \red{B: where does this come from. What is the value of $b$? and of $a_j$?}
Now, to get the main result, we take a union bound over all possible vectors $a \in \mathbf{N}^J_+$.
By taking a union bound with probability at least 
\begin{equation}\label{eq:probability}
  1 - \sum_{a_1=1}^\infty \dots  \sum_{a_J=1}^\infty \left(\frac{6}{\pi^2}\right)^J \frac{\delta}{ a_1^2 a_2^2 \cdots a_J^2},
\end{equation}
inequality~\eqref{eq:ineq_delta_a} holds uniformly for all $a \in \mathbf{N}^J_+$.
Since 
\begin{equation*}
  \sum_{i=1}^\infty \frac{1}{i^2} = \frac{\pi^2}{6},
\end{equation*}
this probability given by line~\eqref{eq:probability} simplifies to $1 - \delta$.

\subsection{Proof of \Thm~\texorpdfstring{\ref{thm:l2ws}}{}}\label{proof:l2wsproof}
% For a new $\theta$ not in the training set, we have that with probability $1 / (N + 1)$, 
% $\|\theta\|_2 > \bar{\theta}$.
% Also with probability $1 / (N + 1)$, $\|z^\star(\param)\|_2 > \bar{z}$.
\begin{lemma}\label{lem:spectral_norm}
  % Let $A \in \reals^{m \times n}$ and $C \in \reals^{m \times n}$ where the following inequalities hold element-wise: $0 \leq A \leq C$.
  % Then $\|A\|_2 \leq \|C\|_2$.
  If $0 \leq A \leq C$ element-wise for $A \in \reals^{m \times n}$ and $C \in \reals^{m \times n}$, then $\|A\|_2 \leq \|C\|_2$.
\end{lemma}
\begin{proof}
  The proof of the lemma proceeds as follows:
  \begin{align*}
    \|A\|_2 &= \max_{\|v\|_2 = 1,v\geq0} \|Av\|_2 \\
    &= \|Av^\star\|_2\\
    &\leq \|Cv^\star\|_2\\
    &\leq \|C\|_2.
  \end{align*}
  The first line follows from the definition of the spectral norm, and noting that since $A \geq0$, a maximizer occurs where $v\geq0$.
  To see this, observe that if a vector $\bar{v}$ is a maximizer, then so is $|\bar{v}|$.
  In the second line, we let $v^\star$ be the maximizer.
  The third line comes from $A \leq C$.
  The last line follows from the definition of the spectral norm.
\end{proof}

% Let $W_i$ and $b_i$ be the mean of the weight matrix and bias vector respectively of the $i$-th layer.
% Let $\Sigma_i$ and $\sigma_i$ be the variance of the weight matrix and bias vector respectively of the $i$-th layer.
\paragraph{Bounding the spectral norm of the weight matrix.}
We first let $U_i \sim \mathcal{N}(0,\Sigma_i)$ and $u_i \sim \mathcal{N}(0,\sigma_i)$ and define the following matrices:
\begin{equation*}
  \tilde{U}_i = \begin{bmatrix}
    U_i \\
    u_i^T
  \end{bmatrix}, \quad
  \tilde{\Sigma}_i = \begin{bmatrix}
    \Sigma_i \\
    \sigma_i^T
  \end{bmatrix}
  .
\end{equation*}
% The bulk of the proof resides in bounding $\|\tilde{U}_i\|_2$ with high probability.
% We now state \Thm~1.5 from~\citet{Tropp_2011}.
We now state a result from~\citet[\Sec~4.3]{Tropp_2011} that will allows us to bound $\|\tilde{U}_i\|_2$ with high probability.
\begin{theorem}\label{thm:tropp}
  Consider a fixed matrix $B \in \reals^{d_1 \times d_2}$ and a random matrix $\Gamma \in \reals^{d_1 \times d_2}$ whose entries are independent standard normal variables.
  Define the variance parameter
  \begin{equation}\label{eq:var_tropp}
      v^2 = \max \{\max_j\|B_{j:}\|^2_2, \max_k\|B_{:k}\|^2_2\},
    \end{equation}
  where $B_{j:}$ and $B_{:k}$ are the $j$-th row and $k$-th column of the matrix $B$.
  % \begin{equation}\label{eq:var_tropp}
  %   v^2 = \max \{\|\sum_k B_k^T B_k\|_2, \|\sum_k B_k B_k^T\|_2\}.
  % \end{equation}
  Then for all $t \geq 0$, 
  \begin{equation*}
    \mathbf{P}\left(\|\Gamma \odot B\|_2 \geq t \right) \leq (d_1 + d_2)e^{-t^2 / 2 v^2}.
  \end{equation*}
\end{theorem}
% We use \Thm~\ref{thm:tropp} to probabilistically bound $\|\tilde{U}_i\|_2$.
% We consider the set the matrices $\{B_{j,k}\}_{j=1, \, k=1}^{m_i+1, \, n_i}$ where $B_{j,k}$ 
% for any $j \leq m_i$ is the all zeroes matrix except for the standard deviation of the weight matrix $W_i$ at entry $(j,k)$.
% And for $j=m_i+1$, $B_{j,k}$ is the all zeroes matrix except for the standard deviation of $\sigma_i$ at entry $k$.
% We define the set of matrices $\{B^i_{j,k}\}_{j=1, \, k=1}^{m_i+1, \, n_i}$, where $B^i_{j,k}$ is the zero matrix except for the $(j,k)$ entry which is the square root of $\Sigma_i$ at $(j,k)$ for $1 \leq j \leq m_i$ and is the square root of the $k$-th entry of $\sigma_i$ for $j = m_i + 1$.
% Note that $\tilde{U}_i$ takes the form $\sum_{j,k} \xi_{j,k} B^i_{j,k}$ where each $\xi_{j,k}$ is an i.i.d. standard Gaussian.
% is characterized as follows: for $1 \leq j \leq m_i$, the matrix $B^i_{j,k}$ is the zero matrix except for its $(j,k)$-th entry, which is set to the standard deviation of the corresponding entry in the weight matrix $W_i$. For $j = m_i + 1$, each $B_{j,k}$ remains a zero matrix except at the $k$-th entry of the last row, which is assigned the value of $\sigma_i$, representing its standard deviation.
We let $v_i^2$ be the variance parameter of the $i$-th layer from \Eqn~\eqref{eq:var_tropp}:
\begin{equation*}
  v_i^2 = \max \{\max_j\|(\tilde{\Sigma}_i)^{1/2}_{j:}\|^2_2, \max_k\|(\tilde{\Sigma}_i)^{1/2}_{:k}\|^2_2\}.
\end{equation*}
% \begin{equation*}
%   v_i^2 = \max \{\|\sum_{j,k} B_{j,k}^T B_{j,k}\|_2, \|\sum_{j,k} B_{j,k} B_{j,k}^T\|_2\}.
% \end{equation*}
Using \Thm~\ref{thm:tropp}, we bound the spectral norm of $\tilde{U}_i$ as 
\begin{equation}\label{ineq:spectral_bnd}
  \mathbf{P}(\|\tilde{U}_i\|_2 \geq \tau_i) \leq (m_i + n_i + 1)e^{-\tau_i^2 / 2 v_i^2}.
\end{equation}
% Using Lemma~\ref{lem:spectral_norm} we can bound $v_i$ with $v_i^{\rm max} = \max\{m_i+1, n_i\} \sqrt{\max \{\|\Sigma_i\|_\infty, \|\sigma_i\|_\infty\}}$ since $B_{j,k}^T B_{j,k} \leq \sigma^{\rm max} \ones \ones^T$ element-wise.
% To see this, we construct another set of matrices $\tilde{B}^i_{j,k}$ which takes the same form as $B^i_{j,k}$ except takes a value of $\sigma_i^{\rm max}$ at $(j,k)$.
% Then the matrix $\sum_{j,k} B^i_{j,k}^T B^i_{j,k}$
We set the right hand side to be $\delta / L$ to get
\begin{equation*}
  \tau_i = v_i \sqrt{2 \log(L(m_i + n_i + 1)/\delta)},
\end{equation*}
thereby bounding the spectral norm of $\tilde{U}_i$ by $\tau_i$ with probability at least $1 - \delta / L$.
We take a union bound across all layers so that with probability $1 - \delta$, the inequalities $\|\tilde{U}_i\|_2 \leq \tau_i$ for $i=0, \dots, L-1$ hold simultaneously.

% We do the same thing for the bias term.
% With probability at least $1 - \delta / L$, $\|b_i\|_2$ is bounded by $\mu_i$ where $\mu_i = \sigma \sqrt{2 \log(L(m_i + 1)/\delta)}$.
\paragraph{Bounding the output of each layer.}
We turn our attention to bounding the output of the $i$-th layer, which we denote as $y_i(x)$ (where $y_0 = x$).
Due to the bias terms, it is helpful to include the notation $\bar{y}_i(x) = (y_i(x), 1)$.
% Let $y_i$ be the output of the $i$-th layer.
We then have the following bound for $i =0,\dots, L-2$:
\begin{align}
  \|y_{i+1}(x)\|_2 &= \|\psi((\tilde{W}_{i} + \tilde{U}_{i}) \bar{y}_i(x))\|_2 \notag\\
  &\leq \|\tilde{W}_{i} + \tilde{U}_{i}\|_2 \|\bar{y}_i(x)\|_2 \notag\\
  &\leq (\|W_{i}\|_2 + \|b_{i}\|_2 + \|\tilde{U}_{i}\|_2) (\|y_i(x)\|_2 + 1) \label{ineq:layer_bound}.
\end{align}
The first inequality follows from the ReLU activation function and Cauchy-Schwarz inequality.
The second inequality follows from the triangle inequality.
Note that the final layer does not have a ReLU activation, but the inequality holds nonetheless to bound $\|y_L(x)\|_2$.
Now, we let $a_0^\star = \bar{x}$ and 
\begin{equation*}
  a_{i+1}^\star = (\|W_{i}\|_2 + \|b_{i}\|_2  + \tau_{i})(a_i^\star + 1).
\end{equation*}
It follows from inequalities~\eqref{ineq:spectral_bnd} and~\eqref{ineq:layer_bound} that with probability at least $1 - \delta$, the quantity $a_i^\star$ upper bounds the output of the $i$-th layer.
% Now we let $a_i^\star$ denote an upper bound on the output of the $i$-th layer, considering the zeroth layer to be the inputs, \ie, $a_0^\star = \bar{x}$.
% We following bound holds with probability $1 - \delta / L$:
% \begin{equation*}
%   a_{j+1}^\star = (\|W_j\|_2 + \|b_i\|_2  + \tau_i)(a_j^\star + 1).
% \end{equation*}
Since $h_\theta(x)$ is the output of the $L$-th layer, the bound $\|h_\theta(x)\|_2 \leq a^\star_L$ holds with probability $1 - \delta$.

\section{Operator theory definitions}\label{sec:op_theory}
First, recall that the set of fixed-points of operator $T$ is denoted as $\fix T$.
\begin{definition}[$\beta$-contractive operator]\label{def:contractive}
  An operator $T$ is $\beta$-contractive for $\beta \in (0, 1)$ if
  \begin{equation*}
    \|Tx - Ty\|_2 \leq \beta \|x - y\|_2 \quad \forall x, y \in \dom T.
  \end{equation*}
\end{definition}

\begin{definition}[$\beta$-linearly convergent operator]\label{def:lin_conv}
  An operator $T$ is $\beta$-linearly convergent for $\beta \in [0, 1)$ if
  \begin{equation*}
    \dist_{\fix T}(Tx) \leq \beta \dist_{\fix T}(x) \quad \forall x \in \dom T.
  \end{equation*}
\end{definition}

\begin{definition}[Non-expansive operator]
  An operator $T$ is non-expansive if
  \begin{equation*}
    \|Tx - Ty\|_2 \leq \|x - y\|_2, \quad \forall x, y \in \dom T.
  \end{equation*}
\end{definition}

\begin{definition}[$\alpha$-averaged operator]\label{def:averaged}
  An operator $T$ is $\alpha$-averaged for $\alpha \in (0, 1)$ if there exists a non-expansive operator $R$ such that $T = (1 - \alpha) I + \alpha R$.
\end{definition}

\reviewChanges{
\paragraph{Rates of convergence.}
The convergence rate of the fixed-point iterations can be summarized as follows for any $z^\star(x) \in \fix T_x$.
If operator $T$, with parameter $x$, is $\beta$-linearly convergent, with $\beta \in (0, 1)$, then~\citep{l2ws}
\begin{equation}\label{eq:lin_conv_rate}
  \|z^{k+1}(x) - z^k(x)\|_2 \leq 
    2 \beta^k \|z^\star(x) - z^0(x)\|_2.
\end{equation}
This rate also applies to $\beta$-contractive operators as they are a subset of $\beta$-linearly-convergent operators.
If operator $T$ with parameter $x$ is $\alpha$-averaged then the average\reviewChanges{d} iteration, also called \reviewChanges{the} {Krasnosel'ski\u{\i}-Mann} iteration, satisfies the following bound~\citep{Lieder2018ProjectionBM}:
\begin{equation}\label{eq:avg_rate}
  \frac{\|z^{k+1}(x) - z^k(x)\|_2}{\|z^\star(x) - z^0(x)\|_2} \leq \begin{dcases}
    \sqrt{\frac{1}{k+1}\left(\frac{k}{k+1}\right)^k \frac{1}{\alpha(1 - \alpha)}} & \text{if} \quad \frac{1}{2} \leq \alpha \leq \frac{1}{2}\left(1 + \sqrt{\frac{k}{k+1}}\right) \\
     \frac{1}{2}(2 \alpha - 1)^k & \text{if} \quad \frac{1}{2}\left(1 + \sqrt{\frac{k}{k+1}}\right) \leq \alpha \leq 1.
  \end{dcases}
\end{equation}
}

\end{document}